\documentclass[mathematics,article,accept,pdftex,moreauthors]{Definitions/mdpi} 
\firstpage{1} 
\pubvolume{13}
\issuenum{21}
\articlenumber{3462}
\pubyear{2025}
\copyrightyear{2025}
\externaleditor{Adara M. Blaga} 
\datereceived{18 July 2025} 
\daterevised{25 September 2025} 
\dateaccepted{25 September 2025} 
\datepublished{30 October 2025} 
\hreflink{https://doi.org/10.3390/math13213462} 

\Title{First-Order Axiom Systems \hspace{1pt}$\mathscr{E}_{d}$\hspace{1pt}  and \hspace{1pt}$\mathscr{E}_{da}$\hspace{1pt} Extending Tarski's \hspace{1pt}$\mathscr{E}_{2}$\linebreak{}
with Distance and Angle Function Symbols 
for Quantitative Euclidean Geometry}

\TitleCitation{First-Order Axiom Systems $\mathscr{E}_{d}$ and $\mathscr{E}_{da}$
Extending Tarski's $\mathscr{E}_{2}$ 
with Distance and Angle Function Symbols 
for Quantitative Euclidean Geometry}

\Author{Hongyu 
 Guo \orcidA{}}

\AuthorNames{Hongyu Guo}
\AuthorCitation{Guo,  H.}

\address[1]{Department of Computer and Information Sciences, Texas A\&M University--Victoria, Victoria, 
 TX 77901, USA;
 guoh@tamuv.edu
}

\abstract{
Tarski's first-order axiom system $\mathscr{E}_{2}$ for Euclidean
geometry is notable for its completeness and decidability. However,
the Pythagorean theorem---either in its modern algebraic form $a^{2}+b^{2}=c^{2}$
or in Euclid's \emph{Elements}---cannot be directly expressed in
$\mathscr{E}_{2}$, since neither distance nor area is a primitive
notion in the language of $\mathscr{E}_{2}$. In this paper, we introduce
an alternative axiom system $\mathscr{E}_{d}$ in a two-sorted language,
which takes a two-place distance function $d$ as the only geometric
primitive. We also present a conservative extension $\mathscr{E}_{da}$
of it, which also incorporates a three-place angle function $a$,
both formulated strictly within first-order logic. The system $\mathscr{E}_{d}$
has two distinctive features: it is simple (with a single geometric
primitive) and it is quantitative. Numerical distance can be directly
expressed in this language. The \emph{Axiom of Similarity} plays a
central role in $\mathscr{E}_{d}$, effectively killing two birds
with one stone: it provides a rigorous foundation for the theory of
proportion and similarity, and it implies \emph{Euclid\textquoteright s
Parallel Postulate} (EPP). The Axiom of Similarity can be viewed as
a quantitative formulation of EPP. The Pythagorean theorem and other
quantitative results from similarity theory can be directly expressed
in the languages of $\mathscr{E}_{d}$ and $\mathscr{E}_{da}$, motivating
the name \emph{Quantitative Euclidean Geometry}. The traditional analytic
geometry can be united under synthetic geometry in $\mathscr{E}_{d}$.
Namely, analytic geometry is not treated as a model of $\mathscr{E}_{d}$,
but rather, its statements can be expressed as first-order formal
sentences in the language of $\mathscr{E}_{d}$. The system $\mathscr{E}_{d}$
is shown to be consistent, complete, and decidable. Finally, we extend
the theories to hyperbolic geometry and  Euclidean geometry in higher
dimensions.
}


\MSC{03B30; 03C35; 51M05; 51M10\vspace{-5pt}}

\usepackage[T1]{fontenc}
\usepackage{array}
\usepackage{textcomp}
\usepackage{mathrsfs}
\usepackage{amsmath}
\usepackage{amsthm}
\usepackage{amssymb}
\usepackage{graphicx}
\usepackage{wasysym}
\newtheorem{MyDefinition}{Definition}
\newtheorem{MyCorollary}{Corollary}
\newtheorem{MyLemma}{Lemma}
\newtheorem{thm}{Theorem}%

\begin{document}


\noindent\vspace{-18pt}

\section*{{Contents}}
1. Introduction

2. Theory $\mathscr{E}_{d}$---Quantitative Euclidean Geometry with
a Single Geometric Primitive

\hspace{12pt}Notion---Distance Function $d$

3. Analytic Geometry United under Synthetic Geometry in $\mathscr{E}_{d}$

4. Theory $\mathscr{E}_{da}$---A Conservative Extension of $\mathscr{E}_{d}$
with Angle Function $a$

5. Consistency of $\mathscr{E}_{d}$ and $\mathscr{E}_{da}$

6. Theory $\mathscr{E}_{d}$ and Tarski's $\mathscr{E}_{2}$ Are Mutually Interpretable
with Parameters

7. Completeness and Decidability of $\mathscr{E}_{d}$

8. Theory $\mathscr{H}_{da}$---Quantitative Hyperbolic Geometry

9. Theories $\mathscr{E}_{d}^{n}$ and $\mathscr{E}_{da}^{n}$---Quantitative
Euclidean Geometry in Higher Dimensions

Appendix \ref{appendix:a}. Axioms of Real Closed Fields (RCF)

Appendix \ref{appendix:b}. Tarski's Axioms of Plane Euclidean Geometry $\mathscr{E}_{2}$

Appendix \ref{appendix:c}. SMSG Axioms of Euclidean Geometry

Appendix \ref{appendix:d}. The Convoluted Statement of the Pythagorean Theorem in Tarski's $\mathscr{E}_2$

\section{Introduction}\label{sec:Section-1}

For more than two thousand years, Euclid\textquoteright s \emph{Elements}
\cite{Heath} stood as a standard of logical rigor. In the late nineteenth
century, however, its shortcomings were exposed, prompting efforts
to rebuild Euclidean geometry on firmer foundations. Two milestones
were Hilbert\textquoteright s \emph{Grundlagen der Geometrie} (1899)
\cite{Hilbert} and the axiom system $\mathscr{E}_{2}$ developed
by Tarski and his students \cite{Tarski-1,Tarski-2,Schwabhauser}.
Hilbert used natural language, while Tarski introduced a  first-order formal
system ``which can be formulated and established without the help
of any set-theoretical devices'' \citep{Tarski-1} (p.~16).

Tarski's $\mathscr{E}_{2}$ is formulated in a 1-sorted language:
the only primitive objects are points. There are two primitive predicate
symbols---a 3-place predicate symbol $B$ for betweenness, and a
4-place predicate symbol $D$ for segment congruence---with $B(\mathbf{pvq})$
meaning that $\mathbf{v}$ lies between $\mathbf{p}$ and $\mathbf{q}$,
and $D(\mathbf{pquv})$ meaning segment $\mathbf{pq}$ (as a pair
of end points) is congruent to segment $\mathbf{uv}$. This system
has 11 axioms \cite{Tarski-2,Schwabhauser} (including one axiom schema),
which are listed in Appendix \ref{appendix:b}. Tarski established
that $\mathscr{E}_{2}$ is complete and decidable.

What part of the content in the \emph{Elements} is not formalized
in $\mathscr{E}_{2}$? Notably, $\mathscr{E}_{2}$ contains no notion
of area---not even for polygons. In modern mathematics, area is a
set function assigning a real number to every measurable set. Even
when restricted to polygons, it still requires set-theoretic machinery,
which Tarski deliberately avoided in $\mathscr{E}_{2}$. 

The area function of polygons could be discussed in second-order logic,
or in first-order logic with two sorts---one sort for points and
the other for finite sets of points. The latter approach was discussed
by Tarski in his 1959 paper \cite{Tarski-1}, with a sketch of a system
$\mathscr{E}_{2}'$ as a possible extension of $\mathscr{E}_{2}$:
\begin{quotation}
``The theory $\mathscr{E}_{2}'$ is obtained by supplement the logical
base of $\mathscr{E}_{2}$ with a small fragment of set theory. Specifically,
we include in the symbolism of $\mathscr{E}_{2}'$ new variables $X,Y,\ldots$
assumed to range over arbitrary finite sets of points (or, what in
this case amounts essentially to the same, over arbitrary finite sequences
of points); we also include a new logical constant, the membership
symbol $\in$, to denote the membership relation between points and
finite point sets. \ldots{} In consequence the theory of $\mathscr{E}_{2}'$
considerably exceeds $\mathscr{E}_{2}$ in means of expression and
power. In $\mathscr{E}_{2}'$ we can formulate and study various notions
which are traditionally discussed in textbooks of elementary geometry
but which cannot be expressed in $\mathscr{E}_{2}$; e.g., the notions
of a polygon with arbitrarily many vertices, and of the circumference
and the area of a circle.

As regards metamathematical problems which have been discussed and
solved for $\mathscr{E}_{2}$ in Theorems 1--4, three of them---the
problems of representation, completeness, and finite axiomatizability---are
still open when referred to $\mathscr{E}_{2}'$. In particular, we
do not know any simple characterization of all models of $\mathscr{E}_{2}'$,
nor, do we know whether any two such models are equivalent with respect
to all sentences formulated in $\mathscr{E}_{2}'$.''
\end{quotation}
Tarski did, however, resolve the decision problem for $\mathscr{E}_{2}'$
in his Theorem 5 \cite{Tarski-1}:
\begin{quotation}
Theorem 5. The theory $\mathscr{E}_{2}'$ is undecidable, and so are
all of its consistent extensions.
\end{quotation}
\noindent This follows from the fact that Peano arithmetic (PA) is
(relatively) interpretable in $\mathscr{E}_{2}'$. 

Note that in the same paper \cite{Tarski-1}, Tarski discussed several
axiom systems, which he denoted $\mathscr{E}_{2}$, $\mathscr{E}{}_{2}'$,
$\mathscr{E}_{2}''$, and $\mathscr{E}_{2}'''$, but with his primary
focus on $\mathscr{E}_{2}$. Each of these, in his own words, qualifies
as a feasible interpretation of ``elementary geometry'', and the
problem of deciding which is the unique one ``seems to be rather
hopeless and deprived of broader interest''. For clarity, we retain
Tarski\textquoteright s original symbols when referring to these systems,
since terms such as ``\emph{Tarski\textquoteright s geometry}''
or ``\emph{elementary geometry}'' are ambiguous. Accordingly, claims
like \textquotedblleft Euclidean geometry is complete and decidable\textquotedblright{}
or \textquotedblleft elementary Euclidean geometry is complete and
decidable\textquotedblright{} are inaccurate, as the systems $\mathscr{E}{}_{2}'$,
$\mathscr{E}_{2}''$, and $\mathscr{E}_{2}'''$ provide immediate
counterexamples.

Having considered the role of area, we now turn to the theory of proportion
and similarity, which deals with quantitative relations such as segment
lengths. This is the focus of the present paper, which we refer to
as \emph{Quantitative Euclidean Geometry}. In informal treatments
of geometry, it is usually taken for granted that every line segment
has a length, which is a real number. The Pythagorean theorem provides
a good example to illustrate this point. In its modern form, the theorem
asserts
\begin{equation}
\overline{\mathbf{ab}}^{2}+\overline{\mathbf{ac}\vphantom{b}}^{2}=\overline{\mathbf{bc}}^{2}\label{eq:Pythagorean-1}
\end{equation}
for a right triangle, where the bars denote numerical segment lengths.
Ancient civilizations also knew this relation. The Egyptians reportedly
used a knotted rope to form a 3-4-5 right triangle. In China, it appeared
as the\emph{ gou-gu} theorem (or Gougu theorem). \emph{Gou} and \emph{gu}
refer to the shorter and longer legs of a right triangle respectively,
since \emph{gu} means \textquotedblleft thigh\textquotedblright{}
in Classical Chinese. Euclid, however, expressed the theorem differently.
   In what follows, we examine the possible approaches to the
Pythagorean theorem: what has been done, what can be done differently,
and what cannot be done within first-order logic and without any appeal
to set theory.
\textbf{\begin{enumerate}
\item[(1)] The Area Approach by Euclid
\end{enumerate}}

The Pythagorean theorem in the \emph{Elements} is stated as
follows \cite{Heath}:
\begin{quotation}
Proposition I.47 ~~In right-angled triangles the square on the side
subtending the right angle is equal to the squares on the sides containing
the right angle (Figure \ref{fig:Pythagorean-Euclid}).
\end{quotation}
\begin{center}
\vspace{-12pt}\includegraphics[scale=0.19]{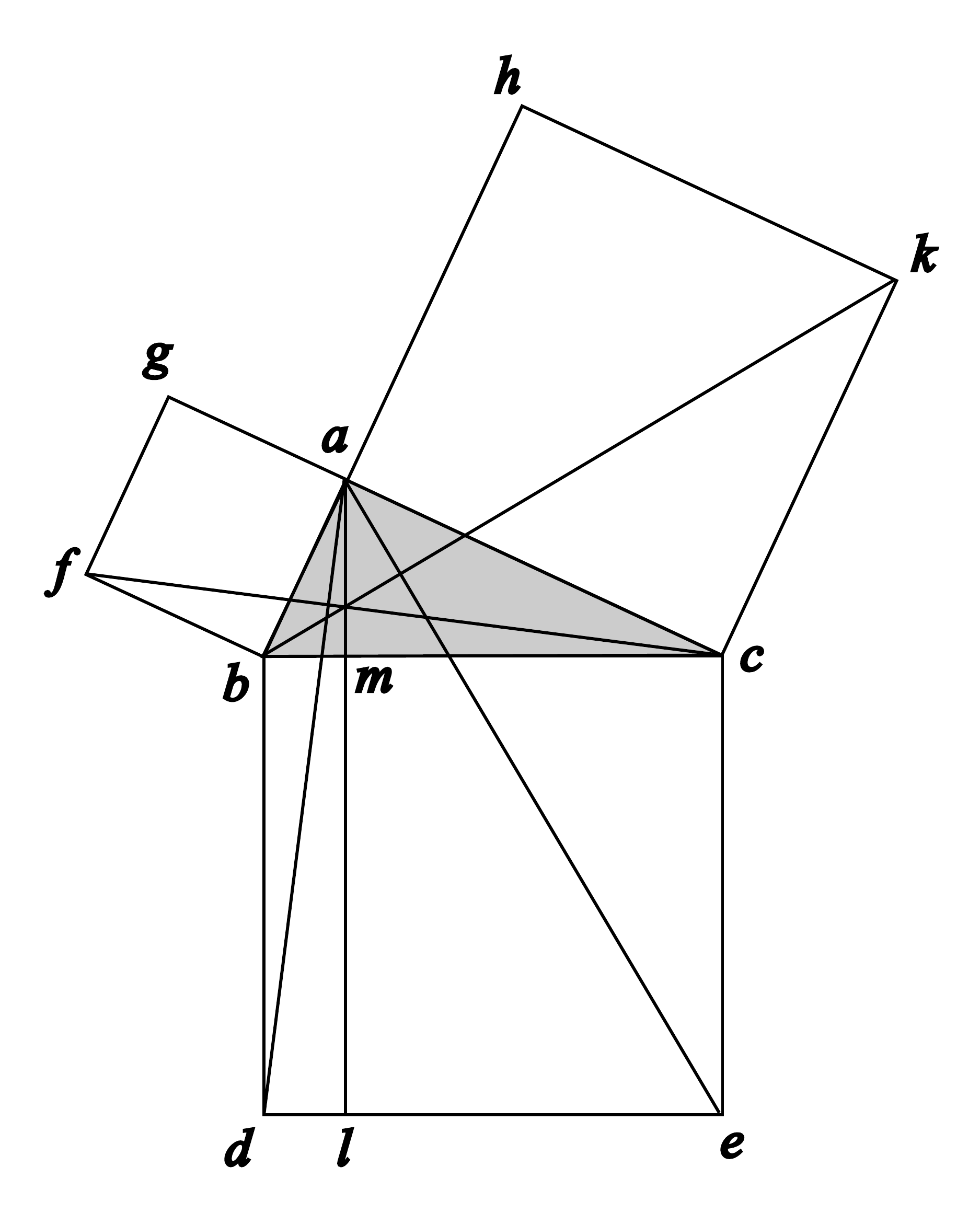}
\par\end{center}

\vspace{-15pt}
\begin{center}
\captionsetup{width=0.9\textwidth}
\captionof{figure}{Euclid: the Pythagorean theorem---The notion of area is used to state the theorem.} \label{fig:Pythagorean-Euclid}
\end{center}
\vspace{-5pt}

For Euclid, the square on $\mathbf{ab}$ meant a geometric figure (a rectangle of equal sides),
not the numerical product of a length and itself.

What, then, did he mean by saying that one square is \emph{equal to}
two other squares? Euclid did not provide an explicit definition,
yet one can infer it from his proof:

\textbf{Theorem}. (Pythagoras-Euclid I.47) In a right triangle, the
square on the hypotenuse can be dissected into $n$ triangles, and
the union of the two squares on the legs can also be dissected into
$n$ triangles with the resulting triangles congruent in pairs.

This formulation cannot be expressed in Tarski\textquoteright s language
$L(\mathscr{E}_{2})$, because $\mathscr{E}_{2}$ does not contain
a notion of area.

Hilbert developed a theory of area of polygons. His language is informal,
and he uses the concept of set freely. Hilbert defines two predicates:
\emph{equidecomposable} and \emph{equicomplementable}. Consider his
definition of ``equicomposable'':
\begin{quotation}
Definition. Two polygons are called \emph{equidecomposable} if they
can be decomposed into \emph{a finite number} of triangles that are
congruent in pairs.
\end{quotation}
In Hilbert's terms, what Euclid proved in his I.47 is actually: In
a right triangle, the square on the hypotenuse is equidecomposable
with the union of the two squares on the legs. However, \textquotedblleft a
finitely number\textquotedblright{} refers to a natural number, which
is not in the language of $\mathscr{E}_{2}$. The set of natural numbers
($\mathbb{N}$) is not definable in $\mathscr{E}_{2}$ or in real
closed fields (RCF). This helps explain why both $\mathscr{E}_{2}$
and RCF are complete and decidable, while the first-order Peano arithmetic
of natural numbers is not.

While Hilbert's theory of area is non-numerical, dealing with equidecomposable
and equicomplementable relations only, we may have a theory of area
in the numerical form (see Hartshorne \cite{Hartshorne}). However,
an area function that assigns a number to each polygon is a \emph{set
function}, which cannot be part of a first-order language similar
to $L(\mathscr{E}_{2})$. Consider, for example, a theorem in Hartshorne
\citep{Hartshorne} (p.~206):
\begin{quotation}
Theorem 23.2 ~~In a Hilbert plane with (P), there is an area function 
$\alpha$, with values in the additive group of the field of segment
arithmetic $F$, that satisfies and is uniquely determined by the
following additional condition: For any triangle $ABC$, whenever
we choose one side $AB$ to be the base and let it have length $b\in F$,
and let $h$ be the length of an altitude perpendicular to the base,
then $\alpha(ABC)=\frac{1}{2}bh$.
\end{quotation}
A formal first-order language prohibits quantification of functions
(such as ``there exists a function''),
not to mention that the area function $\alpha$ is a higher-order
function, meaning its domain is a set of sets.

Thus, the area approach is not viable for us: like Tarski, our aim
is to avoid the use of set theory and remain within a first-order
framework.

(As a side note, Hartshorne defines a Hilbert plane as one that satisfies Hilbert’s first three groups of axioms, and by (P) he means the axiom of parallels. Thus, this is simply a plane satisfying Hilbert’s first four groups of axioms, but without the axioms of continuity.)

\textbf{\begin{enumerate}
\item[(2)] The ``Segment Arithmetic'' Approach by Hilbert and SST
\end{enumerate}}

Here and in what follows, we abbreviate Schwabh\"{a}user, Szmielew,
and Tarski \cite{Schwabhauser} as ``SST''.

The Pythagorean theorem can also be formulated as the length relationship
in \mbox{Equation (\ref{eq:Pythagorean-1})}, and proved using similar
triangles. But how do Euclid, Hilbert, and Tarski treat the foundations
of proportion and similarity?

The theory of similarity is based on the theorem of Thales, also known
as the \emph{Fundamental Theorem of Proportion} or the \emph{Fundamental
Theorem of Similarity}. This is Proposition VI.2 in the \emph{Elements}:
\begin{quotation}
Proposition VI.2 ~~If a straight line be drawn parallel to one of
the sides of a triangle, it will cut the sides of the triangle proportionally;
and, if the sides of the triangle be cut proportionally, the line
joining the points of section will be parallel to the remaining side
of the triangle (Figure \ref{fig:Euclid-proportion}).
\end{quotation}
\begin{center}
\includegraphics[scale=0.21]{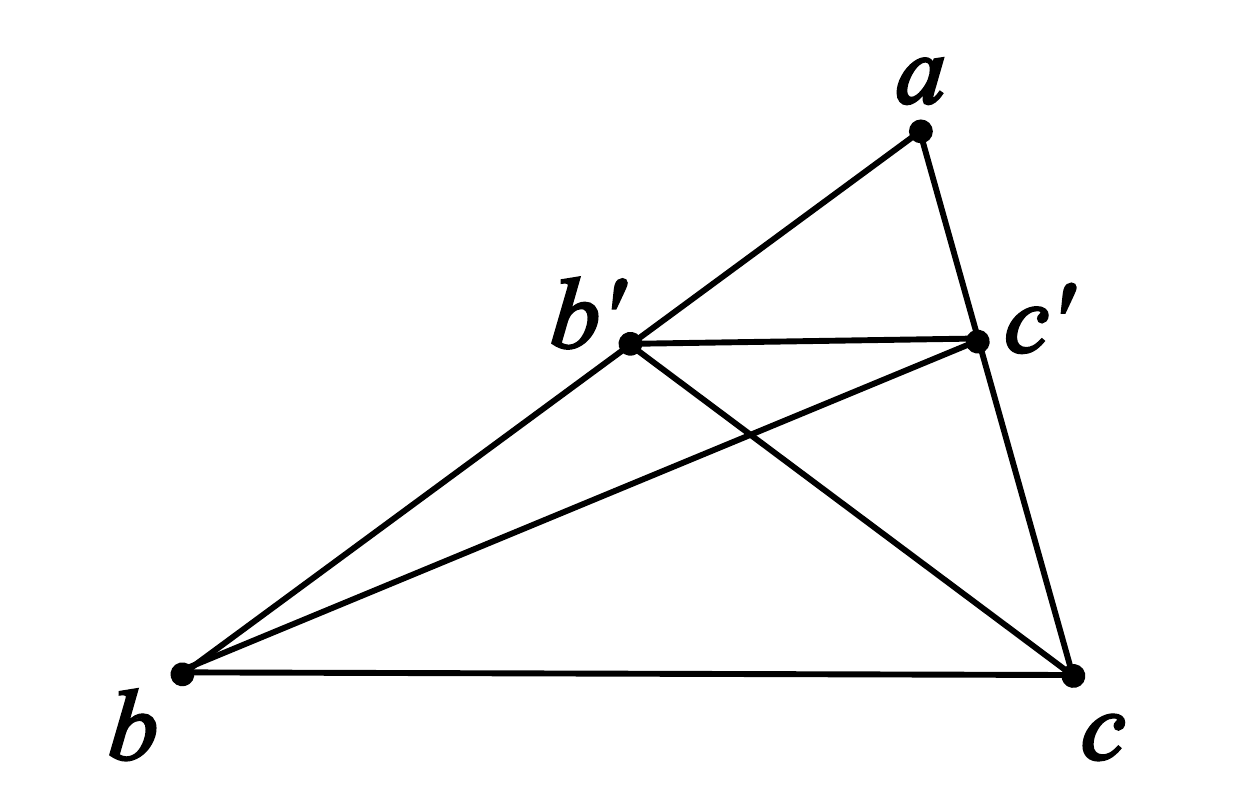}
\par\end{center}

\vspace{-15pt}
\begin{center}
\captionsetup{width=1.0\textwidth}
\captionof{figure}{Euclid: Fundamental Theorem of Proportion---The notion of area is not used to state the theorem but used in the proof.} \label{fig:Euclid-proportion}
\end{center}
\vspace{-5pt}

Euclid used the theory of area to prove this Fundamental Theorem of
Proportion. He argued that $\triangle\mathbf{b'bc'}$ and $\triangle\mathbf{c'cb'}$
have ``equal-area'' because they lie on the same base $\mathbf{b'c'}$
and are contained between the same parallels $\mathbf{b'c'}$ and
$\mathbf{bc}$. If we use $A(\mathbf{\mathbf{b'bc'}})$ to denote
the area of $\triangle\mathbf{\mathbf{b'bc'}}$, this can be written
as $A(\mathbf{b'bc'})=A(\mathbf{c'cb'})$. Consider $\triangle\mathbf{ab'c'}$
and $\triangle\mathbf{bb'c'}$. Their bases $\mathbf{ab'}$ and $\mathbf{bb'}$
are on the same line, and they have the same altitude. Therefore,
$A(\mathbf{ab'c'})/A(\mathbf{b'bc'})=\overline{\mathbf{ab'}}/\overline{\mathbf{b'b}}$.
Similarly, if we consider $\triangle\mathbf{ab'c'}$ and $\triangle\mathbf{c'cb'}$,
we can conclude $A(\mathbf{ab'c'})/A(\mathbf{c'cb'})=\overline{\mathbf{ac'}}/\overline{\mathbf{c'c}}$.
Therefore, $\overline{\mathbf{ab'}}/\overline{\mathbf{b'b}}=\overline{\mathbf{ac'}}/\overline{\mathbf{c'c}}$.
Note that we have used modern notation above. Euclid never used the
word \textquotedblleft area\textquotedblright{} explicitly.

Here, in this Fundamental Theorem of Proportion, the notion of area is not needed to
state the theorem, but Euclid used the theory of area to prove it.
Since Tarski's $\mathscr{E}_{2}$ lacks the notion of area, this proof
cannot be carried over to $\mathscr{E}_{2}$.

If the segments $\mathbf{ab'}$ and $\mathbf{b'b}$ are commensurable, there
exists an \emph{elementary} proof of this theorem that does not rely
on area. However, if $\mathbf{ab'}$ and $\mathbf{b'b}$ are incommensurable,
a limiting process must be invoked; in effect, this limiting method
is equivalent to an argument based on area.

Descartes and later Hilbert employed a solution by redefining the multiplication of segments.
Consider a right triangle, as shown in Figure \ref{fig:Hilbert-segment}. They choose the segment
$\mathbf{01}$ as a unit. Suppose $\mathbf{a}$ and $\mathbf{b}$
are two points on these orthogonal lines. Draw the line through $\mathbf{a}$
and $\mathbf{1}$, and then draw the line $\mathbf{bc}$ parallel
to $\mathbf{1a}$, intersecting $\mathbf{0a}$ at $\mathbf{c}$. The
segment $\mathbf{0c}$ is then \emph{defined to be} the product of
$\mathbf{0a}$ and $\mathbf{0b}$.
\begin{center}
\includegraphics[scale=0.21]{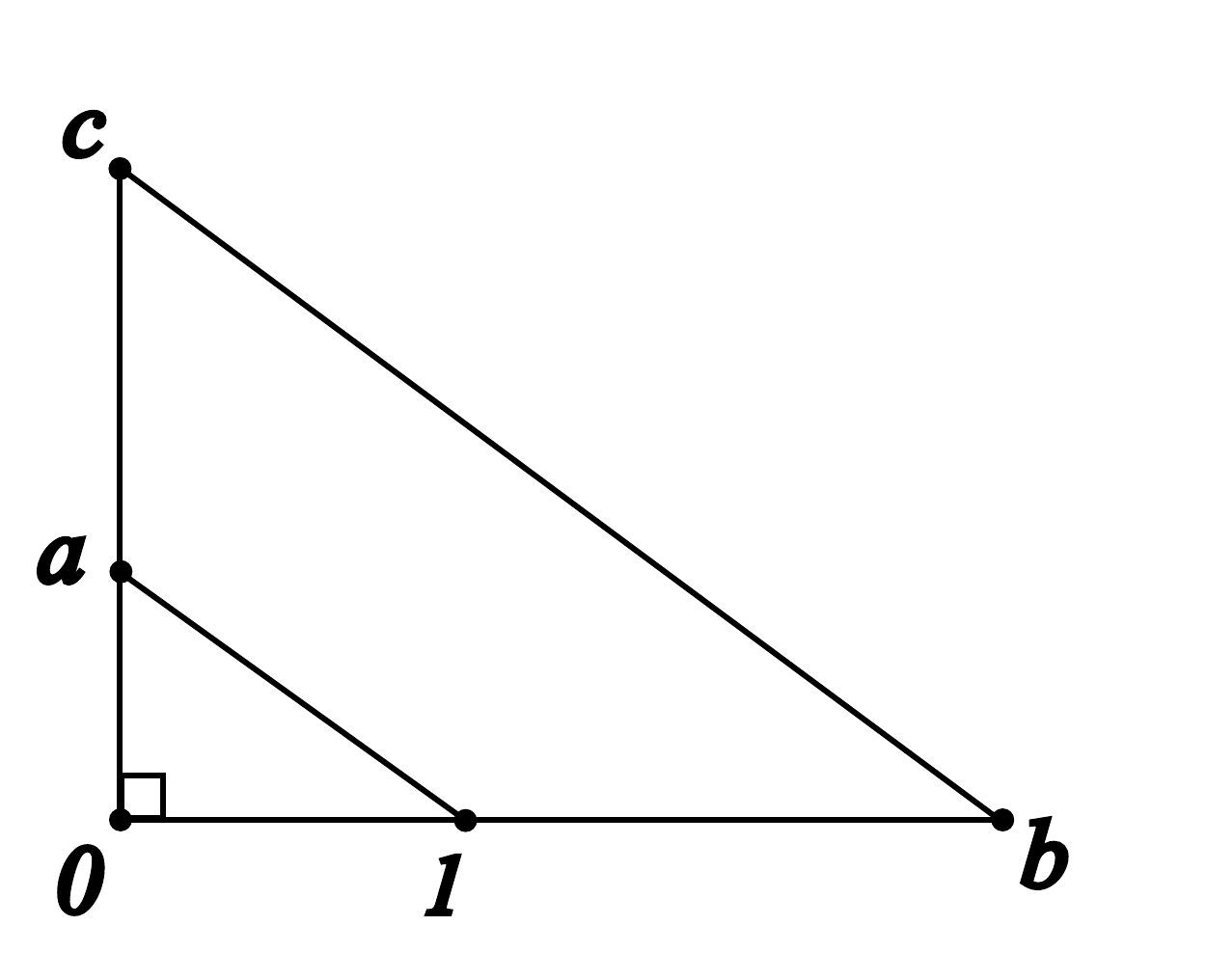}
\par\end{center}

\vspace{-15pt}
\begin{center}
\captionsetup{width=0.98\textwidth}
\captionof{figure}{Descartes and Hilbert: The definition of the product of two segments is a third segment.} \label{fig:Hilbert-segment}
\end{center}
\vspace{-5pt}

As Edwin Hewitt once remarked, ``Old theorems never die; they turn
into definitions.'' Indeed, here the theorem of Thales is turned into 
the definition of the multiplication of segments. It is important
to examine definitions carefully in mathematics. There are cases in
mathematical treatments where definitions are made convoluted in order
to make the theorems appear simple.

SST \cite{Schwabhauser} carried this further, recasting the 
segment arithmetic within $\mathscr{E}_{2}$. They defined addition
and multiplication directly on a coordinate line: the \textquotedblleft geometric
sum\textquotedblright{} and \textquotedblleft geometric product\textquotedblright{}
of two points yield another point on the same line, via projection
constructions (see these definitions in Section \ref{sec:Tarski's-Interpretable}).

The Pythagorean theorem in the form of $\overline{\mathbf{ab}}^{2}+\overline{\mathbf{ac}\vphantom{b}}^{2}=\overline{\mathbf{bc}}^{2}$
appears as Theorem 15.8 in SST \cite{Schwabhauser}, but this is a
highly abbreviated formula with the definitions of geometric product
and geometric sum. The complexity of this formula is hidden in the
abbreviations (or definitions). When these abbreviations are fully
expanded (see Appendix \ref{appendix:d}) using the definition of
the geometric product, the theorem becomes a convoluted statement about
constructed line segments. Each square is interpreted as a line segment.
To Euclid, a segment multiplied by a segment yields a rectangular
shape, while to Hilbert and SST, a segment multiplied by a segment
yields another segment. So, in SST \cite{Schwabhauser}, the meaning
of $\overline{\mathbf{ab}}^{2}+\overline{\mathbf{ac}\vphantom{b}}^{2}=\overline{\mathbf{bc}}^{2}$
is that two line segments add up to a third segment (see Appendix
\ref{appendix:d}).
\vspace{2pt}

In summary, we have examined two approaches to the Pythagorean theorem:
\vspace{4pt}

(i) Interpreting $(\mathbf{ab})^{2}$ as (the area of) the shape of
a square, as Euclid did. We have concluded that this is not a feasible
direction for us to proceed.
\vspace{4pt}

(ii) Interpreting $(\mathbf{ab})^{2}$ as (the equivalence class of)
a line segment, constructed using parallel projections, which is indirect
and cumbersome. 
\vspace{4pt}

Besides these two, a third possible approach is the following:
\textbf{\begin{enumerate}
\item[(3)] The Numerical Approach by Interpreting $(\mathbf{ab})^2$ as the product of a number $\overline{\mathbf{ab}}$ and itself
\end{enumerate}}

This is the direction we shall pursue in the present paper.
\vspace{2pt}

George Birkhoff \cite{Birkhoff} proposed a system with four postulates,
incorporating the notions of numerical measures of lengths and angles.
However, it is not a formal system: it freely uses sets, real numbers,
and $(1,1)$ correspondences between sets. For example, he postulates
the ``$(1,1)$ correspondence between the points on a line and the
real numbers $x$'' to introduce the distance measure, and the ``$(1,1)$
correspondence between half-lines and the real numbers $a$ (mod $2\pi$)''
to introduce the angle measure.
\vspace{2pt}

In 1961, the School Mathematics Study Group (SMSG) developed an axiom
system~\cite{SMSG} (see Appendix \ref{appendix:c}) intended for
American high school geometry courses. Like Birkhoff\textquoteright s
approach, it employs real numbers to measure distances and angles
but similarly lacks a formal language framework. While it designates
point, line, and plane as undefined terms, this appears to refer only
to primitive objects---the system does not explicitly declare undefined
predicates or functions. A review of the SMSG axioms reveals that
the system implicitly treats the distance function, angle function,
area function, and volume function as undefined. Terms such as lie
on, lie in, real number, set, contain, as well as correspondence,
intersection, and union of sets, are used in the axioms but are not
defined. The SMSG axiom list was made large and redundant, and the
lack of independence of the axioms was intentional for pedagogical
reasons.
\vspace{2pt}

The goal of the present paper is to incorporate numerical measures of distances
and angles as function symbols into a first-order theory of Euclidean
geometry, while avoiding set theory and retaining simplicity. Function
symbols are admissible as primitives in first-order logic, since they
are special cases of predicates.
\vspace{2pt}

We will show that the Axiom of Similarity plays a central role, effectively
killing two birds with one stone: it provides the foundation for the
theory of proportion and similarity, and it renders Euclid\textquoteright s
Parallel Postulate (EPP) as a theorem. The Axiom of Similarity can
be viewed as a quantitative formulation of EPP. The Pythagorean
theorem and other quantitative results from similarity theory can
then be expressed directly in the first-order languages of $\mathscr{E}_{d}$
and $\mathscr{E}_{da}$.

\section{Theory $\mathscr{E}_{d}$---Quantitative Euclidean Geometry with
a Single Geometric Primitive Notion---Distance Function \boldmath{$d$}\label{sec:Theory-Ed}}
\medskip{}

\noindent\textbf{Undefined Primitive Notions}
\medskip{}

The system $\mathscr{E}_{d}$ is formulated in a first-order formal
language
\[
L(\mathscr{E}_{d})=L(d;\,=,+,\cdot,<,0,1).
\]

$L(\mathscr{E}_{d})$ is a 2-sorted language: one sort for points,
and the other for numbers. We shall use standard lowercase letters
such as $x,y$ for numbers, and boldface lowercase letters such as
$\mathbf{a},\mathbf{b}$ for points.

The only geometric primitive notion is a 2-place function symbol $d$,
with $d(\mathbf{ab})$ denoting the distance between two points
$\mathbf{a}$ and $\mathbf{b}$, which is a number.

The symbol $=$ is overloaded for convenience: it denotes either the
equality of two numbers or the identity of two points, depending on
the context.

The symbols $+$ and $\cdot$ are 2-place function symbols for numbers,
denoting addition and multiplication. The symbol $<$ is a 2-place
predicate symbol for numbers, with $x<y$ denoting that $x$ is less
than $y$.

The symbols $0$ and $1$ are constant symbols for numbers.
\vspace{4pt}

\noindent\rule[0.5ex]{1\columnwidth}{1pt}

\noindent\textbf{\emph{$\mathscr{E}_{d}$---The Theory of Plane
Quantitative Euclidean Geometry with Distance Function }}

\noindent\textbf{\emph{in the Language $L(d;\,=,+,\cdot,<,0,1)$}}

\noindent\rule[0.5ex]{1\columnwidth}{1pt}

\medskip{}

We list the axioms of $\mathscr{E}_{d}$ in the following.\medskip{}

\noindent\textbf{Axioms RCF1 through RCF17} (for real closed fields, 
see Appendix \ref{appendix:a})\smallskip{}

RCF1 through RCF17 are the axioms for real closed fields, governing
the predicate and function symbols $=,+,\cdot,<$ for numbers. This
is the first-order formulation of real numbers. As real closed fields
are well studied, this is not the focus of the present paper.\medskip{}

Axioms D1 through D7 in the following are geometric axioms, in which
all variables range over points, except for $x$ in D4, which ranges
over numbers.\medskip{}

\noindent\textbf{Axiom D1}. Axiom of Nonnegativeness 
\[
\forall\mathbf{a}\forall\mathbf{b}\,\,\,\,0\leqslant d(\mathbf{ab}).
\]

\noindent\textbf{Axiom D2}. Axiom of Identity
\[
\forall\mathbf{a}\forall\mathbf{b}\,\,\,\,d(\mathbf{ab})=0\leftrightarrow\mathbf{a}=\mathbf{b}.
\]

\noindent\textbf{Axiom D3}. Axiom of Symmetry
\[
\forall\mathbf{a}\forall\mathbf{b}\,\,\,\,d(\mathbf{ab})=d(\mathbf{ba}).
\]

Axioms D1, D2, and D3 are the first three axioms of a metric space.
A metric space also includes one more axiom---the triangle inequality.
We do not need it as an axiom, as we will prove it in Theorem \ref{thm:(Triangle-inequality)}.
\medskip{}

\noindent\textbf{Axiom D4}. Axiom of segment extension
\[
\forall\mathbf{a}\forall\mathbf{b}\forall x\,\,\,\,0\leqslant x\rightarrow[\exists\mathbf{c}\,\,\,d(\mathbf{ac})=d(\mathbf{ab})+d(\mathbf{bc})\wedge d(\mathbf{bc})=x],
\]
where $\mathbf{a},\mathbf{b},\mathbf{c}$ are points and $x$ is a
number.\medskip{}

\noindent\textbf{Axiom D5}. Axiom of Similarity
\[
\begin{aligned} & \forall\mathbf{a}\forall\mathbf{b}\forall\mathbf{c}\forall\mathbf{b'}\forall\mathbf{c'}\\
 & [d(\mathbf{ab})=d(\mathbf{ab'})+d(\mathbf{b'b})\wedge d(\mathbf{ac})=d(\mathbf{ac'})+d(\mathbf{c'c})\wedge d(\mathbf{ab'})\cdot d(\mathbf{ac})=d(\mathbf{ac'})\cdot d(\mathbf{ab})]\\
 & \rightarrow d(\mathbf{ab'})\cdot d(\mathbf{bc})=d(\mathbf{b'c'})\cdot d(\mathbf{ab}).
\end{aligned}
\]

\begin{center}
\includegraphics[scale=0.2]{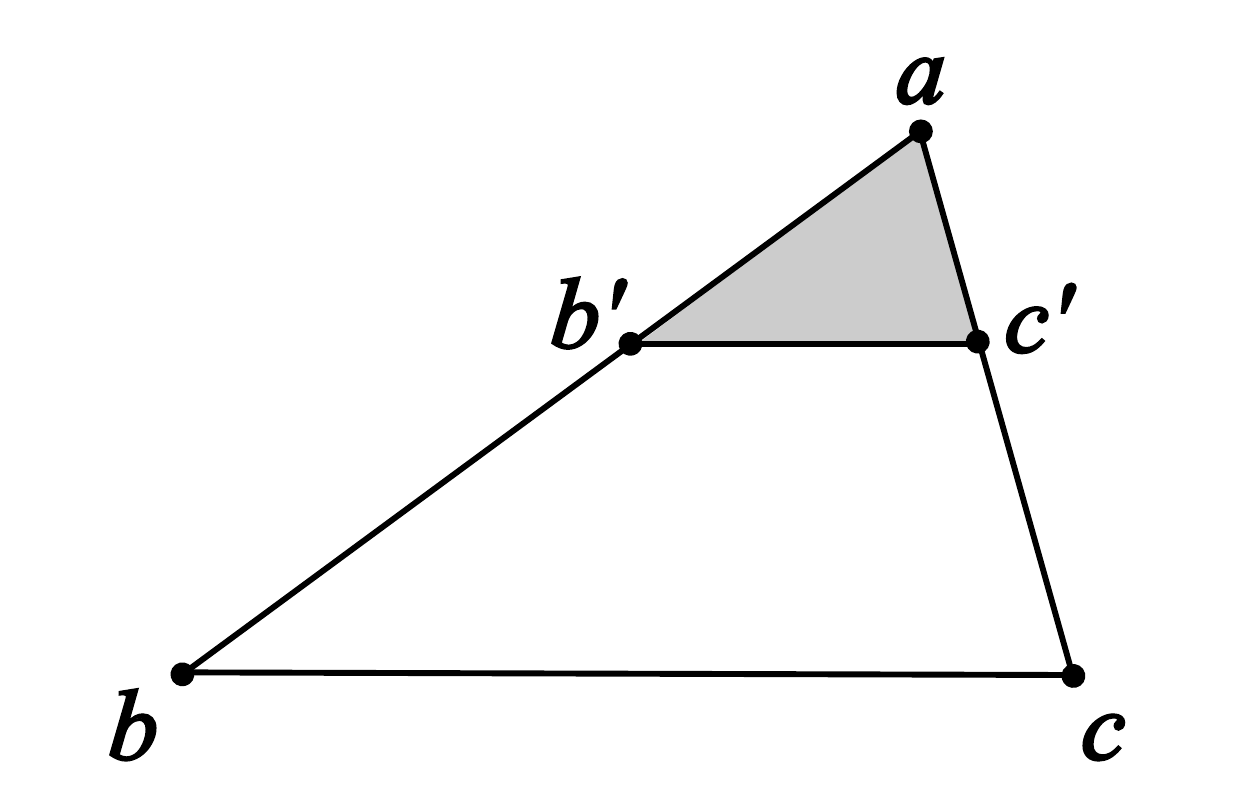}
\par\end{center}

\vspace{-15pt}
\begin{center}
\captionsetup{width=0.29\textwidth}
\captionof{figure}{Axiom of similarity.} \label{fig:Axiom-similarity}
\end{center}
\vspace{-5pt}

Informal explanation in plain English:

In $\triangle\mathbf{abc}$, suppose $\mathbf{b'c'}$ cuts $\mathbf{ab}$
and $\mathbf{ac}$ at $\mathbf{b'}$ and $\mathbf{c'}$ respectively.
If there is a number $k$ such that $d(\mathbf{ab})=k\cdot d(\mathbf{ab'})$
and $d(\mathbf{ac})=k\cdot d(\mathbf{ac'})$, then
\[
d(\mathbf{bc})=k\cdot d(\mathbf{b'c'}).
\]
Axiom D5 is expressed in the form of multiplications rather than ratios,
because we do not need to use an extra variable $k$. In addition,
an equation of products is more general than an equation of ratios,
as we do not need to worry about division by zero. \medskip{}

\noindent\textbf{Axiom D6}. Axiom of Dimension Lower Bound (at least
2)
\[
\exists\mathbf{p}_{0}\exists\mathbf{p}_{1}\exists\mathbf{p}_{2}\,\,\,\,d(\mathbf{p}_{0}\mathbf{p}_{1})\neq0\wedge d(\mathbf{p}_{0}\mathbf{p}_{1})=d(\mathbf{p}_{0}\mathbf{p}_{2})\wedge d(\mathbf{p}_{0}\mathbf{p}_{1})=d(\mathbf{p}_{1}\mathbf{p}_{2}).
\]

\begin{center}
\includegraphics[scale=0.22]{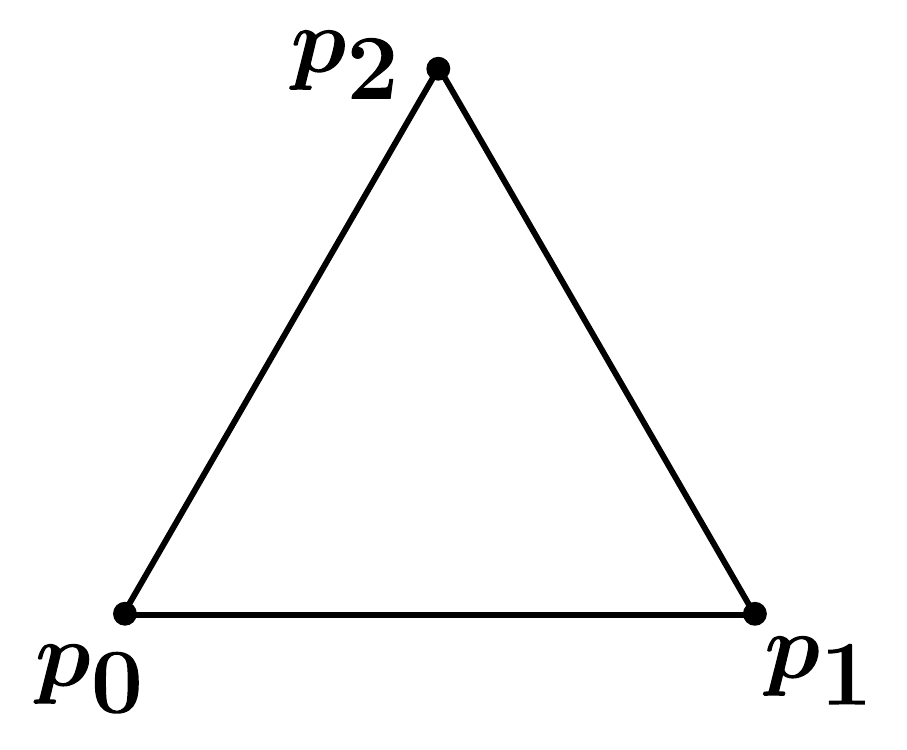}
\par\end{center}

\vspace{-15pt}
\begin{center}
\captionsetup{width=0.38\textwidth}
\captionof{figure}{The dimension is at least 2.} \label{fig:Lower-dim}
\end{center}
\vspace{-5pt}

\noindent\textbf{Axiom D7}. Axiom of Dimension Upper Bound (at most
2)
\begin{align*}
 & \forall\mathbf{p}_{0}\forall\mathbf{p}_{1}\forall\mathbf{p}_{2}\forall\mathbf{p}_{3}\\
 & [d(\mathbf{p}_{0}\mathbf{p}_{1})=d(\mathbf{p}_{0}\mathbf{p}_{2})\wedge d(\mathbf{p}_{0}\mathbf{p}_{1})=d(\mathbf{p}_{0}\mathbf{p}_{3})\\
 & \wedge d(\mathbf{p}_{0}\mathbf{p}_{1})=d(\mathbf{p}_{1}\mathbf{p}_{2})\wedge d(\mathbf{p}_{0}\mathbf{p}_{1})=d(\mathbf{p}_{1}\mathbf{p}_{3})\wedge d(\mathbf{p}_{0}\mathbf{p}_{1})=d(\mathbf{p}_{2}\mathbf{p}_{3})]\\
 & \rightarrow d(\mathbf{p}_{0}\mathbf{p}_{1})=0.
\end{align*}

\begin{center}
\includegraphics[scale=0.22]{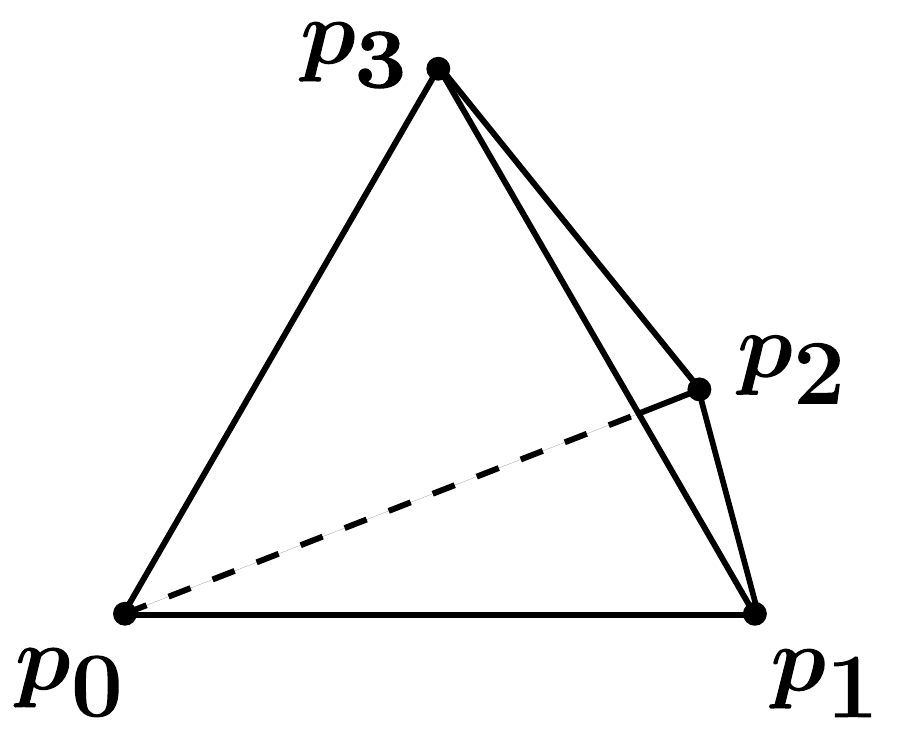}
\par\end{center}

\vspace{-15pt}
\begin{center}
\captionsetup{width=0.38\textwidth}
\captionof{figure}{The dimension is at most 2.} \label{fig:Higher-dim}
\end{center}
\vspace{-5pt}

 This concludes the list of axioms for $\mathscr{E}_{d}$. Note that
Tarski's $\mathscr{E}_{2}$ includes T11---axiom schema of continuity
(see Appendix \ref{appendix:b}). We do not require it here, since
RCF17 together with D4 implies T11.

\vspace{8pt}

\noindent A remark on the dimension axioms: 

Axiom D6 can be replaced by a weaker form:\medskip{}

\noindent\textbf{Axiom D6\textquotesingle}. Axiom of Dimension Lower
Bound (at least 2)
\vspace{2pt}
\begin{align*}
 & \exists\mathbf{p}_{0}\exists\mathbf{p}_{1}\exists\mathbf{p}_{2}\\
 & \,\,d(\mathbf{p}_{0}\mathbf{p}_{1})<d(\mathbf{p}_{1}\mathbf{p}_{2})+d(\mathbf{p}_{2}\mathbf{p}_{0})\\
 & \wedge d(\mathbf{p}_{1}\mathbf{p}_{2})<d(\mathbf{p}_{2}\mathbf{p}_{0})+d(\mathbf{p}_{0}\mathbf{p}_{1})\\
 & \wedge d(\mathbf{p}_{2}\mathbf{p}_{0})<d(\mathbf{p}_{0}\mathbf{p}_{1})+d(\mathbf{p}_{1}\mathbf{p}_{2}).
\end{align*}

This asserts that there exist three non-collinear points. When considered
in isolation, D6 is strictly stronger than D6\textquotesingle, because
D6 implies D6\textquotesingle ~but not vice versa. However, in the
presence of the rest of the axioms, they are equivalent, since D6\textquotesingle
~also implies D6. In general, using an axiom that is too weak may leave
certain true statements unprovable, while using one that is too strong
may make it no longer independent of other axioms. D6 is independent
in the system, because there exists a model of 1-dimensional geometry
where D6 fails while all other axioms remain true. When standing alone,
the negation of D6 does not imply that the dimension is one. However,
the negation of D6 together with the rest of the axioms does imply
that the dimension is one. We prefer D6 because it is simpler to
state (using primitives only without definitions or abbreviations)
and it generalizes more easily to higher dimensions (Section \ref{sec:Geometry-of-Higher}).
Basically, it is aligned with the affine-independence approach to
dimension but stated in metric terms: there exists a nondegenerate
2-dimensional (regular) simplex, but there does not exist a nondegenerate
3-dimensional (regular) simplex. This can be easily generalized to
$n$-dimensional geometry: there exists a nondegenerate $n$-dimensional
(regular) simplex, but there does not exist a nondegenerate ($n+1$)-dimensional
(regular) simplex.
\vspace{4pt}

\begin{MyDefinition}\label{def:Between} \emph{(Between)}

\emph{We shall use $B(\mathbf{abc})$ as the abbreviation for the
following:
\vspace{4pt}
\begin{align*}
 & B(\mathbf{abc})\\
:=\,\,\,\,\,\, & d(\mathbf{ac})=d(\mathbf{ab})+d(\mathbf{bc})
\end{align*}
}

\end{MyDefinition}

Note that the betweenness defined here is non-strict, meaning $B(\mathbf{acb})$
holds if $\mathbf{b}=\mathbf{a}$ or $\mathbf{b}=\mathbf{c}$. To
compare, Tarski used non-strict betweenness, while Hilbert used strict
betweenness.
\vspace{4pt}

\begin{MyDefinition} \emph{(Collinear)\label{def:Collinear}}

\emph{We shall use $L(\mathbf{abc})$ as the abbreviation for the
following:
\vspace{4pt}
\begin{align*}
 & L(\mathbf{abc})\\
:=\,\,\,\,\,\, & B(\mathbf{abc})\vee B(\mathbf{bca})\vee B(\mathbf{cab})
\end{align*}
}

\emph{Informally, $L(\mathbf{abc})$ means that the points $\mathbf{a},\mathbf{b},\mathbf{c}$
are collinear.}

\end{MyDefinition}
\pagebreak{}

\begin{MyDefinition} \emph{(Triangle congruence by SSS)} \label{def:TriAngle-congruence} 

\emph{We shall use $\triangle\mathbf{abc}\equiv\triangle\mathbf{a'b'c'}$
as the abbreviation for the following: 
\begin{align*}
 & \triangle\mathbf{abc}\equiv\triangle\mathbf{a'b'c'}\\
:=\,\,\,\,\,\, & d(\mathbf{ab})=d(\mathbf{a'b'})\wedge d(\mathbf{bc})=d(\mathbf{b'c'})\wedge d(\mathbf{ca})=d(\mathbf{c'a'})
\end{align*}
Informally, $\triangle\mathbf{abc}$ is said to be congruent to $\triangle\mathbf{a'b'c'}$
if all three pairs of corresponding sides are congruent.}

\end{MyDefinition}

\begin{MyDefinition} \emph{(Triangle similarity by SSS)} \label{def:Similar-triangles} 

\emph{We shall use $\triangle\mathbf{abc}\sim\triangle\mathbf{a'b'c'}$
as the abbreviation for the following: 
\begin{align*}
 & \triangle\mathbf{abc}\sim\triangle\mathbf{a'b'c'}\\
:=\,\,\,\,\,\, & \exists k\,\,k\neq0\wedge d(\mathbf{ab})=k\cdot d(\mathbf{a'b'})\wedge d(\mathbf{bc})=k\cdot d(\mathbf{b'c'})\wedge d(\mathbf{ca})=k\cdot d(\mathbf{c'a'})
\end{align*}
}

\emph{Informally, $\triangle\mathbf{abc}$ is said to be similar to
$\triangle\mathbf{a'b'c'}$ if their three corresponding sides are
in proportion. Clearly, triangle congruence is a special case of triangle
similarity where $k=1$.}

\end{MyDefinition}

Traditionally, two triangles are defined to be congruent if all their
corresponding sides are congruent and all their corresponding angles
are congruent. Two triangles are defined to be similar, if all their
corresponding sides are in proportion and all their corresponding
angles are congruent. However, it is easy to realize that the conditions
on the angles are redundant in these definitions, because the side
conditions alone imply that the corresponding angles are congruent---a
fact that will be proved as theorems. Therefore, Definitions \ref{def:TriAngle-congruence}
and \ref{def:Similar-triangles} differ from the traditional definitions
of triangle congruence and triangle similarity found in most textbooks.

\begin{MyDefinition} \label{def:Angle-congruence} \emph{(Angle congruence
by docking congruent triangles) }

\emph{We shall use $\angle\mathbf{pvq}\equiv\angle\mathbf{p'v'q'}$
as the abbreviation for the following: 
\begin{align*}
 & \angle\mathbf{pvq}\equiv\angle\mathbf{p'v'q'}\\
:=\,\,\,\,\,\, & \exists\mathbf{p''}\exists\mathbf{q''}\,\,\triangle\mathbf{p''v'q''}\equiv\triangle\mathbf{pvq}\wedge\left[B(\mathbf{v'p''p'})\vee B(\mathbf{v'p'p''})\right]\wedge\left[B(\mathbf{v'q''q'})\vee B(\mathbf{v'q'q''})\right]
\end{align*}
}

\end{MyDefinition}
\begin{center}
\includegraphics[scale=0.22]{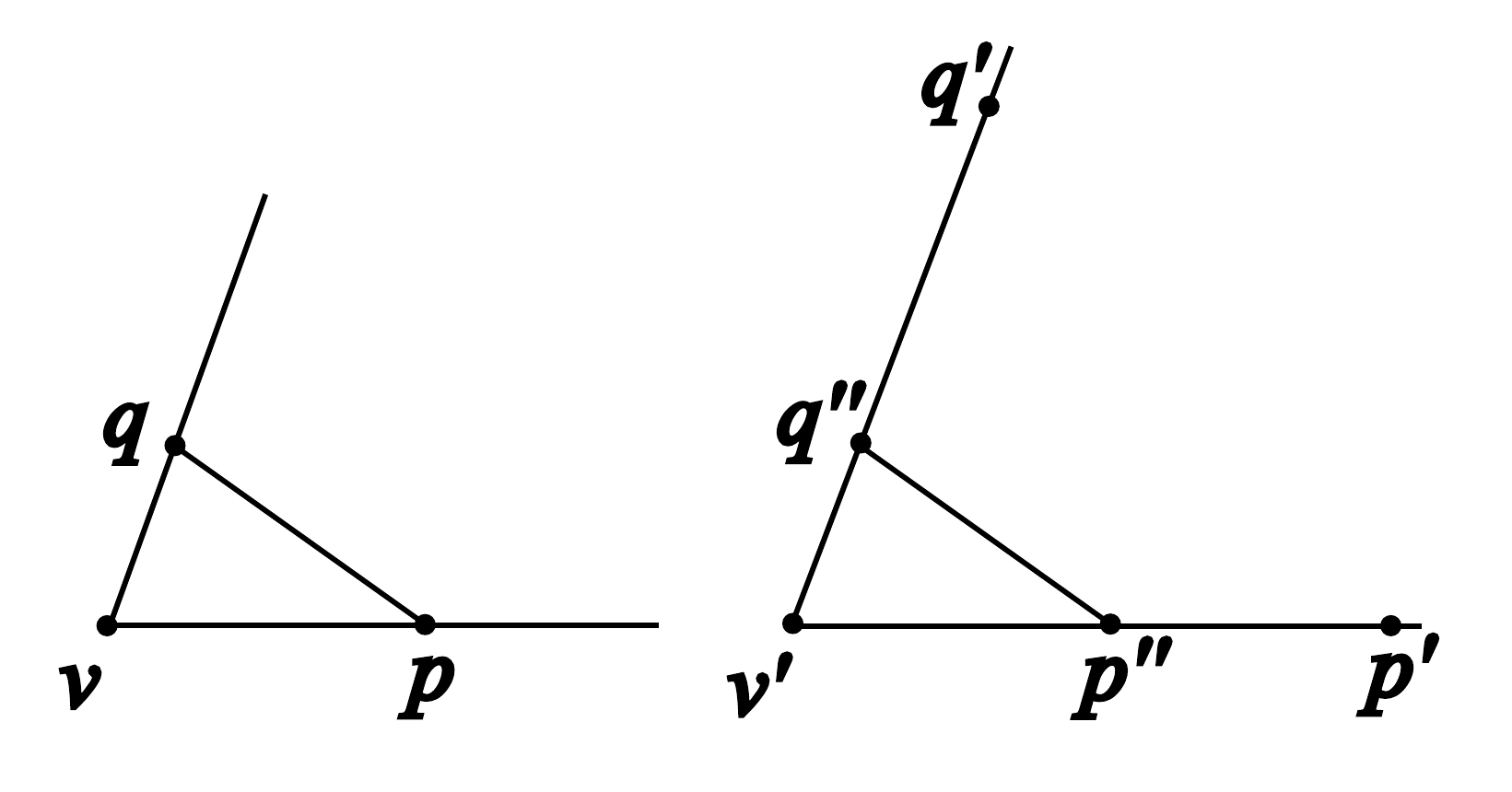}
\par\end{center}

\vspace{-15pt}
\begin{center}
\captionsetup{width=0.35\textwidth}
\captionof{figure}{Congruence of angles.} \label{fig:Congruence-angles}
\end{center}
\vspace{-5pt}

Informally, angle $\angle\mathbf{pvq}$ is congruent to angle $\angle\mathbf{p'v'q'}$,
if triangle $\triangle\mathbf{pvq}$ can be ``docked'' into triangle
$\triangle\mathbf{p'v'q'}$ at the vertex $\mathbf{v'}$. By ``docked'',
we mean there exists a triangle $\triangle\mathbf{p''v'q''}$ such
that $\triangle\mathbf{p''v'q''}\equiv\triangle\mathbf{pvq}$.

\begin{MyDefinition}\label{def:Sum-angles} \emph{(Addition of angles)}

\noindent\emph{Case 1: Addition of non-supplementary angles}

\emph{If $\neg B(\mathbf{pvq})$ and $\mathbf{p}\neq\mathbf{v}\wedge\mathbf{q}\neq\mathbf{v}$,
we shall use $\angle\mathbf{pvq}\equiv\angle\mathbf{pvt}\tilde{+}\angle\mathbf{tvq}$
as the abbreviation for the following (this is the case when $\angle\mathbf{pvt}$
and $\angle\mathbf{tvq}$ are not supplementary):
\begin{align*}
 & \angle\mathbf{pvq}\equiv\angle\mathbf{pvt}\tilde{+}\angle\mathbf{tvq}\\
:=\,\,\,\,\,\, & \neg B(\mathbf{pvq})\wedge\left[\exists\mathbf{a}\exists\mathbf{b}\,\,B(\mathbf{atb})\wedge[B(\mathbf{vap}\vee B(\mathbf{vpa})]\wedge[B(\mathbf{vbq}\vee B(\mathbf{vqb})]\right]
\end{align*}
}
\begin{center}
\includegraphics[scale=0.22]{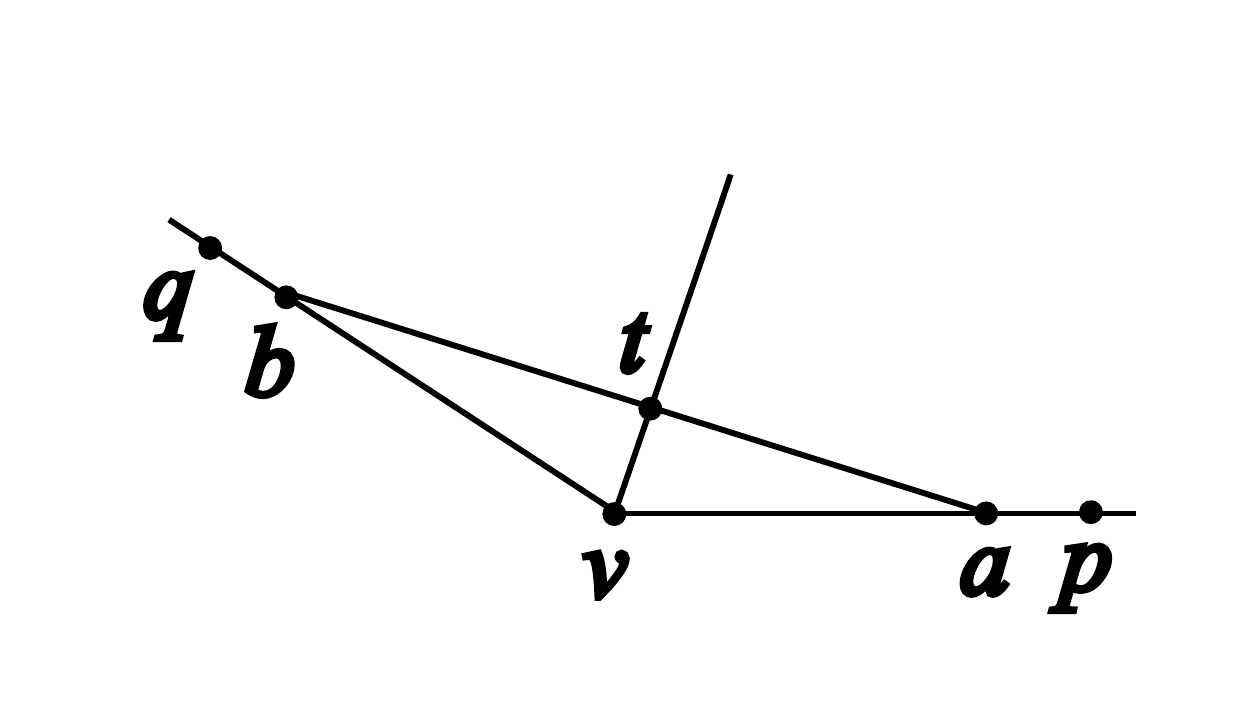}
\par\end{center}

\vspace{-15pt}
\begin{center}
\captionsetup{width=0.55\textwidth}
\captionof{figure}{Addition of angles---non-supplementary case.} \label{fig:Addition-angles-nonsupplimentary}
\end{center}
\vspace{-5pt}

\noindent\emph{Case 2: Addition of supplementary angles}

\emph{If $B(\mathbf{pvq})$},
\begin{align*}
 & \angle\mathbf{pvq}\equiv\angle\mathbf{pvt}\tilde{+}\angle\mathbf{tvq}\\
:=\,\,\,\,\,\, & B(\mathbf{pvq})\wedge\mathbf{p}\neq\mathbf{v}\wedge\mathbf{q}\neq\mathbf{v}\wedge\mathbf{t}\neq\mathbf{v}
\end{align*}

\begin{center}
\includegraphics[scale=0.22]{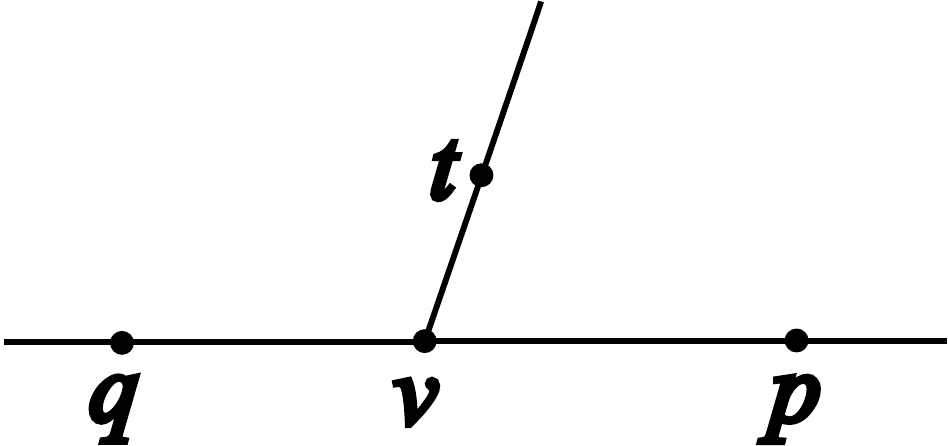}
\par\end{center}

\vspace{-15pt}
\begin{center}
\captionsetup{width=0.5\textwidth}
\captionof{figure}{Addition of angles---supplementary case.} \label{fig:Addition-angles-supplementary}
\end{center}
\vspace{-5pt}

\emph{Putting the two cases together, we define:
\begin{align*}
 & \angle\mathbf{pvq}\equiv\angle\mathbf{pvt}\tilde{+}\angle\mathbf{tvq}\\
:=\,\,\,\,\,\, & \left\{ \neg B(\mathbf{pvq})\wedge\left[\exists\mathbf{a}\exists\mathbf{b}\,\,B(\mathbf{atb})\wedge[B(\mathbf{vap}\vee B(\mathbf{vpa})]\wedge[B(\mathbf{vbq}\vee B(\mathbf{vqb})]\right]\right\} \\
 & \vee\left\{ B(\mathbf{pvq})\wedge\mathbf{p}\neq\mathbf{v}\wedge\mathbf{q}\neq\mathbf{v}\wedge\mathbf{t}\neq\mathbf{v}\right\} 
\end{align*}
}

\emph{When the angles do not have a common vertex, we may also use
$\angle\mathbf{u}_{1}\mathbf{a}\mathbf{u}_{2}\equiv\angle\mathbf{p_{1}bp_{2}}\tilde{+}\angle\mathbf{q}_{1}\mathbf{c}\mathbf{q}_{2}$.
By that, we mean there exist three other angles such that
\begin{align*}
 & \angle\mathbf{p_{1}bp_{2}}\equiv\angle\mathbf{pvt}\\
 & \wedge\angle\mathbf{q}_{1}\mathbf{c}\mathbf{q}_{2}\equiv\angle\mathbf{tvq}\\
 & \wedge\angle\mathbf{u}_{1}\mathbf{a}\mathbf{u}_{2}\equiv\angle\mathbf{pvq}\\
 & \wedge\angle\mathbf{pvq}\equiv\angle\mathbf{pvt}\tilde{+}\angle\mathbf{tvq}.
\end{align*}
}

\end{MyDefinition}

We can also define the comparison of two angles, or the order of angles.

\begin{MyDefinition}\label{def:Angles-comparison} \emph{(Order of
angles)}

\emph{We shall use the abbreviation $\angle\mathbf{pvq}\tilde{\leqslant}\mathbf{pvt}$
for the following:
\begin{align*}
 & \angle\mathbf{p_{1}bp_{2}}\tilde{\leqslant}\angle\mathbf{u}_{1}\mathbf{a}\mathbf{u}_{2}\\
:=\,\,\,\,\,\, & \exists\mathbf{q}_{1}\exists\mathbf{c}\exists\mathbf{q}_{2}\,\,\angle\mathbf{u}_{1}\mathbf{a}\mathbf{u}_{2}=\angle\mathbf{p_{1}bp_{2}}\tilde{+}\angle\mathbf{q}_{1}\mathbf{c}\mathbf{q}_{2}
\end{align*}
}

\emph{We also write $\angle\mathbf{p_{1}bp_{2}}\tilde{<}\angle\mathbf{u}_{1}\mathbf{a}\mathbf{u}_{2}$
to mean $\angle\mathbf{p_{1}bp_{2}}\tilde{\leqslant}\angle\mathbf{u}_{1}\mathbf{a}\mathbf{u}_{2}$
and $\angle\mathbf{p_{1}bp_{2}}\not\equiv\angle\mathbf{u}_{1}\mathbf{a}\mathbf{u}_{2}.$}

\end{MyDefinition}

Note that we have defined the sum of two angles and the meaning of
one angle is greater than another, but we have not assigned numerical
values to angles, because that is not in the language of $\mathscr{E}_{d}$.\medskip{}

\begin{MyDefinition} \emph{\label{def:Perpendicular}(Right angle) }

\emph{We shall use $R(\mathbf{abc})$ as the abbreviation for the
following:
\begin{align*}
 & R(\mathbf{abc})\\
:=\,\,\,\,\,\, & \mathbf{a}\neq\mathbf{b}\wedge\mathbf{a}\ne\mathbf{c}\wedge\mathbf{b}\ne\mathbf{c}\wedge[\exists\mathbf{a'}\,\,B(\mathbf{aba'})\wedge d(\mathbf{ba})=d(\mathbf{ba'})\wedge d(\mathbf{ca})=d(\mathbf{ca'})]
\end{align*}
}

\end{MyDefinition}
\begin{center}
\includegraphics[scale=0.22]{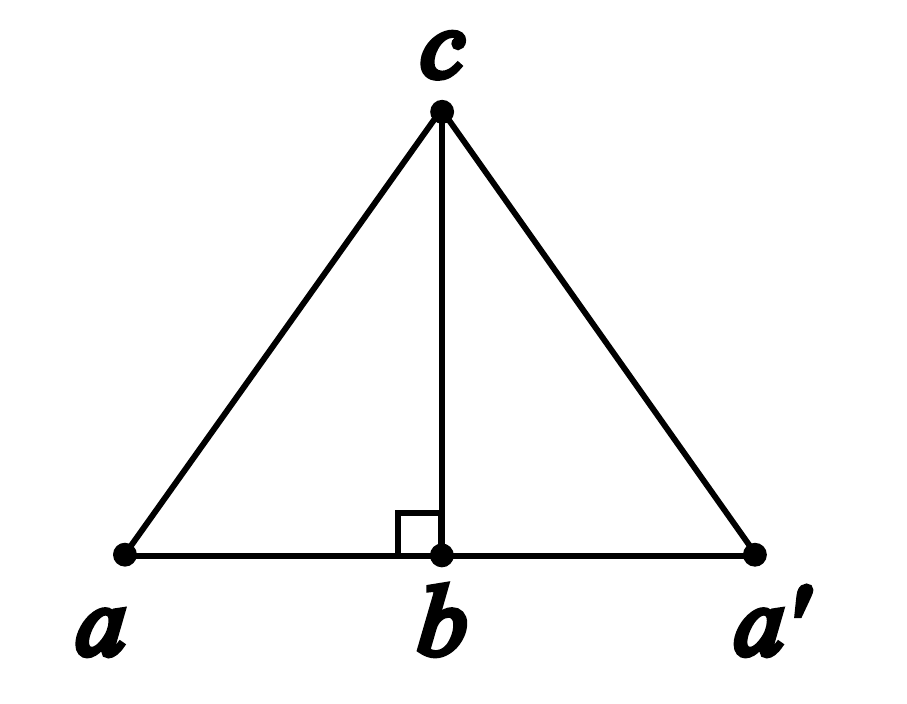}
\par\end{center}

\vspace{-15pt}
\begin{center}
\captionsetup{width=0.23\textwidth}
\captionof{figure}{Right angle.} \label{fig:Right-angle}
\end{center}
\vspace{-5pt}
\begin{thm}
\emph{\label{thm:(SSS-AAA)}(Triangle congruence SSS$\rightarrow$AAA)}

If two triangles are congruent, then their corresponding angles are
congruent. Formally,
\begin{align*}
 & \forall\mathbf{a}\forall\mathbf{b}\forall\mathbf{c}\forall\mathbf{a'}\forall\mathbf{b'}\forall\mathbf{c'}\\
 & \triangle\mathbf{abc}\equiv\triangle\mathbf{a'b'c'}\rightarrow\angle\mathbf{abc}\equiv\angle\mathbf{a'b'c'}\wedge\angle\mathbf{bca}\equiv\angle\mathbf{b'c'a'}\wedge\angle\mathbf{cab}\equiv\angle\mathbf{c'a'b'}.
\end{align*}
\end{thm}

\begin{proof}
The proof is trivial. The definition of the congruence of two triangles
\mbox{(Definition~\ref{def:TriAngle-congruence})} means the triangles have
three sides being congruent respectively. This simply fits the definition
of the congruence of each of the three pairs of corresponding angles
\mbox{(Definition~\ref{def:Angle-congruence}).}
\end{proof}
The reason we include this theorem and the following seemingly simple
theorems is that our Definition \ref{def:TriAngle-congruence} of
triangle congruence is different from most textbooks, and therefore
we need to demonstrate how this definition works.
\begin{thm}
\emph{\label{thm:(SAS)}(Triangle congruence criterion SAS) }

If two triangles have the two sides congruent to two sides respectively,
and the angle contained by the two sides congruent, then the two triangles
are congruent. Formally,\emph{
\begin{align*}
 & \forall\mathbf{a}\forall\mathbf{b}\forall\mathbf{c}\forall\mathbf{a'}\forall\mathbf{b'}\forall\mathbf{c'}\\
 & [d(\mathbf{ab})=d(\mathbf{a'b'})\wedge d(\mathbf{ac})=d(\mathbf{a'c'})\wedge\angle\mathbf{bac}\equiv\angle\mathbf{b'a'c'}]\rightarrow\triangle\mathbf{bac}\equiv\triangle\mathbf{b'a'c'}.
\end{align*}
}
\end{thm}

\begin{center}
\includegraphics[scale=0.25]{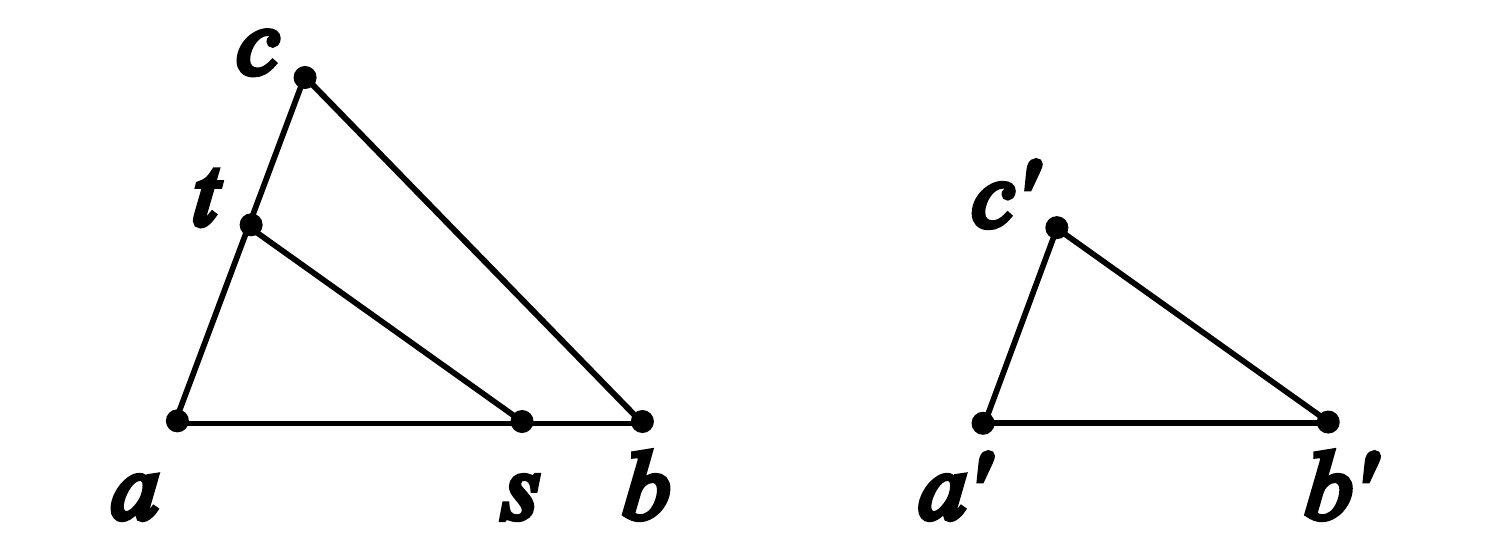}
\par\end{center}

\vspace{-15pt}
\begin{center}
\captionsetup{width=0.46\textwidth}
\captionof{figure}{Triangle congruence criterion SAS.}
\end{center}
\vspace{-15pt}
\begin{proof}
By the definition of $\angle\mathbf{bac}\equiv\angle\mathbf{b'a'c'}$,
there exist $\mathbf{s}$ on the line $\mathbf{ab}$ and $\mathbf{t}$
on the line $\mathbf{ac}$ such that $\triangle\mathbf{sat}\equiv\triangle\mathbf{b'a'c'}$.

We first discuss the two possible cases for point $\mathbf{s}$.

\noindent Case (1): $\mathbf{s}$ is between $\mathbf{a}$ and $\mathbf{b}$.
In this case we have $d(\mathbf{ab})=d(\mathbf{as})+d(\mathbf{sb})$. 

By the definition of $\angle\mathbf{bac}\equiv\angle\mathbf{b'a'c'}$,
we have $d(\mathbf{as})=d(\mathbf{a'b'})$. 

By the given condition $d(\mathbf{ab})=d(\mathbf{a'b'})$, we must
have $d(\mathbf{ab})=d(\mathbf{as})$.

Therefore, $d(\mathbf{sb})=0$.

By Axiom D2, we know $\mathbf{s}=\mathbf{b}$. Namely, point $\mathbf{s}$
coincides with $\mathbf{b}$.

\noindent Case (2): $\mathbf{b}$ is between $\mathbf{a}$ and $\mathbf{s}$.
In this case we have $d(\mathbf{as})=d(\mathbf{ab})+d(\mathbf{bs})$. 

By the definition of $\angle\mathbf{bac}\equiv\angle\mathbf{b'a'c'}$,
we have $d(\mathbf{as})=d(\mathbf{a'b'})$.

By the given condition $d(\mathbf{ab})=d(\mathbf{a'b'})$, we must
have $d(\mathbf{as})=d(\mathbf{ab})$.

Therefore, $d(\mathbf{bs})=0$.

By Axiom D2, we know $\mathbf{b}=\mathbf{s}$. Namely, point $\mathbf{b}$
coincides with $\mathbf{s}$.

\noindent In the same vein, we can prove that point $\mathbf{t}$
coincides with $\mathbf{c}$.

Therefore, $\triangle\mathbf{sat}$ coincides with $\triangle\mathbf{bac}$.
Because $\triangle\mathbf{sat}\equiv\triangle\mathbf{b'a'c'}$, we
must have $\triangle\mathbf{bac}\equiv\triangle\mathbf{b'a'c'}$.
\end{proof}
\begin{thm}
\emph{(Triangle congruence criterion ASA) }

If two triangles have two angles congruent to two angles respectively,
and the sides adjoining the congruent angles are congruent, then the
two triangles are congruent. Formally,\emph{
\begin{align*}
 & \forall\mathbf{a}\forall\mathbf{b}\forall\mathbf{c}\forall\mathbf{a'}\forall\mathbf{b'}\forall\mathbf{c'}\\
 & [\angle\mathbf{bac}\equiv\angle\mathbf{b'a'c'}\wedge\angle\mathbf{abc}\equiv\angle\mathbf{a'b'c'}\wedge d(\mathbf{ab})=d(\mathbf{a'b'})]\rightarrow\triangle\mathbf{cab}\equiv\triangle\mathbf{c'a'b'}.
\end{align*}
}
\end{thm}

\begin{center}
\vspace{-5pt}\includegraphics[scale=0.25]{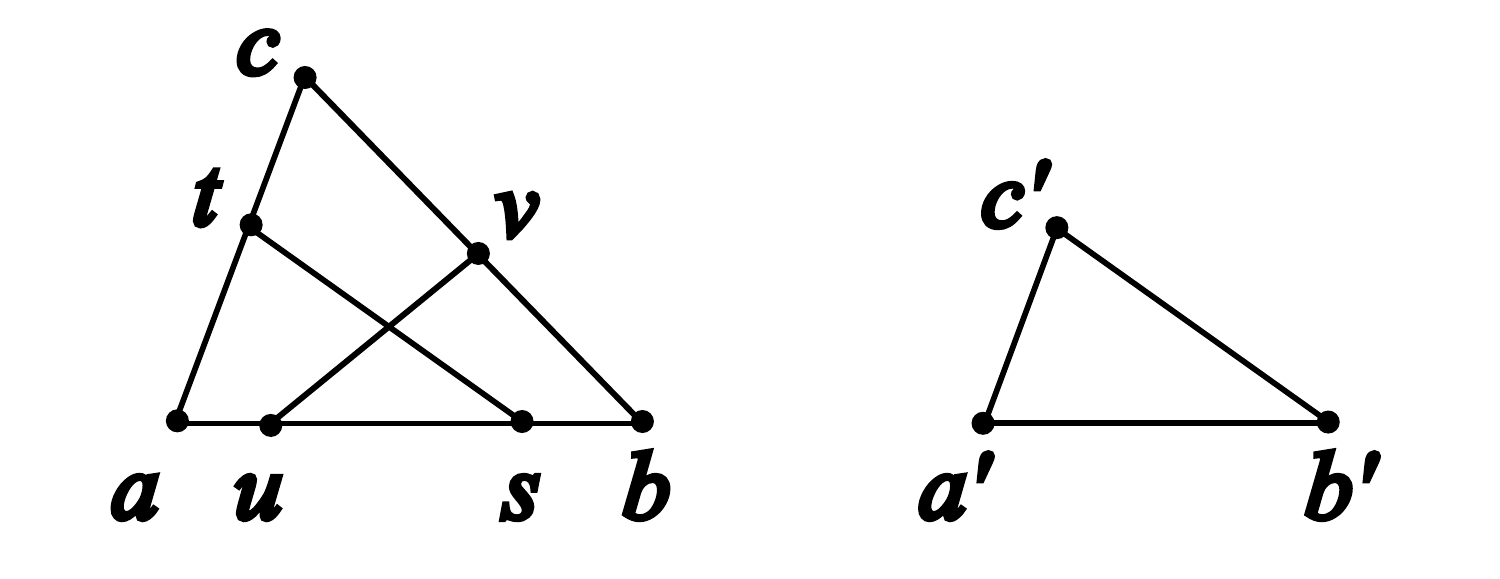}
\par\end{center}

\vspace{-15pt}
\begin{center}
\captionsetup{width=0.46\textwidth}
\captionof{figure}{Triangle congruence criterion ASA.} \label{fig:Congruence-ASA}
\end{center}
\vspace{-5pt}
\begin{proof}
By the definition of $\angle\mathbf{bac}\equiv\angle\mathbf{b'a'c'}$,
there exist points $\mathbf{s}$ and $\mathbf{t}$ on line $\mathbf{ab}$
and $\mathbf{ac}$ respectively such that $\triangle\mathbf{sat}\equiv\triangle\mathbf{b'a'c'}$.
Because $d(\mathbf{ab})=d(\mathbf{a'b'})$, by a similar argument
to that we used in the proof of Theorem \ref{thm:(SAS)}, we can show
that $\mathbf{s}$ coincides with $\mathbf{b}$.

By the definition of $\angle\mathbf{abc}\equiv\angle\mathbf{a'b'c'}$,
there exist points $\mathbf{u}$ and $\mathbf{v}$ on line $\mathbf{ab}$
and $\mathbf{bc}$ respectively such that $\triangle\mathbf{ubv}\equiv\triangle\mathbf{a'b'c'}$.
By a similar argument, we can show that $\mathbf{u}$ coincides with
$\mathbf{a}$. 

Now that $\mathbf{u}$ coincides with $\mathbf{a}$, and $\mathbf{s}$
coincides with $\mathbf{b}$, Figure \ref{fig:Congruence-ASA} can
be simplified as shown below.

\begin{center}

\includegraphics[scale=0.25]{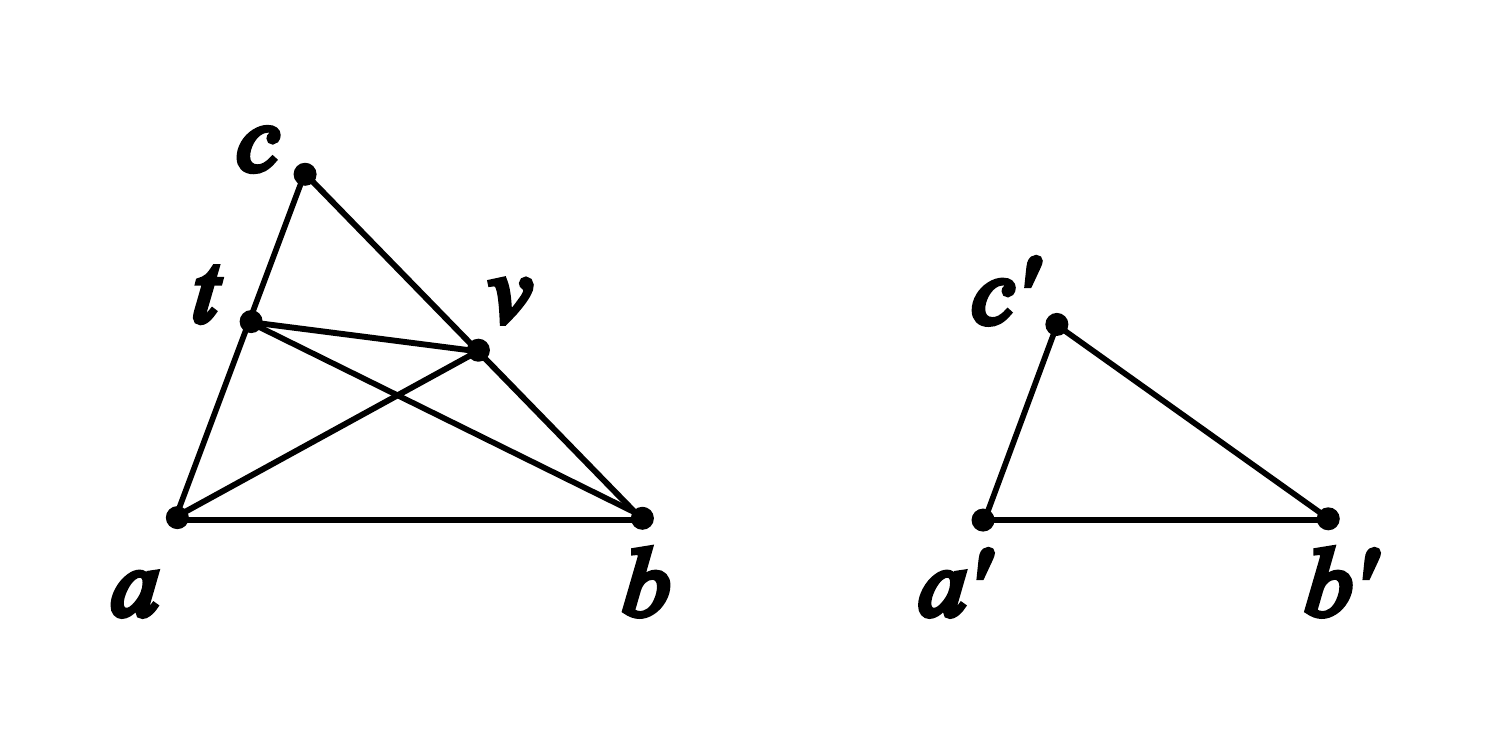}

\end{center}

Notice that by the definitions of $\angle\mathbf{bac}\equiv\angle\mathbf{b'a'c'}$
and $\angle\mathbf{abc}\equiv\angle\mathbf{a'b'c'}$, we have $\triangle\mathbf{abt}\equiv\triangle\mathbf{a'b'c'}$
and $\triangle\mathbf{abv}\equiv\triangle\mathbf{a'b'c'}$. Therefore,
$\triangle\mathbf{abt}\equiv\triangle\mathbf{abv}$. That implies
$\angle\mathbf{bat}\equiv\angle\mathbf{bav}$. Since $\angle\mathbf{bat}$
and $\angle\mathbf{bav}$ have the same side $\mathbf{ab}$, we know
$\mathbf{a},\mathbf{t},\mathbf{v}$ must be collinear. Because $d(\mathbf{at})=d(\mathbf{av})$,
$\mathbf{t}$ must coincide with $\mathbf{v}$. Furthermore, because
$\mathbf{c}$ is on both lines $\mathbf{at}$ and $\mathbf{bt}$,
$\mathbf{c}$ must coincide with $\mathbf{t}$. Namely, all three
points $\mathbf{t},\mathbf{v},\mathbf{c}$ coincide. Therefore, $\triangle\mathbf{cab}\equiv\triangle\mathbf{c'a'b'}$.
\end{proof}
\begin{MyCorollary}

\noindent\emph{\label{thm:(Isosceles-triangle)}(Isosceles triangle)}

In an isosceles triangle, the base angles are congruent to each other.
Formally,\emph{
\[
\forall\mathbf{a}\forall\mathbf{b}\forall\mathbf{c}\,\,\,d(\mathbf{ca})=d(\mathbf{cb})\rightarrow\angle\mathbf{cab}\equiv\angle\mathbf{cba}.
\]
}

\end{MyCorollary}

\begin{center}

\includegraphics[scale=0.25]{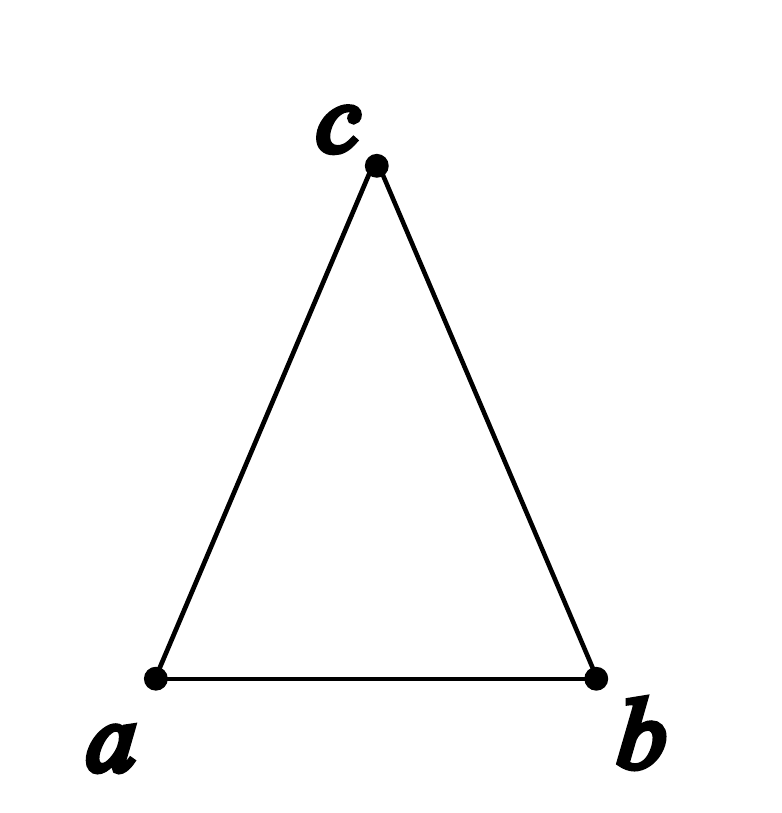}

\end{center}

\vspace{-15pt}
\begin{center}
\captionsetup{width=0.3\textwidth}
\captionof{figure}{Isosceles triangle.}
\end{center}
\vspace{-5pt}
\begin{proof}
The triangle can be viewed as two triangles in different orientations.
The first triangle $\triangle\mathbf{cab}$ has three sides $\mathbf{ca}$,
$\mathbf{ab}$, $\mathbf{bc}$ in this order. The second triangle
$\triangle\mathbf{cba}$ has three sides $\mathbf{cb}$, $\mathbf{ba}$,
$\mathbf{ac}$ in this order. The three sides are congruent respectively.
Therefore, $\triangle\mathbf{cab}\equiv\triangle\mathbf{cba}$. By
Theorem \ref{thm:(SSS-AAA)}, the corresponding angles are congruent.
Namely, $\angle\mathbf{cab}\equiv\angle\mathbf{cba}$.
\end{proof}

\begin{thm}
\emph{\label{thm:Similar-triangle-equal-angles} (Triangle similarity
SSS$\rightarrow$AAA) }

If two triangles are similar, then their corresponding angles are
congruent. Formally,
\[
\begin{aligned} & \forall\mathbf{a}\forall\mathbf{b}\forall\mathbf{c}\forall\mathbf{b'}\forall\mathbf{c'}\\
 & [\mathbf{a}\neq\mathbf{b}\wedge B(\mathbf{ab'b})\wedge B(\mathbf{ac'c})\wedge d(\mathbf{ab'})\cdot d(\mathbf{ac})=d(\mathbf{ac'})\cdot d(\mathbf{ab})]\\
 & \rightarrow[\angle\mathbf{ab'c'}\equiv\angle\mathbf{abc}\wedge\angle\mathbf{ac'b'}\equiv\angle\mathbf{acb}].
\end{aligned}
\]
\end{thm}

\begin{center}
\includegraphics[scale=0.2]{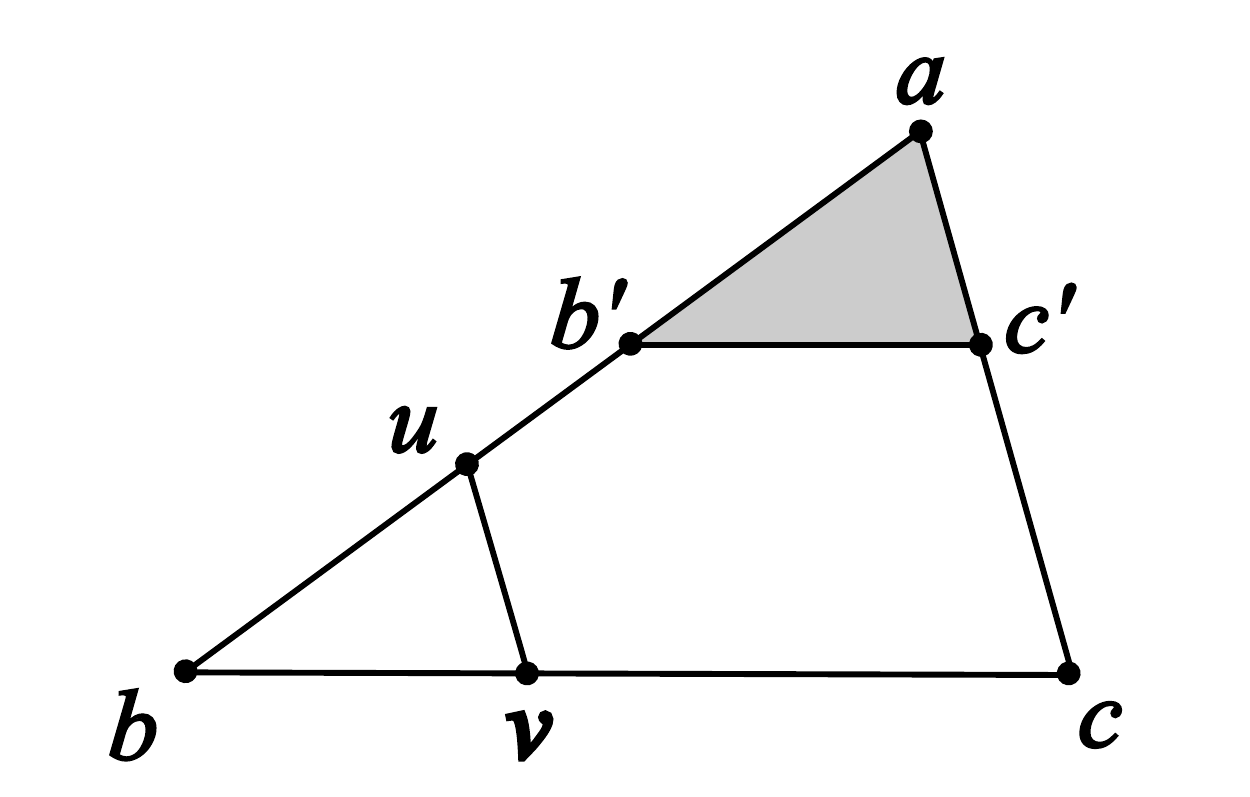}
\par\end{center}

\vspace{-15pt}
\begin{center}
\captionsetup{width=0.68\textwidth}
\captionof{figure}{Corresponding angles of similar triangles are congruent.}
\end{center}
\vspace{-15pt}
\begin{proof}
Let $k=d(\mathbf{ab'})/d(\mathbf{ab})$. By the given conditions,
we have $d(\mathbf{ab'})=k\cdot d(\mathbf{ab})$ and $d(\mathbf{ac'})=k\cdot d(\mathbf{ac})$.
By Axiom D5, $d(\mathbf{b'c'})=k\cdot d(\mathbf{bc})$.

Find the point $\mathbf{u}$ such that $B(\mathbf{aub})$ and $d(\mathbf{ub})=d(\mathbf{ab'})$.
Similarly, find the point $\mathbf{v}$ such that $B(\mathbf{bvc})$
and $d(\mathbf{bv})=d(\mathbf{b'c'})$. By Axiom D5, we have $d(\mathbf{uv})=k\cdot d(\mathbf{ac})$.
Therefore, $\triangle\mathbf{ab'c'}\equiv\triangle\mathbf{ubv}$.
This implies $\angle\mathbf{ab'c'}\equiv\angle\mathbf{abc}$. Similarly
we can prove $\angle\mathbf{ac'b'}\equiv\angle\mathbf{acb}$.
\end{proof}

\begin{thm}
\emph{(Triangle similarity criterion AA) \label{thm:Triangle-similarity-criterion-AA} }

If two triangles have two pairs of corresponding angles congruent,
then they are similar. Formally,
\[
\begin{aligned} & \forall\mathbf{a}\forall\mathbf{b}\forall\mathbf{c}\forall\mathbf{b'}\forall\mathbf{c'}\\
 & [\mathbf{a}\neq\mathbf{b}\wedge\mathbf{a}\neq\mathbf{b'}\wedge B(\mathbf{ab'b})\wedge B(\mathbf{ac'c})\wedge\angle\mathbf{ab'c'}\equiv\angle\mathbf{abc}]\rightarrow[d(\mathbf{ab'})\cdot d(\mathbf{ac})=d(\mathbf{ac'})\cdot d(\mathbf{ab})].
\end{aligned}
\]
\end{thm}

\begin{center}
\includegraphics[scale=0.2]{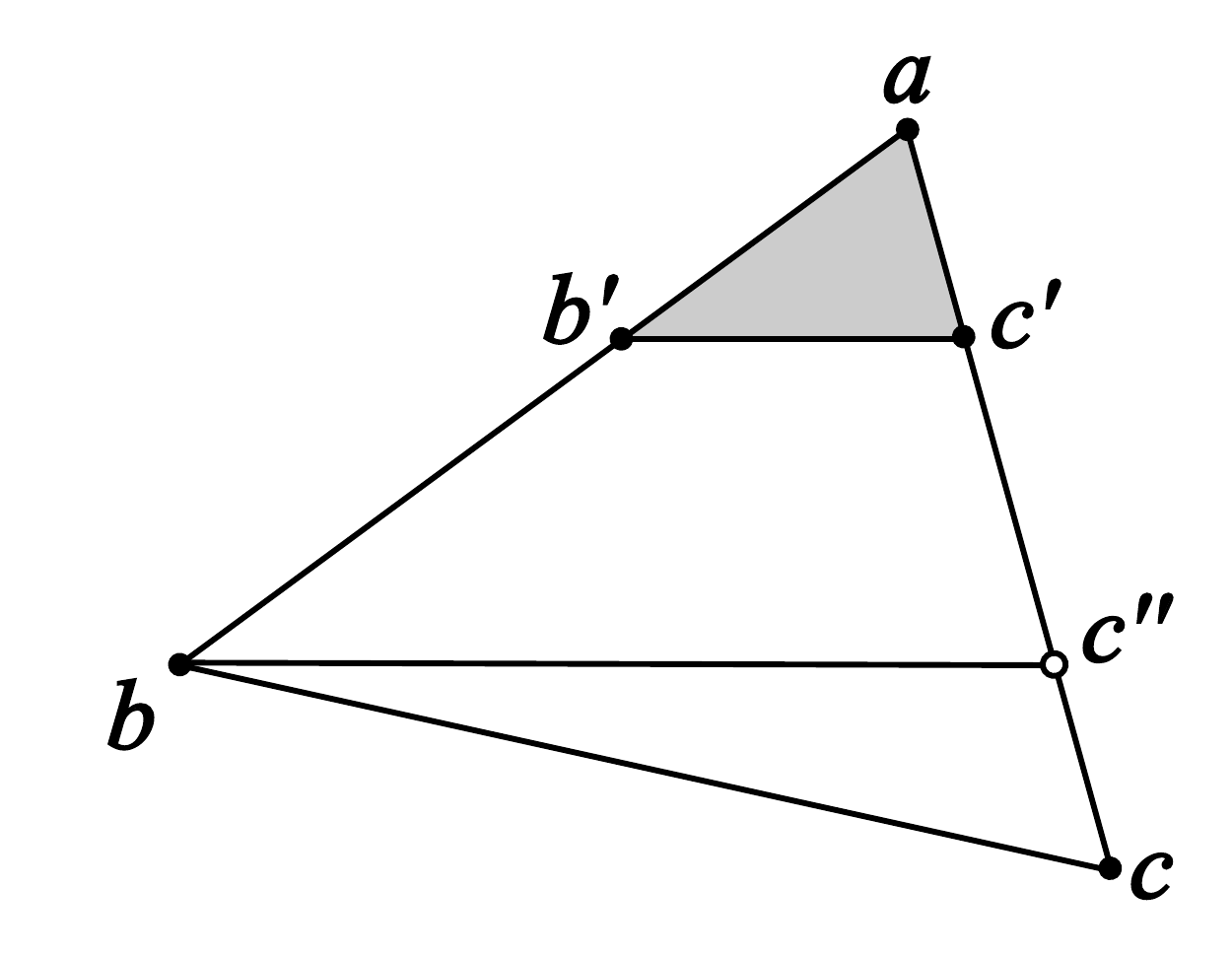}
\par\end{center}

\vspace{-15pt}
\begin{center}
\captionsetup{width=0.42\textwidth}
\captionof{figure}{Triangle similarity criterion AA.}
\end{center}
\vspace{-15pt}
\begin{proof}
Let $k=d(\mathbf{ab})/d(\mathbf{ab'})$. On line $\mathbf{ac'}$,
find the point $\mathbf{c''}$ such that $d(\mathbf{ac''})=k\cdot d(\mathbf{ac'})$
and $B(\mathbf{acc''})$ or $B(\mathbf{ac''c})$. By Axiom D5, $\triangle\mathbf{ab'c'}$
is similar to $\triangle\mathbf{abc''}$. By Theorem \ref{thm:Similar-triangle-equal-angles},
$\angle\mathbf{ab'c'}\equiv\angle\mathbf{abc''}$. But it is given
that $\angle\mathbf{ab'c'}\equiv\angle\mathbf{abc}$. Therefore, $\angle\mathbf{abc''}\equiv\angle\mathbf{abc}$.
We say $\mathbf{c}$ must coincide with $\mathbf{c''}$, because both
$\mathbf{c}$ and $\mathbf{c''}$ are on the same line $\mathbf{ac'}$.
Otherwise it would contradict the fact $\angle\mathbf{abc''}\equiv\angle\mathbf{abc}$.
Therefore, $d(\mathbf{ac''})=d(\mathbf{ac})$ and $d(\mathbf{ac})=k\cdot d(\mathbf{ac'})$,
and it implies $d(\mathbf{ab'})\cdot d(\mathbf{ac})=d(\mathbf{ac'})\cdot d(\mathbf{ab})$.
\end{proof}

\begin{thm}
\emph{(Interior angle sum of a triangle) \label{thm:Interior-angle-sum}}

In any triangle $\triangle\mathbf{abc}$, the sum of its interior
angles is a straight angle. Formally,
\begin{align*}
 & \forall\mathbf{a}\forall\mathbf{b}\forall\mathbf{c}\forall\mathbf{p}\forall\mathbf{q}\forall\mathbf{v}\\
\, & B(\mathbf{apb})\wedge B(\mathbf{bvc})\wedge B(\mathbf{aqc})\wedge d(\mathbf{ap})=d(\mathbf{pb})\wedge d(\mathbf{bv})=d(\mathbf{vc})\wedge d(\mathbf{aq})=d(\mathbf{qc})\\
 & \rightarrow\angle\mathbf{cba}\equiv\angle\mathbf{cvq\wedge\angle\mathbf{bac}\equiv\angle\mathbf{qvp}\wedge\angle\mathbf{acb}\equiv\angle\mathbf{pvb}}.
\end{align*}
\end{thm}

\begin{center}
\vspace{-10pt}\includegraphics[scale=0.2]{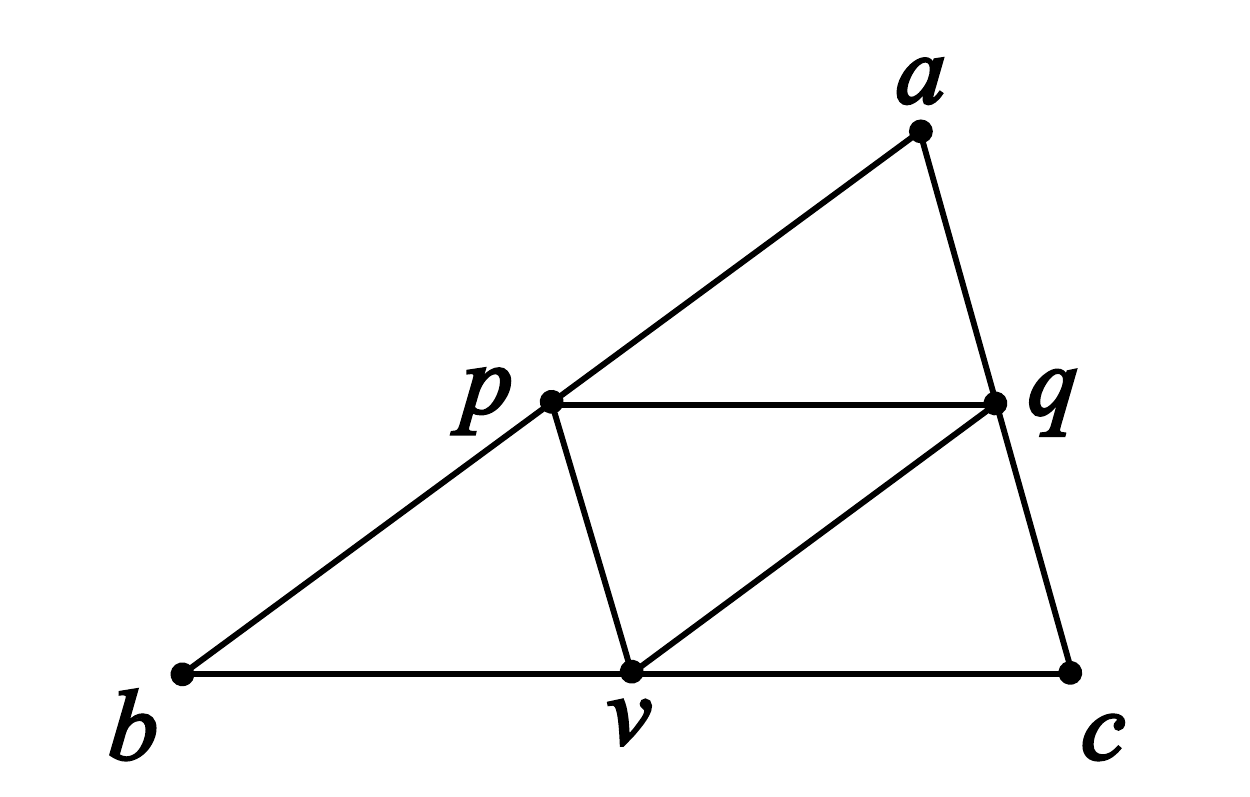}
\par\end{center}

\vspace{-15pt}
\begin{center}
\captionsetup{width=0.65\textwidth}
\captionof{figure}{The interior angle sum of a triangle is a straight angle.}
\end{center}
\vspace{-5pt}
\begin{proof}
The given conditions imply that $\mathbf{p},\mathbf{q},\mathbf{v}$
are the midpoints of $\mathbf{ab}$, $\mathbf{ac}$, and $\mathbf{bc}$
respectively. By Axiom D5, $d(\mathbf{pv})=\frac{1}{2}d(\mathbf{ac})$.
This implies $d(\mathbf{pv})=d(\mathbf{qc})$. In the same vein, we
can show $d(\mathbf{qv})=d(\mathbf{pb})$. Therefore, we have $\triangle\mathbf{vbp}\equiv\triangle\mathbf{cvq}$.
As a consequence, $\angle\mathbf{vbp}$ is congruent to $\angle\mathbf{cvq}$,
and $\angle\mathbf{qcv}$ is congruent to $\angle\mathbf{pvb}$.

By Axiom D5 again, we have $d(\mathbf{pv})=d(\mathbf{aq})$, and $d(\mathbf{qv})=d(\mathbf{ap})$.
Therefore, we have $\triangle\mathbf{paq}\equiv\triangle\mathbf{qvp}$.
As a consequence, $\angle\mathbf{paq}$ is congruent to $\angle\mathbf{qvp}$.
Because $\angle\mathbf{cvq}$, $\angle\mathbf{qvp}$, and $\angle\mathbf{pvb}$
form a straight angle, the interior angles of the triangle, $\angle\mathbf{cba}$,
$\angle\mathbf{bac}$, and $\angle\mathbf{acb}$ also ``sum up''
to a straight angle.
\end{proof}
It is well known that the interior angle sum theorem is equivalent
to Euclid\textquoteright s Parallel Postulate (EPP), and we will omit
the proof of EPP here. 

\medskip{}

As we can see in the following, using the distance function $d$,
the Pythagorean theorem can be formally and directly stated in the
language $L(\mathscr{E}_{d})$.
\begin{thm}
\emph{\label{thm:(Pythagoras)}(Pythagoras) 
\[
\forall\mathbf{a}\forall\mathbf{b}\forall\mathbf{c}\,\,\,R(\mathbf{bac})\rightarrow d^{2}(\mathbf{ab})+d^{2}(\mathbf{ac})=d^{2}(\mathbf{bc}),
\]
}

\noindent where $d^{2}(\mathbf{ab})$ is the abbreviation of $\left(d(\mathbf{ab)}\right)^{2}$.
\end{thm}

\begin{center}
\includegraphics[scale=0.2]{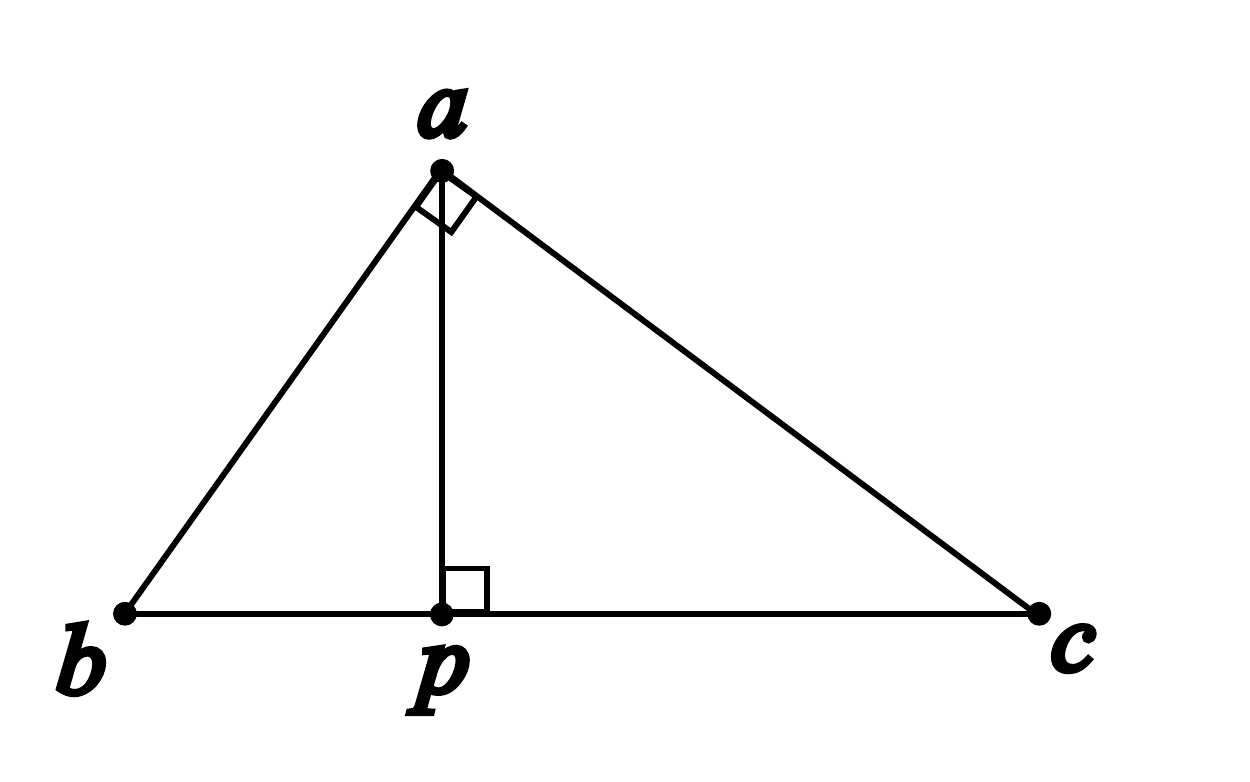}
\par\end{center}

\vspace{-15pt}
\begin{center}
\captionsetup{width=0.42\textwidth}
\captionof{figure}{The Pythagorean theorem.} \label{fig:Pythagorean-Thm}
\end{center}
\vspace{-5pt}
\begin{proof}
From $\mathbf{a}$, drop a perpendicular to $\mathbf{bc}$ and let
the foot be $\mathbf{p}$, as shown in Figure \ref{fig:Pythagorean-Thm}.

First, we want to prove that $B(\mathbf{bpc})$. 

We use proof by contradiction. For the sake of contradiction, suppose
that the foot $\mathbf{p}$ is not between $\mathbf{b}$ and $\mathbf{c}$.
Without loss of generality, assume $B(\mathbf{pbc})$, as shown in
the figure below.

\begin{center}

\includegraphics[scale=0.2]{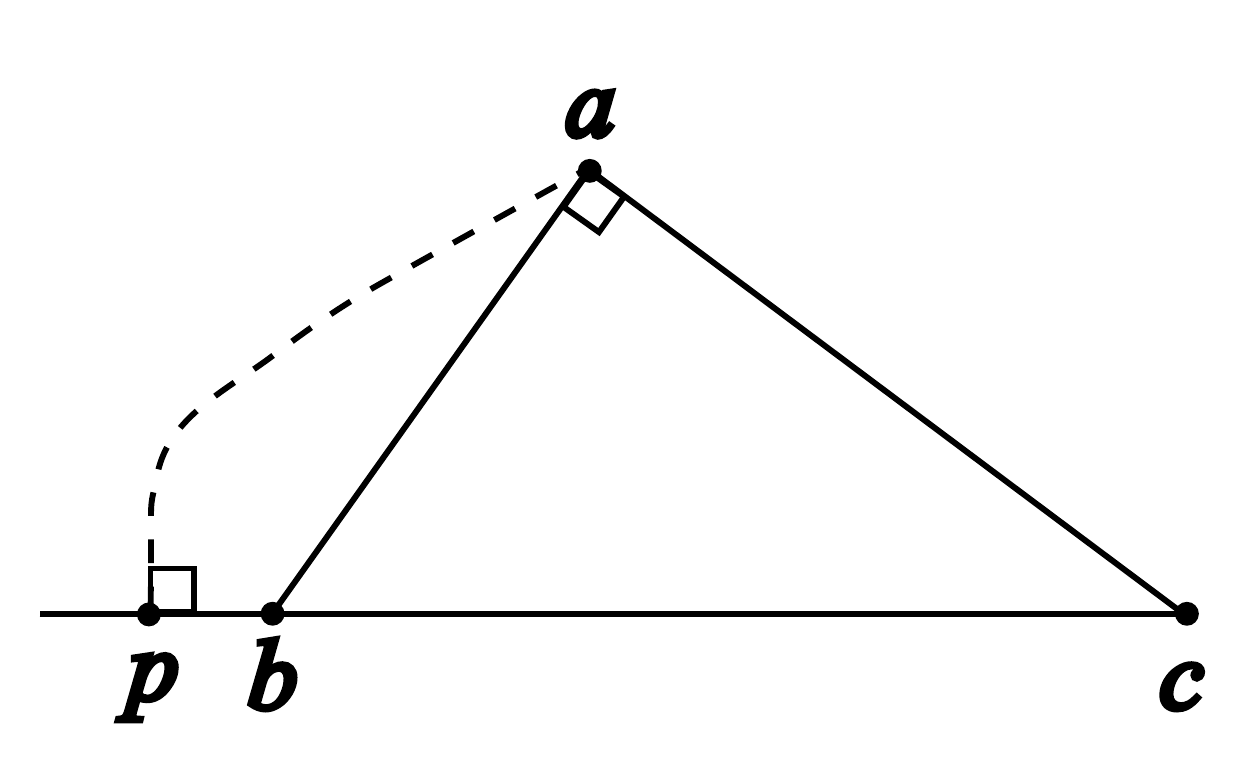}

\end{center}

Consider the triangle $\triangle\mathbf{apc}$. The interior angles
are $\angle\mathbf{apc}$, $\angle\mathbf{pac}$, and $\angle\mathbf{acp}$.

Note that $\angle\mathbf{pac}$ is the sum of $\angle\mathbf{pab}$
and $\angle\mathbf{bac}$, and $\angle\mathbf{bac}$ is a right angle.
$\angle\mathbf{apc}$ is also a right angle. 

In the sum of the interior angles of triangle $\triangle\mathbf{apc}$,
it includes two right angles ($\angle\mathbf{apc}$ and $\angle\mathbf{bac}$),
plus two extra angles ($\angle\mathbf{pab}$ and $\angle\mathbf{acp}$).
This contradicts Theorem \ref{thm:Interior-angle-sum}, which asserts
the sum of the interior angles of a triangle is equal to a straight
angle. Therefore, we must have $B(\mathbf{bpc})$.

Now consider Figure \ref{fig:Pythagorean-Thm} again.

The triangle $\Delta\mathbf{apb}$ is a right triangle. $\Delta\mathbf{bpa}$
and $\Delta\mathbf{bac}$ have two pairs of congruent angles: $\angle\mathbf{pba}$
is congruent to $\angle\mathbf{abc}$, and $\angle\mathbf{bpa}$ is
congruent to $\angle\mathbf{bac}$ (both are right angles). By Theorem
\ref{thm:Triangle-similarity-criterion-AA},
\[
d(\mathbf{bp})\cdot d(\mathbf{bc})=d(\mathbf{ab})\cdot d(\mathbf{ab}).
\]

\noindent Namely,
\begin{equation}
d^{2}(\mathbf{ab})=d(\mathbf{bp})\cdot d(\mathbf{bc}).\label{eq:Pythagoras-proof-1}
\end{equation}
In the similar vein, we can prove that
\begin{equation}
d^{2}(\mathbf{ac})=d(\mathbf{pc})\cdot d(\mathbf{bc}).\label{eq:Pythagoras-proof-2}
\end{equation}
Adding Equation (\ref{eq:Pythagoras-proof-1}) to Equation (\ref{eq:Pythagoras-proof-2}),
we obtain
\begin{align*}
d^{2}(\mathbf{ab})+d^{2}(\mathbf{ac}) & =d(\mathbf{bp})\cdot d(\mathbf{bc})+d(\mathbf{pc})\cdot d(\mathbf{bc})\\
 & =\left[d(\mathbf{bp})+d(\mathbf{pc})\right]\cdot d(\mathbf{bc})\\
 & =d(\mathbf{bc})\cdot d(\mathbf{bc})\\
 & =d^{2}(\mathbf{bc}).
\end{align*}
\end{proof}
The next theorem is the triangle inequality for the distance function.
\begin{thm}
\emph{\label{thm:(Triangle-inequality)}(Triangle inequality) 
\[
\forall\mathbf{a}\forall\mathbf{b}\forall\mathbf{c}\,\,\,d(\mathbf{ab})\leqslant d(\mathbf{ac})+d(\mathbf{bc}).
\]
}
\end{thm}

\begin{center}
\includegraphics[scale=0.2]{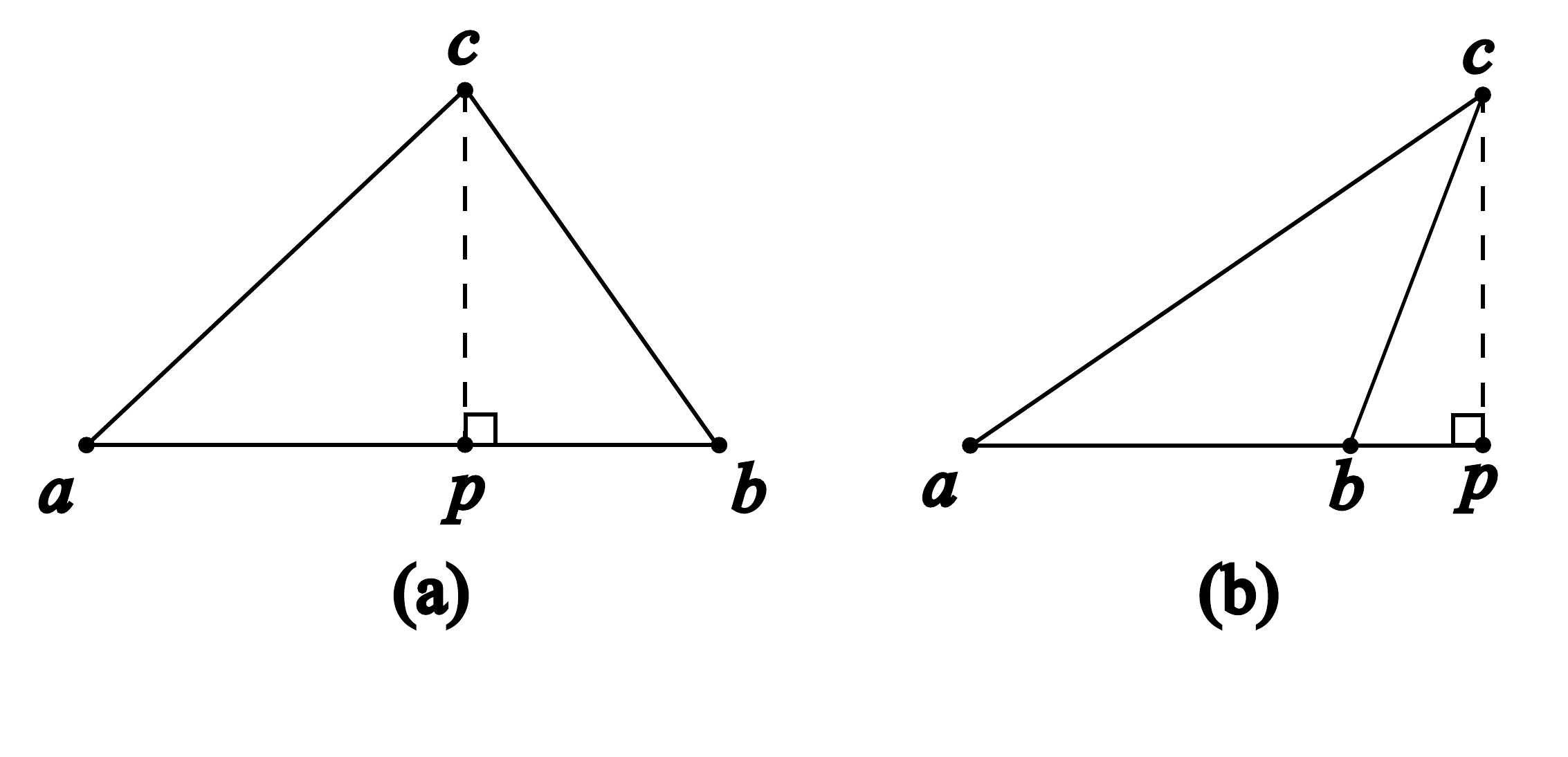}
\par\end{center}

\vspace{-15pt}
\begin{center}
\captionsetup{width=0.3\textwidth}
\captionof{figure}{Triangle inequality.} \label{fig:Triangle-inequality-1}
\end{center}
\vspace{-5pt}
\begin{proof}
From $\mathbf{c}$, drop a perpendicular to $\mathbf{ab}$ and let
the foot be $\mathbf{p}$. The triangle $\Delta\mathbf{apc}$ is a
right triangle. The point $\mathbf{p}$ may be between $\mathbf{a}$
and $\mathbf{b}$, or not between $\mathbf{a}$ and $\mathbf{b}$.
We discuss three cases:\medskip{}

\noindent Case (1). $B(\mathbf{apb})$ (Figure \ref{fig:Triangle-inequality-1}a)

By Theorem \ref{thm:(Pythagoras)}, in the right triangle $\Delta\mathbf{cpa}$,
we have
\[
d^{2}(\mathbf{ap})+d^{2}(\mathbf{pc})=d^{2}(\mathbf{ac}).
\]
Because $0\leqslant d^{2}(\mathbf{pc})$, we have 
\[
d^{2}(\mathbf{ap})\leqslant d^{2}(\mathbf{ac}).
\]
By Axiom D1, this implies
\begin{equation}
d(\mathbf{ap})\leqslant d(\mathbf{ac}).\label{eq:First-right-triangle}
\end{equation}
Similarly, if we consider the right triangle $\Delta\mathbf{cpb}$,
we obtain
\begin{equation}
d(\mathbf{pb})\leqslant d(\mathbf{bc}).\label{eq:Second-right-triangle}
\end{equation}
Adding Equation (\ref{eq:First-right-triangle}) to Equation (\ref{eq:Second-right-triangle}),
we obtain
\[
d(\mathbf{ap})+d(\mathbf{pb})\leqslant d(\mathbf{ac})+d(\mathbf{bc}).
\]
Because $B(\mathbf{apb})$, $d(\mathbf{ab})=d(\mathbf{ap})+d(\mathbf{pb})$.
Therefore, 
\[
d(\mathbf{ab})\leqslant d(\mathbf{ac})+d(\mathbf{bc}).
\]

\noindent Case (2). $B(\mathbf{abp})$ (Figure \ref{fig:Triangle-inequality-1}b)

In the right triangle $\Delta\mathbf{cpa}$, similar to Case (1),
we have
\begin{equation}
d(\mathbf{ap})\leqslant d(\mathbf{ac}).\label{eq:First-right-triangle-1}
\end{equation}

\noindent Because $B(\mathbf{abp})$, we have $d(\mathbf{ap})=d(\mathbf{ab})+d(\mathbf{bp})$.

\noindent By Axiom D1, $0\leqslant d(\mathbf{bp})$, we have $d(\mathbf{ab})\leqslant d(\mathbf{ap})$.
Therefore, $d(\mathbf{ab})\leqslant d(\mathbf{ac})$. This means the
single side $\mathbf{ac}$ is at least as long as $\mathbf{ab}$.
When we add a nonnegative length $d(\mathbf{bc})$, we have
\[
d(\mathbf{ab})\leqslant d(\mathbf{ac})+d(\mathbf{bc}).
\]

\noindent Case (3). $B(\mathbf{pab})$

The same proof for Case (2) works for this case too, by a symmetry
argument.\textbf{ }
\end{proof}
\begin{thm}
\emph{\label{thm:Five-segments}(Axiom of five segments, see Appendix
\ref{appendix:b})
\[
\begin{aligned} & \forall\mathbf{a}\forall\mathbf{b}\forall\mathbf{c}\forall\mathbf{d}\forall\mathbf{\mathbf{a}'}\forall\mathbf{\mathbf{b}'}\forall\mathbf{c}'\forall\mathbf{\mathbf{d}'}\\
 & [\mathbf{a}\neq\mathbf{b}\wedge B(\mathbf{abd})\wedge B(\mathbf{\mathbf{a}'\mathbf{b}'d'})\wedge d(\mathbf{ac})=d(\mathbf{a'c'})\wedge d(\mathbf{ab})=d(\mathbf{a'b'})\wedge d(\mathbf{bc})=d(\mathbf{b'c'})\\
 & \wedge d(\mathbf{bd})=d(\mathbf{b'd'})]\rightarrow d(\mathbf{cd})=d(\mathbf{c'd'}).
\end{aligned}
\]
}
\end{thm}

\vspace{-10pt}
\begin{center}
\includegraphics[scale=0.2]{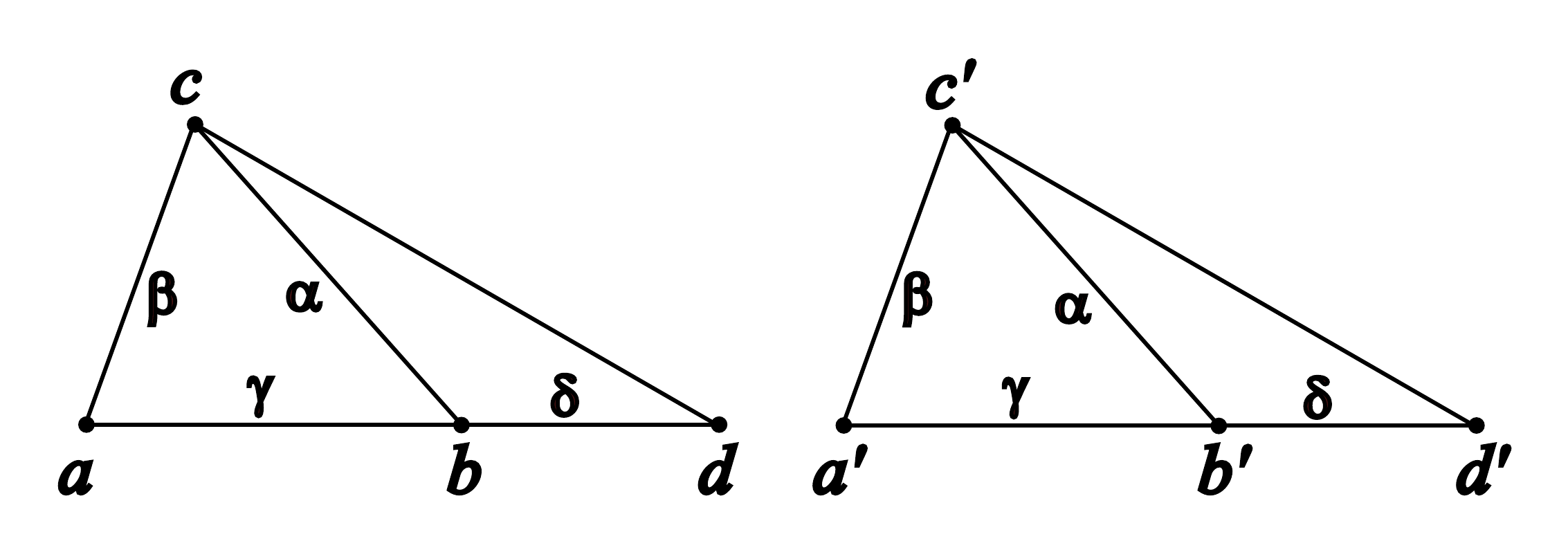}
\par\end{center}

\vspace{-15pt}
\begin{center}
\captionsetup{width=0.35\textwidth}
\captionof{figure}{Axiom of five segments.}
\end{center}
\vspace{-5pt}
\begin{proof}
Let $\alpha=d(\mathbf{bc})=d(\mathbf{b'c'})$, $\beta=d(\mathbf{ac})=d(\mathbf{a'c'})$,
$\gamma=d(\mathbf{ab})=d(\mathbf{a'b'})$, and $\delta=d(\mathbf{bd})=d(\mathbf{b'd'})$.

From $\mathbf{c}$, drop a perpendicular to $\mathbf{ab}$, and let
the foot be $\mathbf{p}$. From $\mathbf{c}'$, drop a perpendicular
to $\mathbf{a'b'}$, and let the foot be $\mathbf{p'}$ (see the figure
below).

Let $h=d(\mathbf{cp})$, $x=d(\mathbf{ap})$, and $y=d(\mathbf{pb})$.

\begin{center}

\includegraphics[scale=0.2]{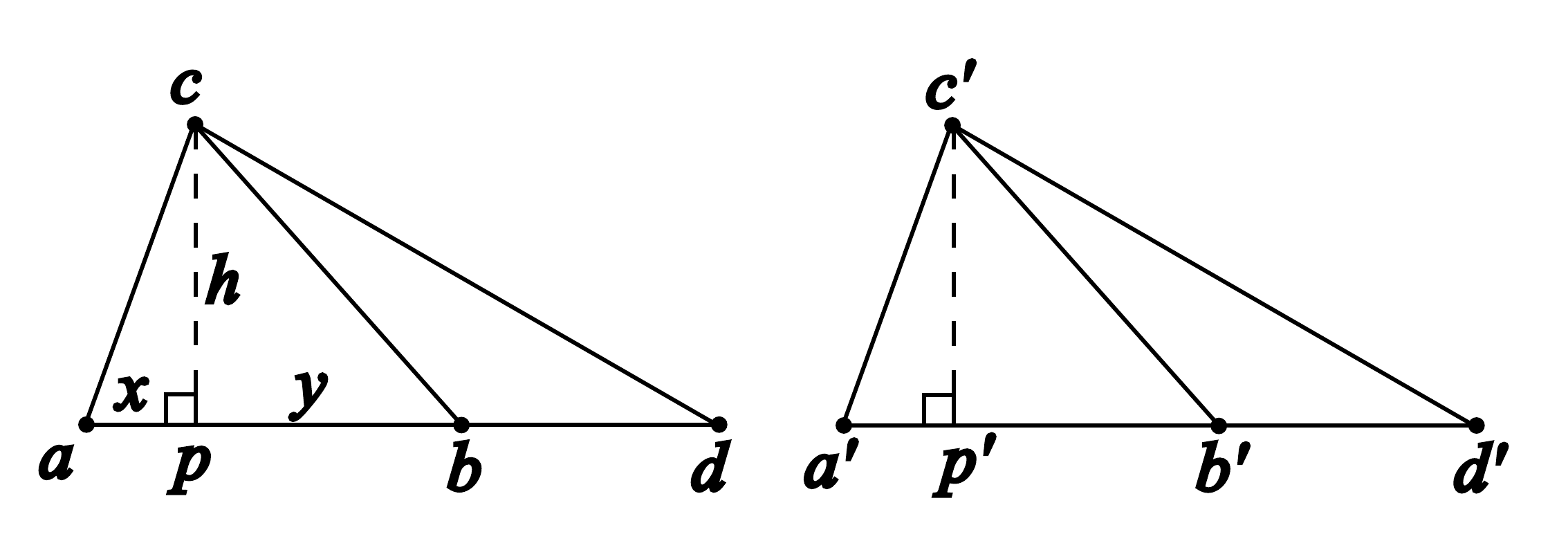}

\end{center}

\noindent Case (1): $B(\mathbf{apb})\wedge B(\mathbf{a'p'b'})$

First we want to show, because the two triangles are congruent, $\triangle\mathbf{abc}\equiv\triangle\mathbf{a'b'c'}$,
their corresponding altitudes are also congruent, namely $d(\mathbf{cp})=d(\mathbf{c'p'})$.

By the Pythagorean theorem,
\begin{align*}
x^{2}+h^{2} & =\beta^{2},\\
y^{2}+h^{2} & =\alpha^{2},\\
x+y & =\gamma.
\end{align*}
Solving these equations, we obtain
\[
h=\frac{1}{2\gamma}\sqrt{2\alpha^{2}\beta^{2}+2\beta^{2}\gamma^{2}+2\gamma^{2}\alpha^{2}-\alpha^{4}-\beta^{4}-\gamma^{4}}.
\]
In this expression, the altitude $h$ is uniquely determined by the
side lengths $\alpha,\beta,\gamma$ of $\triangle\mathbf{abc}$. We
therefore conclude that the congruent triangle $\triangle\mathbf{a'b'c'}$
must have the same altitude $h$.

Now consider the right triangle $\triangle\mathbf{cpd}$. By the Pythagorean
theorem,
\begin{align*}
d(\mathbf{cd}) & =\sqrt{h^{2}+\left(y+\delta\right)^{2}}\\
 & =\sqrt{\alpha^{2}+\delta^{2}+\left(\frac{\delta}{\gamma}\right)\left(\alpha^{2}+\gamma^{2}-\beta^{2}\right)}.
\end{align*}
Note that $d(\mathbf{cd})$ is uniquely determined by $\alpha,\beta,\gamma,\delta$,
the lengths of the four segments that have their corresponding congruent
counterparts in the other configuration with primed points. Therefore,
we must have $d(\mathbf{cd})=d(\mathbf{c'd'})$.\medskip{}

\noindent Case (2): If $\mathbf{p}$ is not between $\mathbf{a}$
and $\mathbf{b}$, the proof is just similar, and we will omit it.
The only difference is, the relationship $x+y=\gamma$ should be replaced
by $y-x=\gamma$ or $x-y=\gamma$. 
\end{proof}
\begin{thm}
\emph{(Axiom of Pasch, inner form, see Appendix \ref{appendix:b})\label{thm:Axiom-of-Pasch}}
\[
\forall\mathbf{a}\forall\mathbf{b}\forall\mathbf{c}\forall\mathbf{p}\forall\mathbf{q}\,\,\,\,[B(\mathbf{bcp})\wedge B(\mathbf{aqb})]\rightarrow\exists\mathbf{t}\,\,[B(\mathbf{ptq})\wedge B(\mathbf{atc})].
\]
\end{thm}

\begin{center}
\includegraphics[scale=0.25]{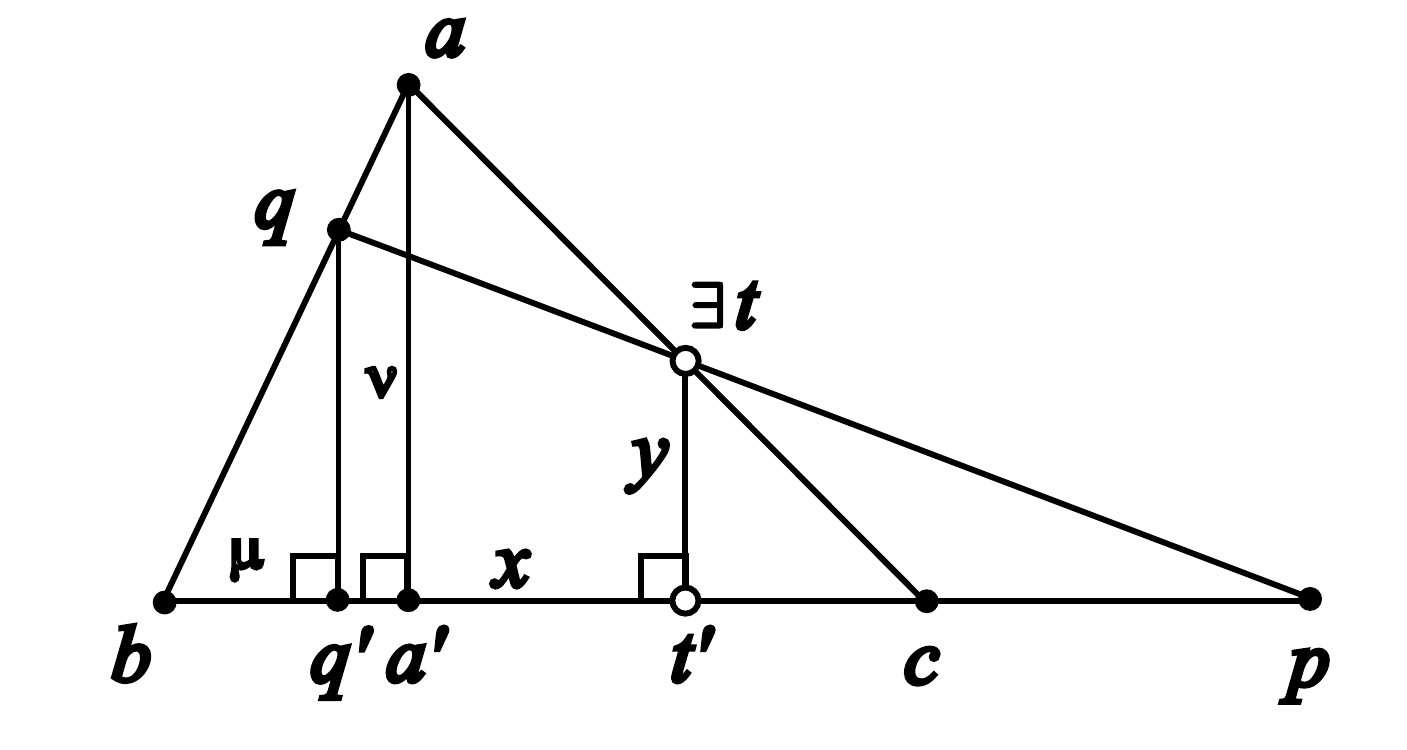}
\par\end{center}

\vspace{-15pt}
\begin{center}
\captionsetup{width=0.38\textwidth}
\captionof{figure}{Axiom of Pasch, inner form.}
\end{center}
\vspace{-5pt}
\begin{proof}
Let $d(\mathbf{bc})=\alpha$, $d(\mathbf{bp})=\beta$, and $d(\mathbf{qb})/d(\mathbf{ab})=k$.

From $\mathbf{a}$, drop a perpendicular to $\mathbf{bp}$. Let the
foot be $\mathbf{a'}$, $d(\mathbf{ba'})=\mu$ and $d(\mathbf{aa'})=\nu$.

\noindent Case (1): $\mathbf{a'}$ is between $\mathbf{a}$ and $\mathbf{c}$.

First, we construct a point $\mathbf{t}$ in two steps.

\noindent Step 1: On line $\mathbf{bc}$, find the point $\mathbf{t'}$
between $\mathbf{b}$ and $\mathbf{c}$ such that $d(\mathbf{bt'})=x$,
where
\[
\begin{aligned}x & =\frac{(1-k)\alpha\beta+k\mu(\beta-\alpha)}{\beta-k\alpha}.\end{aligned}
\]

It is easy to see $x\geqslant0$. By Axiom D4, for any number $x\geqslant0$,
such a point $\mathbf{t'}$ exists.

\noindent Step 2: From $\mathbf{t'}$, draw $\mathbf{t't}$ perpendicular
to $\mathbf{bc}$ on the side of $\mathbf{a}$ such that $d(\mathbf{t't})=y$,
where
\[
\begin{aligned}y & =\frac{k\nu(\beta-\alpha)}{\beta-k\alpha}.\end{aligned}
\]

It is easy to see $y\geqslant0$. By Axiom D4, for any number $y\geqslant0$,
such a point $\mathbf{t}$ exists.

We claim that $\mathbf{t}$ is a point that satisfies $B(\mathbf{ptq})$
and $B(\mathbf{atc})$.

\noindent Verification of $B(\mathbf{ptq})$:

By repeated applications of the Pythagorean theorem to the right triangles
$\triangle\mathbf{qq'b}$, $\triangle\mathbf{aa'b}$, $\triangle\mathbf{qq'p}$,
and $\triangle\mathbf{tt'p}$, we obtain
\begin{align*}
d(\mathbf{pt}) & =\frac{\beta-\alpha}{\beta-k\alpha}\sqrt{(\beta-k\mu)^{2}+k^{2}\nu^{2}},\\
d(\mathbf{tq}) & =\frac{(1-k)\alpha}{\beta-k\alpha}\sqrt{(\beta-k\mu)^{2}+k^{2}\nu^{2}},\\
d(\mathbf{pq}) & =\sqrt{(\beta-k\mu)^{2}+k^{2}\nu^{2}}.
\end{align*}

It is easy to verify,
\[
d(\mathbf{pt})+d(\mathbf{tq})=d(\mathbf{pq}).
\]

Therefore, $B(\mathbf{ptq})$.

\noindent Verification of $B(\mathbf{atc})$:

By repeated applications of the Pythagorean theorem to the right triangles
$\triangle\mathbf{qq'b}$, $\triangle\mathbf{aa'b}$, $\triangle\mathbf{aa'c}$,
and $\triangle\mathbf{tt'c}$, we obtain
\begin{align*}
d(\mathbf{at}) & =\frac{(1-k)\beta}{\beta-k\alpha}\sqrt{(\alpha-\mu)^{2}+\nu^{2}},\\
d(\mathbf{tc}) & =\frac{k(\beta-\alpha)}{\beta-k\alpha}\sqrt{(\alpha-\mu)^{2}+\nu^{2}},\\
d(\mathbf{ac}) & =\sqrt{(\alpha-\mu)^{2}+\nu^{2}}.
\end{align*}

It is easy to verify,
\[
d(\mathbf{at})+d(\mathbf{tc})=d(\mathbf{ac}).
\]

Therefore, $B(\mathbf{atc})$.

\noindent Case (2): $\mathbf{a'}$ is not between $\mathbf{a}$ and
$\mathbf{c}$.

The proof is similar to Case (1), and we will omit it.
\end{proof}

\section{Analytic Geometry United under Synthetic Geometry in $\mathscr{E}_{d}$}

The mere introduction of a distance function symbol into the language
does not by itself render $\mathscr{E}_{d}$ as analytic geometry.
Traditionally, synthetic geometry is based on axioms while analytic
geometry is based on algebra. In traditional analytic geometry, a
point is \emph{represented} by a pair of real numbers $(x,y)$. This
is effectively saying that a point is \emph{defined to be} a pair
of real numbers $(x,y)$. For two points $(x_{1},y_{1})$ and $(x_{2},y_{2})$,
the distance between them is \emph{defined to be} $d=\sqrt{(x_{2}-x_{1})^{2}+(y_{2}-y_{1})^{2}}$.
In this sense, analytic geometry is treated as a model of synthetic
geometry. In contrast, the distance function symbol $d$ in $\mathscr{E}_{d}$
is an undefined primitive notion and is governed by the axioms, rather
than being defined. That is why $\mathscr{E}_{d}$ remains synthetic
geometry. What $\mathscr{E}_{d}$ makes possible, however, is effectively
framing analytic geometry within the formal first-order language of
$\mathscr{E}_{d}$, thereby having synthetic and analytic geometry
unified within $\mathscr{E}_{d}$.

Consider the following example. In traditional analytic geometry, we say $\mathbf{p}$ has coordinates
$(p_{x},p_{y})$ and $\mathbf{q}$ has coordinates $(q_{x},q_{y})$.
Suppose $\mathbf{u}$ lies on the line passing through distinct points $\mathbf{p}$
and $\mathbf{q}$. If $\mathbf{u}$ has coordinates $(x,y)$, then
$x$ and $y$ satisfy the following equation (Figure \ref{fig:Analytic-geometry}):
\begin{equation}
y-q_{y}=\frac{p_{y}-q_{y}}{p_{x}-q_{x}}\left(x-q_{x}\right).\label{eq:Line-equation}
\end{equation}

\begin{center}
\includegraphics[scale=0.3]{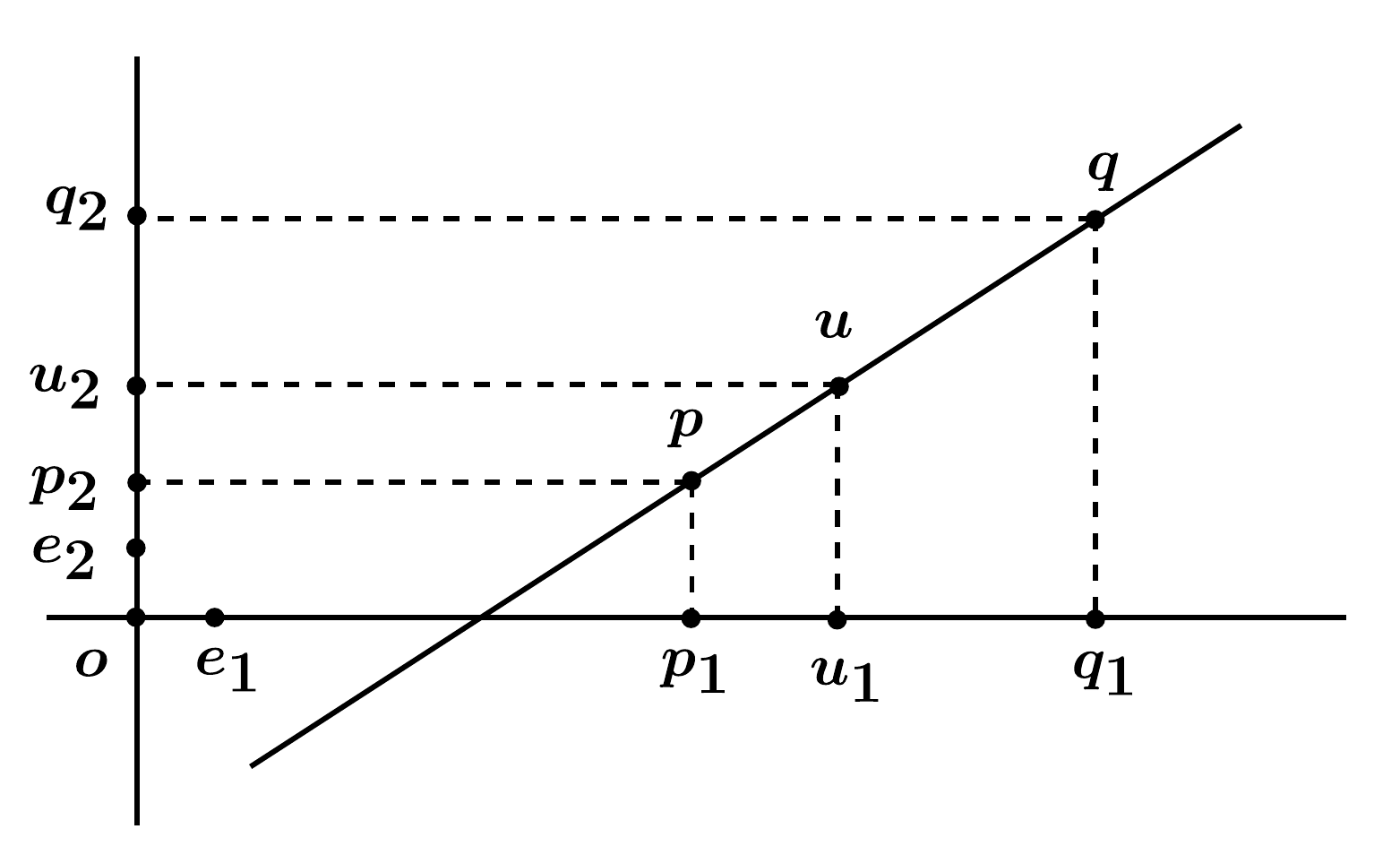}
\par\end{center}

\vspace{-10pt}
\begin{center}
\captionsetup{width=0.7\textwidth}
\captionof{figure}{Analytic geometry is united under synthetic geometry in $\mathscr{E}_d$.} \label{fig:Analytic-geometry}
\end{center}
\vspace{-5pt}

The coordinates $(x,y)$ are related to distances. The precise
relation depends on whether they are positive, where
\[
\left(B(\mathbf{o}\mathbf{e}_{1}\mathbf{u}_{1})\vee B(\mathbf{o}\mathbf{u}_{1}\mathbf{e}_{1})\right)\rightarrow x=d(\mathbf{ou}_{1}),
\]
or negative, where
\[
B(\mathbf{u}_{1}\mathbf{o}\mathbf{e}_{1})\rightarrow x=-d(\mathbf{ou}_{1}).
\]
The same applies to all other points as well.

In analytic geometry, one may say informally, ``Equation (\ref{eq:Line-equation})
represents the line (or\emph{ }it is the equation of the line) passing
through distinct points $\mathbf{p}$ and $\mathbf{q}$''. We are now
ready to express this statement as a formal sentence in $\mathscr{E}_{d}$:
\begin{align*}
 & \forall\mathbf{o}\forall\mathbf{e}_{1}\forall\mathbf{e}_{2}\forall\mathbf{u}\forall\mathbf{u}_{1}\forall\mathbf{u}_{2}\forall\mathbf{p}\forall\mathbf{p}_{1}\forall\mathbf{p}_{2}\forall\mathbf{q}\forall\mathbf{q}_{1}\forall\mathbf{q}_{2}\forall x\forall y\forall p_{x}\forall p_{y}\forall q_{x}\forall q_{y}\\
 & \mathbf{p}\neq\mathbf{q}\wedge R(\mathbf{e}_{1}\mathbf{o}\mathbf{e}_{2})\wedge L(\mathbf{o}\mathbf{e}_{1}\mathbf{p}_{1}\mathbf{q}_{1}\mathbf{u}_{1})\wedge L(\mathbf{o}\mathbf{e}_{2}\mathbf{p}_{2}\mathbf{q}_{2}\mathbf{u}_{2})\\
 & \wedge R(\mathbf{pp_{1}\mathbf{o}})\wedge R(\mathbf{pp_{2}\mathbf{o}})\wedge R(\mathbf{qq_{1}\mathbf{o}})\wedge R(\mathbf{qq_{2}\mathbf{o}})\wedge R(\mathbf{uu_{1}\mathbf{o}})\wedge R(\mathbf{uu_{2}\mathbf{o}})\\
 & \wedge\left[\left(B(\mathbf{o}\mathbf{e}_{1}\mathbf{u}_{1})\vee B(\mathbf{o}\mathbf{u}_{1}\mathbf{e}_{1})\right)\rightarrow x=d(\mathbf{ou}_{1})\right]\wedge\left[B(\mathbf{u}_{1}\mathbf{o}\mathbf{e}_{1})\rightarrow x=-d(\mathbf{ou}_{1})\right]\\
 & \wedge\left[\left(B(\mathbf{o}\mathbf{e}_{2}\mathbf{u}_{2})\vee B(\mathbf{o}\mathbf{u}_{2}\mathbf{e}_{2})\right)\rightarrow y=d(\mathbf{ou}_{2})\right]\wedge\left[B(\mathbf{u}_{2}\mathbf{o}\mathbf{e}_{2})\rightarrow y=-d(\mathbf{ou}_{2})\right]\\
 & \wedge\cdots\cdots\\
 & \rightarrow\left[L(\mathbf{upq})\leftrightarrow y-q_{y}=\frac{p_{y}-q_{y}}{p_{x}-q_{x}}\left(x-q_{x}\right)\right],
\end{align*}
where $L(\mathbf{o}\mathbf{e}_{1}\mathbf{p}_{1}\mathbf{q}_{1}\mathbf{u}_{1})$
is the abbreviation for that all these points being
collinear, and the dots indicate analogous conditions for $\mathbf{p}_{1},\mathbf{p}_{2}$
and $\mathbf{q}_{1},\mathbf{q}_{2}$.

\section{Theory $\mathscr{E}_{da}$---A Conservative Extension
of $\mathscr{E}_{d}$ with Angle Function \boldmath{$a$}}

The language of $\mathscr{E}_{d}$ has a single geometric primitive
function $d$, the distance function, apart from the operations for
numbers. We can extend theory $\mathscr{E}_{d}$ further to a theory
$\mathscr{E}_{da}$ to include the angle function $a$, which denotes
the numerical measure of an angle.

The language of $\mathscr{E}_{da}$ is
\[
L(\mathscr{E}_{da})=L(d,a;\,=,+,\cdot,<,0,1),
\]
where $a$ is a 3-place function, with $a(\mathbf{pvq})$ being the
number to represent the measure of angle $\angle\mathbf{pvq}$ formed
by points $\mathbf{p},\mathbf{v}$, and $\mathbf{q}$, with $\mathbf{v}$
as the vertex.

The system $\mathscr{E}_{da}$ must satisfy axioms RCF1 through RCF17
(Appendix \ref{appendix:a}), Axioms D1 through D7, plus A1 through
A4, which shall be listed in this section.

First, we explain what we mean by $\mathscr{E}_{da}$ being a conservative
extension of $\mathscr{E}_{d}$. 

Let $\Sigma$ be a theory in the language $L(\Sigma)$, $\Delta$
a theory in the language $L(\Delta)$, and $L(\Sigma)\subseteq L(\Delta)$.

$\Delta$ is called an extension of $\Sigma$ if $\Sigma\subseteq\Delta$.

$\Delta$ is called a conservative extension of $\Sigma$ if $\Delta$
is an extension of $\Sigma$, and for each $A$ in $L(\Sigma)$, $\Sigma\vdash A$
if and only if $\Delta\vdash A$, meaning the only new theorems in
$\Delta$ are those that use symbols in $L(\Delta)$ that are not
in $L(\Sigma)$ (see Epstein \cite{Epstein}).

In the following, we list the axioms of $\mathscr{E}_{da}$.

\noindent\rule[0.5ex]{1\columnwidth}{1pt}

\noindent$\mathscr{E}_{da}$\textbf{---}\textbf{\emph{The Theory
of Plane Quantitative Euclidean Geometry with Distance and Angle Functions
in the Language $L(d,a;\,=,+,\cdot,<,0,1)$}}

\noindent\rule[0.5ex]{1\columnwidth}{1pt}\medskip{}

\noindent\textbf{Axioms RCF1 through RCF17}\medskip{}

\noindent\textbf{Axioms D1 through D7}\medskip{}

\noindent\textbf{Axiom A1}. Nonnegativeness of angle measure
\[
\forall\mathbf{p}\forall\mathbf{v}\forall\mathbf{q}\,\,\,0\leqslant a(\mathbf{pvq}).
\]

\noindent\textbf{Axiom A2}. Congruent angles
\[
\forall\mathbf{p}\forall\mathbf{v}\forall\mathbf{q}\forall\mathbf{p'}\forall\mathbf{v'}\forall\mathbf{q'}\,\,\,[\angle\mathbf{pvq}\equiv\angle\mathbf{p'v'q'}\leftrightarrow a(\mathbf{pvq})=a(\mathbf{p'v'q'})].
\]

See Definition \ref{def:Angle-congruence} in Section \ref{sec:Theory-Ed}
for the abbreviation $\angle\mathbf{pvq}\equiv\angle\mathbf{p'v'q'}$.

\medskip{}

Informal explanation in plain English: Congruent angles have equal
angle measures and vice versa.

\medskip{}

\noindent\textbf{Axiom A3}. Addition of angles
\[
\forall\mathbf{p}\forall\mathbf{v}\forall\mathbf{q}\forall\mathbf{t}\,\,\,\angle\mathbf{pvt}\tilde{+}\angle\mathbf{tvq}\equiv\mathbf{pvq}\leftrightarrow a(\mathbf{pvt})+a(\mathbf{tvq})=a(\mathbf{pvq}).
\]

\begin{center}
\includegraphics[scale=0.22]{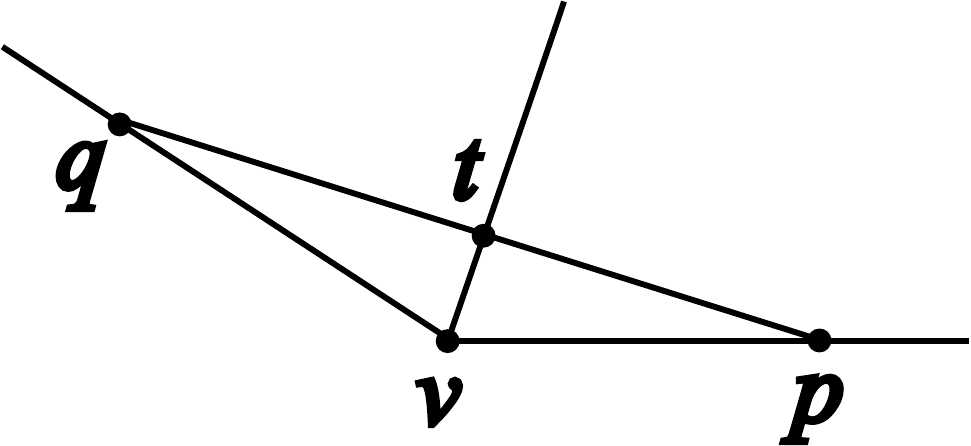}~~~~~~~~\includegraphics[scale=0.22]{Fig09-Def-Angle_addition2}
\par\end{center}

\vspace{-12pt}
\begin{center}
\captionsetup{width=0.32\textwidth}
\captionof{figure}{Addition of angles.} \label{axm:Addition-angles}
\end{center}
\vspace{-5pt}

See Definition \ref{def:Sum-angles} in Section \ref{sec:Theory-Ed}
for the abbreviation $\angle\mathbf{pvt}\tilde{+}\angle\mathbf{tvq}\equiv\mathbf{pvq}$.
\medskip{}

\noindent\textbf{Axiom A4}. Straight angles
\[
\forall\mathbf{p}\forall\mathbf{v}\forall\mathbf{q}\,\,\,[B(\mathbf{pvq})\wedge\mathbf{p}\neq\mathbf{v}\wedge\mathbf{q}\neq\mathbf{v}]\leftrightarrow a(\mathbf{pvq})=180,
\]

\noindent where 180 is the abbreviation for $1+1+\cdots+1$ (one hundred
eighty times), and $1$ is the multiplicative identity, which is a
constant symbol in the system of RCF (Appendix \ref{appendix:a}).\medskip{}

By Axioms A3 and A4, it is easy to see that the maximum value of any
angle measure is 180.
\begin{thm}
$\forall\mathbf{p}\forall\mathbf{v}\forall\mathbf{q}\,\,\,a(\mathbf{pvq})=a(\mathbf{qvp}).$
\end{thm}

Informally, this means the angles are undirected. As we can see, there
is no need to list this as an axiom. The symmetry of the angle measure
is a consequence of the symmetry of distance, $d(\mathbf{pq})=d(\mathbf{qp})$.
\begin{proof}
This is straightforward, because by definition, the angles $\angle\mathbf{pvq}$
and $\angle\mathbf{qvp}$ are congruent.
\end{proof}
\begin{MyLemma}\label{lem:Lemma1} $\forall\mathbf{v}\forall\mathbf{s}\,\,\,a(\mathbf{svs})=0.$

\end{MyLemma}
\begin{proof}
By Axiom A3, if $\neg B(\mathbf{pvq})\wedge B(\mathbf{ptq})$, then
$a(\mathbf{pvq})=a(\mathbf{pvt})+a(\mathbf{tvq})$. Now consider a
special case where $\mathbf{p}=\mathbf{s}$, $\mathbf{t}=\mathbf{s}$,
and $\mathbf{q}=\mathbf{s}$. We have
\[
a(\mathbf{svs})=a(\mathbf{svs})+a(\mathbf{svs}).
\]
Therefore, $a(\mathbf{svs})=0$.
\end{proof}
\begin{thm}
\label{thm:Angle-same-ray}$\forall\mathbf{p}\forall\mathbf{v}\forall\mathbf{q}\,\,\,B(\mathbf{vpq})\rightarrow a(\mathbf{pvq})=0.$
\end{thm}

Informally, if $\mathbf{p}$ and $\mathbf{q}$ are on the same ``ray'',
then the measure of $\angle\mathbf{pvq}$ is zero. This shows there
is no need to list this as an axiom.
\begin{proof}
Because $B(\mathbf{vpq})$, by definition, $\angle\mathbf{pvq}\equiv\angle\mathbf{pvp}$.
By Axiom A2 and Lemma \ref{lem:Lemma1}, $a(\mathbf{pvq})=a(\mathbf{pvp})=0$.
\end{proof}

Next, we shall prove the theorem of vertical angles in both $\mathscr{E}_{da}$
and $\mathscr{E}_{d}$ separately.
\begin{thm}
\emph{\label{thm:Vertical-angles}(Vertical angles) $\forall\mathbf{a}\forall\mathbf{b}\forall\mathbf{c}\forall\mathbf{b'}\forall\mathbf{c'}\,\,\,B(\mathbf{bab'})\wedge B(\mathbf{cac'})\rightarrow\angle\mathbf{bac}\equiv\angle\mathbf{b'ac'}.$}
\end{thm}

\begin{center}
\includegraphics[scale=0.22]{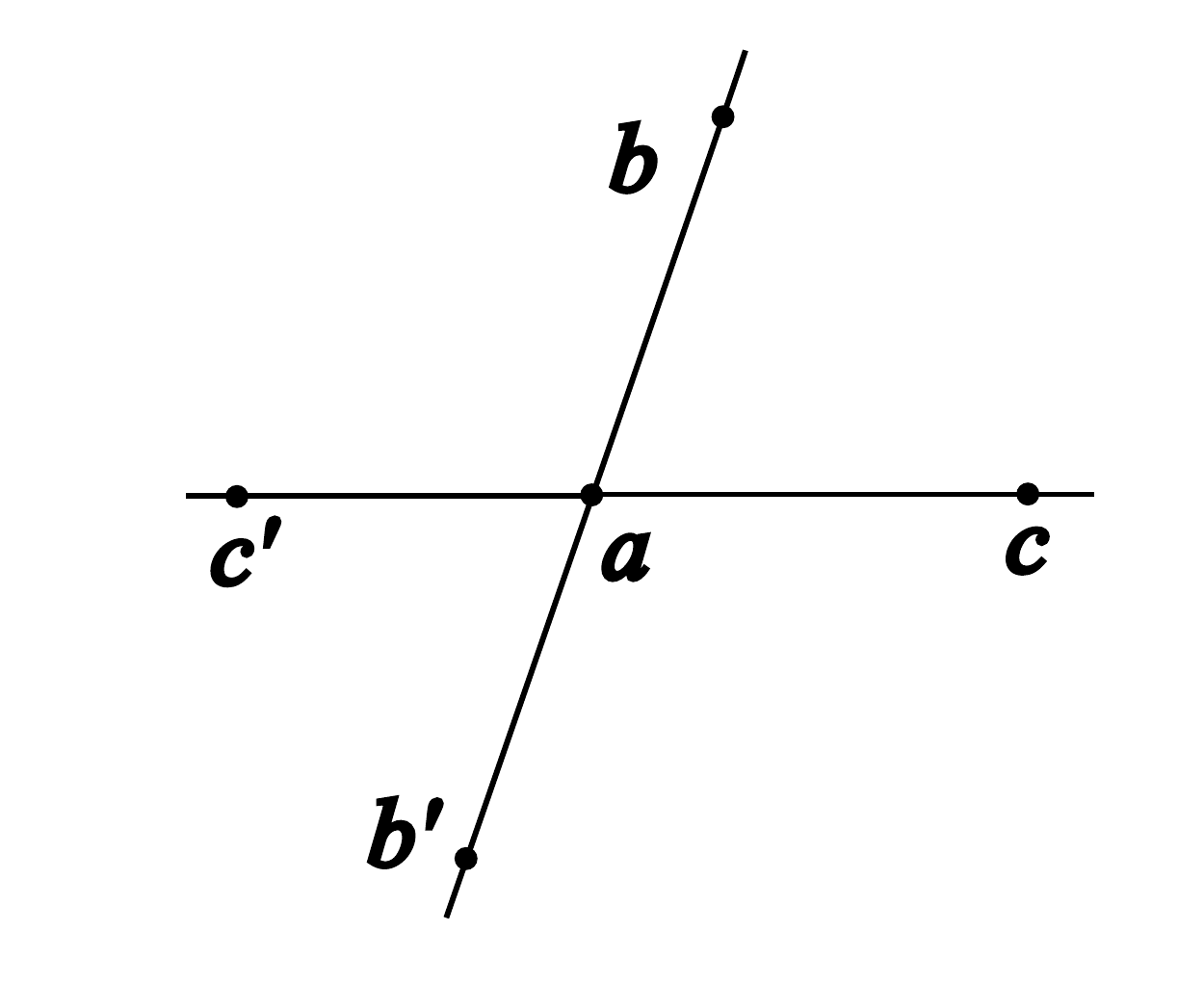}
\par\end{center}

\vspace{-12pt}
\begin{center}
\captionsetup{width=0.26\textwidth}
\captionof{figure}{Verticle angles.}
\end{center}
\vspace{-5pt}
\begin{proof}
(Proof in $\mathscr{E}_{da}$) Because $B(\mathbf{bab'})$, $\angle\mathbf{cab}$
and $\angle\mathbf{bac'}$ are supplementary angles. Therefore, $a(\mathbf{cab})+a(\mathbf{bac'})=180$.
Similarly, $\angle\mathbf{bac'}$ and $\angle\mathbf{c'ab'}$ are
supplementary angles. We have $a(\mathbf{bac'})+a(\mathbf{c'ab'})=180$.
Therefore, we have $a(\mathbf{cab})=a(\mathbf{c'ab'})$. This is the
same as $a(\mathbf{bac})=a(\mathbf{b'ac'})$. By Axiom A2, we have
$\angle\mathbf{bac}\equiv\angle\mathbf{b'ac'}$.
\end{proof}

This is Proposition I.15 in Euclid's \emph{Elements}. Euclid's proof
is basically the same as above, except that Euclid did not make the
measure of an angle explicit, and he simply cited his Common Notion
3: ``If equals be subtracted from equals, the remainders are equal.''
Note that we have defined the addition ($\tilde{+}$) of two adjacent
angles to be a third angle (Definition \ref{def:Sum-angles}), but
there is still a difference between an angle and the measure (as a
number) of the angle.

As we have commented, $\mathscr{E}_{da}$ is a conservative extension
of $\mathscr{E}_{d}$. The vocabulary of Theorem \ref{thm:Vertical-angles}
does not involve the angle function $a$. It should be able to be
proved in $\mathscr{E}_{d}$. An alternative proof strictly within
the theory $\mathscr{E}_{d}$ is shown in the following.

\medskip{}

\begin{proof}
(Proof within $\mathscr{E}_{d}$) By Theorem \ref{thm:Angle-same-ray},
the points on the same ``ray'' make the same (congruent) angle.
Instead of reasoning with points $\mathbf{b},\mathbf{b'},\mathbf{c},\mathbf{c'}$,
we mark points $\mathbf{p},\mathbf{p'},\mathbf{q},\mathbf{q'}$ such
that $d(\mathbf{ap})$, $d(\mathbf{ap'})$, $d(\mathbf{aq})$, and
$d(\mathbf{aq'})$ all have the same length. All we have to show is
that $\angle\mathbf{paq}\equiv\angle\mathbf{p'aq'}$. To show this,
it suffices to show that $d(\mathbf{pq})=d(\mathbf{p'q'})$, because
that would guarantee $\triangle\mathbf{paq}\equiv\triangle\mathbf{\mathbf{p'aq'}}$.

First we observe that $\triangle\mathbf{apq'}$ is an isosceles triangle.
Therefore, by Corollary \ref{thm:(Isosceles-triangle)}, we have $\angle\mathbf{apq'}\equiv\angle\mathbf{aq'p}$.

Consider $\triangle\mathbf{qpq'}$ and $\triangle\mathbf{p'q'p}$.
They have a common side $\mathbf{pq'}$. The sides $\mathbf{qq'}$
and $\mathbf{pp'}$ are congruent. The contained angles $\angle\mathbf{pq'q}$
and $\angle\mathbf{q'pp'}$ are also congruent. Therefore, $\triangle\mathbf{pq'q}\equiv\mathbf{\triangle\mathbf{q'pp'}}$
by Theorem \ref{thm:(SAS)} (SAS). We then have $d(\mathbf{pq})=d(\mathbf{p'q'})$.

Finally we have $\triangle\mathbf{paq}\equiv\triangle\mathbf{\mathbf{p'aq'}}$.
Hence, $\angle\mathbf{paq}\equiv\angle\mathbf{p'aq'}$. This is the
same as $\angle\mathbf{bac}\equiv\angle\mathbf{b'ac'}$.

\begin{center}

\includegraphics[scale=0.22]{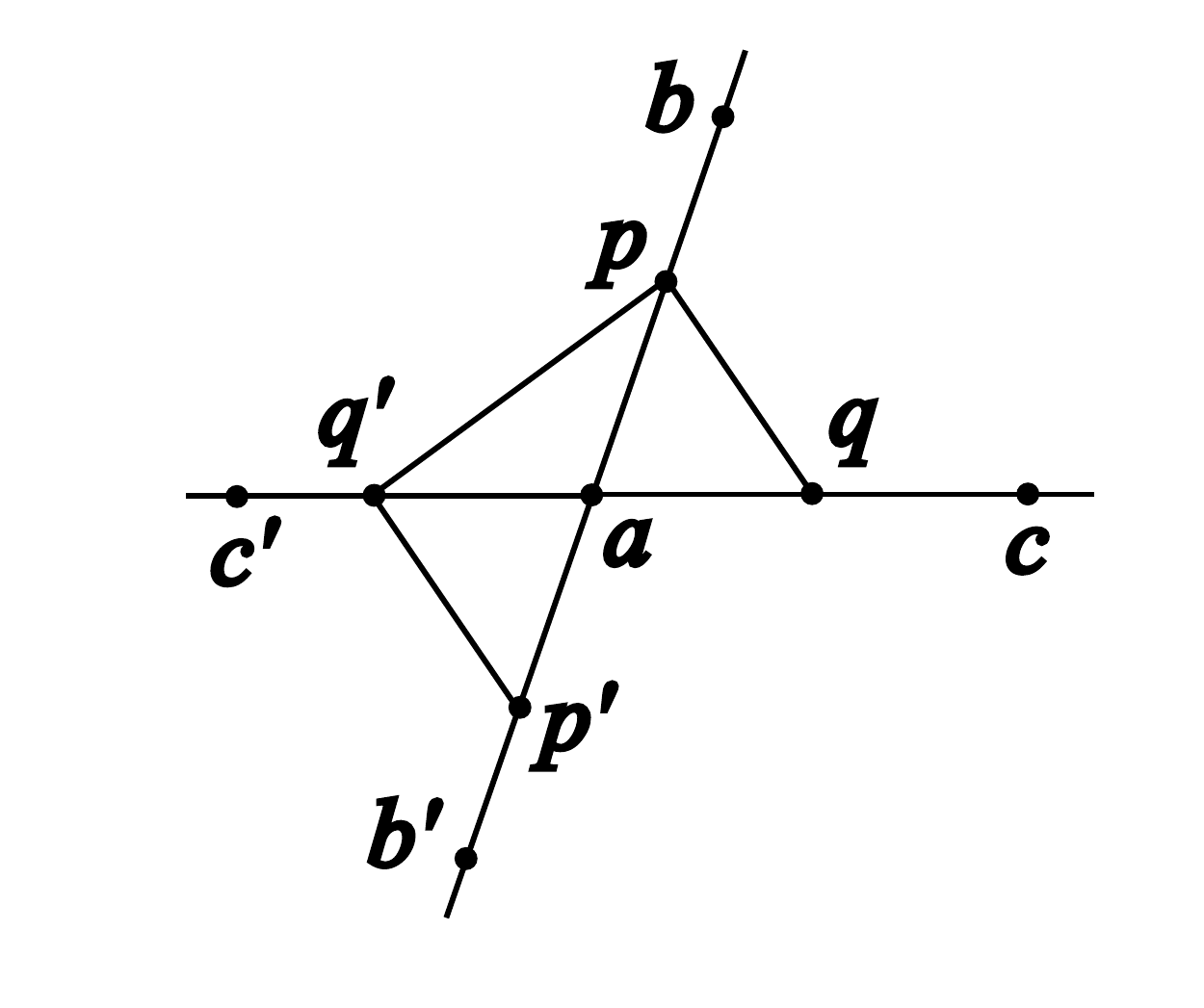}

\end{center}
\end{proof}

\section{Consistency of $\mathscr{E}_{d}$ and $\mathscr{E}_{da}$\label{sec:Consistency}}

To show the consistency of $\mathscr{E}_{d}$, we construct a model
that interprets each of the terms in the language
\[
L(\mathscr{E}_{d})=L(d;\,=,+,\cdot,<,0,1).
\]
The numbers in $L(\mathscr{E}_{d})$ are interpreted as the real numbers,
$+,\cdot$ as addition and multiplication of real numbers, $<$ the
``less than'' relation, $0$ the number $0$, and $1$ the number
$1$.

The real numbers satisfy all the axioms of real closed fields RCF1
through RCF17.

A point is interpreted as a pair of real numbers $(x,y)$.

Let $\mathbf{p}=(x_{1},y_{1})$ and $\mathbf{q}=(x_{2},y_{2})$ be
two points. The distance $d(\mathbf{pq})$ is interpreted as
\[
d(\mathbf{pq})=\sqrt{(x_{2}-x_{1})^{2}+(y_{2}-y_{1})^{2}}.
\]

\noindent Verification of Axioms D1, D2, and D3 is straightforward.\medskip{}

Let $\mathbf{t}=(x_{3},y_{3})$. $B(\mathbf{pqt})$ is interpreted
as $\exists\lambda\in\mathbb{R}$ such that
\begin{align}
x_{3} & =x_{1}+\lambda(x_{2}-x_{1}),\nonumber \\
y_{3} & =y_{1}+\lambda(y_{2}-y_{1}),\label{eq:Linear-extention-1}\\
1 & \leqslant\lambda.\nonumber 
\end{align}
The distance between $\mathbf{q}$ and $\mathbf{t}$ is
\[
d(\mathbf{pt})=(\lambda-1)\sqrt{(x_{2}-x_{1})^{2}+(y_{2}-y_{1})^{2}}.
\]

\medskip{}
\noindent Verification of Axiom D4:

We discuss two cases.\medskip{}

\noindent Case (1): $\mathbf{p}\neq\mathbf{q}$. For any real number
$0\leqslant x$, we can choose
\[
\lambda=1+\frac{x}{\sqrt{(x_{2}-x_{1})^{2}+(y_{2}-y_{1})^{2}}}.
\]

\noindent Substitute this into Equation (\ref{eq:Linear-extention-1}).
The point $\mathbf{t}$ with coordinates $(x_{3},y_{3})$ defined
by \mbox{Equation (\ref{eq:Linear-extention-1})} satisfies $B(\mathbf{pqt})$
and $d(\mathbf{qt})=x$.\medskip{}

\noindent Case (2): $\mathbf{p}=\mathbf{q}$, namely $x_{1}=x_{2}$
and $y_{1}=y_{2}$. It suffices to find $\mathbf{t}=(x_{3},y_{3})$
such that $d(\mathbf{qt})=x$. Such a point $\mathbf{t}$ is not unique,
but it suffices if we choose

\noindent
\begin{align*}
x_{3} & =x_{2}+x,\\
y_{3} & =y_{2}.
\end{align*}

This satisfies $B(\mathbf{pqt})$ and $d(\mathbf{qt})=x$.\medskip{}

\noindent Verification of Axioms D5 through D7 requires just straightforward
calculations in analytic geometry, which we shall omit here.

\medskip{}

To show the consistency of $\mathscr{E}_{da}$, first observe that
the language $L(\mathscr{E}_{d})$ is a subset of the language $L(\mathscr{E}_{da})$.
Recall that
\[
L(\mathscr{E}_{da})=L(d,a;\,=,+,\cdot,<,0,1).
\]
The language $L(\mathscr{E}_{da})$ has an additional primitive function,
the angle function $a$. The theory $\mathscr{E}_{d}$ is a subtheory
of the theory $\mathscr{E}_{da}$. Therefore, the model we have constructed
for $\mathscr{E}_{d}$ can serve as part of a model for $\mathscr{E}_{da}$.
We only need to give an interpretation of the angle function $a$,
and verify Axioms A1 through A4.

Let $\mathbf{v}=(x_{0},y_{0})$, $\mathbf{p}=(x_{1},y_{1})$, $\mathbf{q}=(x_{2},y_{2})$,
$\mathbf{p}\neq\mathbf{v}$, and $\mathbf{q}\neq\mathbf{v}$, as shown
in Figure \ref{fig:Interpretation-angle-measure}.
\begin{center}
\includegraphics[scale=0.22]{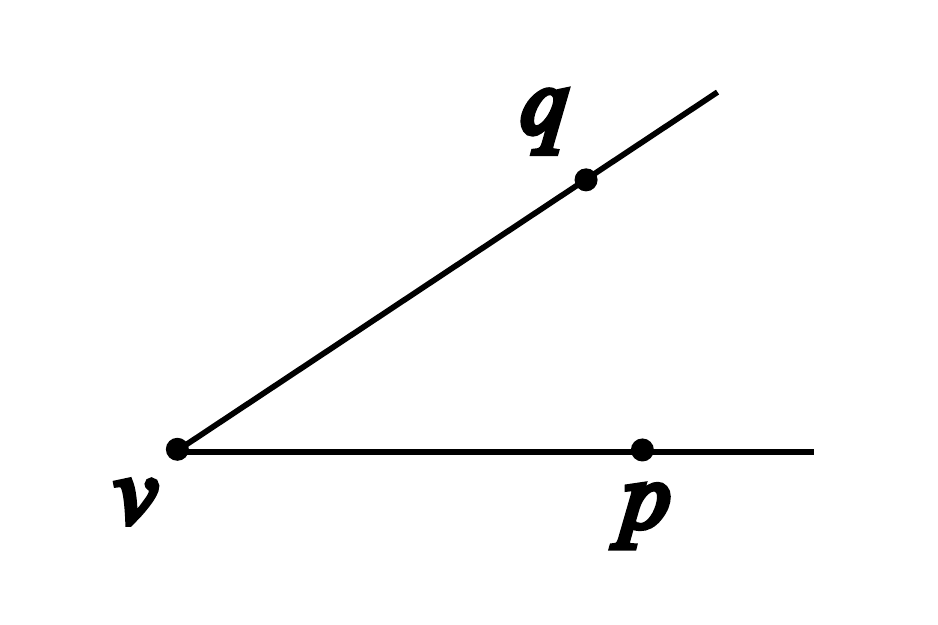}
\par\end{center}

\vspace{-12pt}
\begin{center}
\captionsetup{width=0.46\textwidth}
\captionof{figure}{Interpretation of the angle measure.} \label{fig:Interpretation-angle-measure}
\end{center}
\vspace{-5pt}

We first define
\begin{align*}
\overline{x}_{1} & =x_{1}-x_{0},\\
\overline{y}_{1} & =y_{1}-y_{0},\\
\overline{x}_{2} & =x_{2}-x_{0},\\
\overline{y}_{2} & =y_{2}-y_{0}.
\end{align*}

The angle measure of $\angle\mathbf{pvq}$ is interpreted as
\begin{equation}
a(\mathbf{pvq})=\frac{180}{\pi}\cdot\cos^{-1}\left[\frac{\overline{x}_{1}\cdot\overline{x}_{2}+\overline{y}_{1}\cdot\overline{y}_{2}}{\sqrt{\left(\overline{x}_{1}^{2}+\overline{y}_{1}^{2}\right)\left(\overline{x}_{2}^{2}+\overline{y}_{2}^{2}\right)}}\right].\label{eq:Def-of-Angle-in-Model}
\end{equation}

Verification of Axioms A1 through A4 requires just straightforward
analytic geometry calculations, which are omitted here.

\section{Theory $\mathscr{E}_{d}$ and Tarski's $\mathscr{E}_{2}$ Are Mutually
Interpretable with Parameters \label{sec:Tarski's-Interpretable}}

\medskip{}

\subsection{There Is a Syntactical Translation from $\mathscr{E}_{2}$ to $\mathscr{E}_{d}$\medskip{}
}

Tarski's $\mathscr{E}_{2}$ is parameter-free interpretable in $\mathscr{E}_{d}$.
This interpretation is faithful. In addition, we can make a stronger
assertion than a faithful interpretation: there is a syntactical translation
from $\mathscr{E}_{2}$ to $\mathscr{E}_{d}$ (see definitions of
interpretation and syntactical translation in Enderton \cite{Enderton}).

The syntactical translation scheme is as follows:

Any variable in $\mathscr{E}_{2}$ is translated to a point variable
in $\mathscr{E}_{d}$.

Tarski's primitive predicates in $\mathscr{E}_{2}$ are translated
to $\mathscr{E}_{d}$ as follows:

\begin{enumerate}

\item[(1)]$B(\mathbf{pvq})$ is translated to:

\noindent
\begin{equation}
d(\mathbf{pq})=d(\mathbf{pv})+d(\mathbf{vq}).\label{eq:Def-of-B}
\end{equation}

\noindent\item[(2)]$D(\mathbf{pquv})$ is translated to:
\begin{equation}
d(\mathbf{pq})=d(\mathbf{uv}).\label{eq:Def-of-D}
\end{equation}

\end{enumerate}

Tarski's axioms T1 through T11 for $\mathscr{E}_{2}$ are listed
in Appendix \ref{appendix:b}.

T5 (Axiom of five segments) is proved as Theorem \ref{thm:Five-segments},
T7 (Axiom of Pasch) as \mbox{Theorem \ref{thm:Axiom-of-Pasch}} in Section
\ref{sec:Theory-Ed}. Other axioms are straightforward to verify in
$\mathscr{E}_{d}$, using the above syntactical translation, and Axioms
D1 through D7 plus RCF1 through RCF17.

\subsection{Theory $\mathscr{E}_{d}$ Is Interpretable (with Parameters) into
$\mathscr{E}_{2}$\medskip{}
}

To interpret $\mathscr{E}_{d}$ in $\mathscr{E}_{2}$, we devise a
mapping $\tau$ from $\mathscr{E}_{d}$ to $\mathscr{E}_{2}$.

Any point variable in $\mathscr{E}_{d}$ remains a point variable
in $\mathscr{E}_{2}$. 

We also need to find a way to interpret the number variables of $\mathscr{E}_{d}$
in $\mathscr{E}_{2}$, because $\mathscr{E}_{2}$ has only point variables
but no number variables. This can be done with the help of the 
segment arithmetic with respect to an arbitrarily selected frame of
three points $\mathbf{\mathbf{\hat{o}}}$, $\mathbf{\mathbf{\hat{e}}}$,
and $\mathbf{e'}$ which are non-collinear.

We first select any three non-collinear points $\mathbf{\mathbf{\hat{o}}}$,
$\mathbf{\mathbf{\hat{e}}}$, and $\mathbf{e'}$.

Our sub-universe in $\mathscr{E}_{2}$ is all the points on the line
of $\mathbf{\mathbf{\hat{o}}\mathbf{\mathbf{\hat{e}}}}$, which are
used to interpret the numbers of RCF. The constant $0$ is mapped
to $\mathbf{\mathbf{\hat{o}}}$, and the constant $1$ is mapped to
$\mathbf{\mathbf{\hat{e}}}$. We will call the line $\mathbf{\mathbf{\hat{o}}\mathbf{\mathbf{\hat{e}}}}$
the\emph{ number line}.

We will use the following convention in notation: a hat is used to
denote a point on the number line. For example, the symbol $\mathbf{\hat{x}}$
automatically implies that $L(\mathbf{\hat{o}}\mathbf{\mathbf{\hat{e}}}\mathbf{\hat{x}})$
(see \mbox{Definition \ref{def:Collinear}).}

Figure \ref{def:Sum-Schwabhauser} shows SST's definition \cite{Schwabhauser}
that point $\mathbf{\hat{c}}$ is the\emph{ geometric sum }of points
$\mathbf{\hat{a}}$ and $\mathbf{\hat{b}}$. We will omit the formal
sentence, but just describe it informally:

Join $\mathbf{\hat{e}}$ and $\mathbf{e'}$ with a line; 

Draw $\mathbf{\hat{a}b'}\parallel\mathbf{\hat{e}}\mathbf{e'}$, intersecting
$\mathbf{\hat{o}}\mathbf{e'}$ at $\mathbf{b'}$ ($\parallel$ denotes
parallel);

Draw $\mathbf{b'c'\parallel\mathbf{\hat{o}}\mathbf{\hat{e}}}$;

Draw $\mathbf{\hat{b}c'\parallel\mathbf{\hat{o}}e'}$, intersecting
$\mathbf{b'c'}$ at $\mathbf{c'}$;

Draw $\mathbf{c'\hat{c}}\parallel\mathbf{\hat{a}b'}$, intersecting
$\mathbf{\hat{o}}\mathbf{\hat{e}}$ at $\mathbf{\hat{c}}$. 

We then define $\mathbf{\hat{c}}$ as the \emph{geometric sum} of
$\mathbf{\hat{a}}$ and $\mathbf{\hat{b}}$. In symbols, we write
\[
\mathbf{\hat{c}}=\mathbf{\hat{a}}\hat{+}\mathbf{\hat{b}}.
\]

\begin{center}
\includegraphics[scale=0.17]{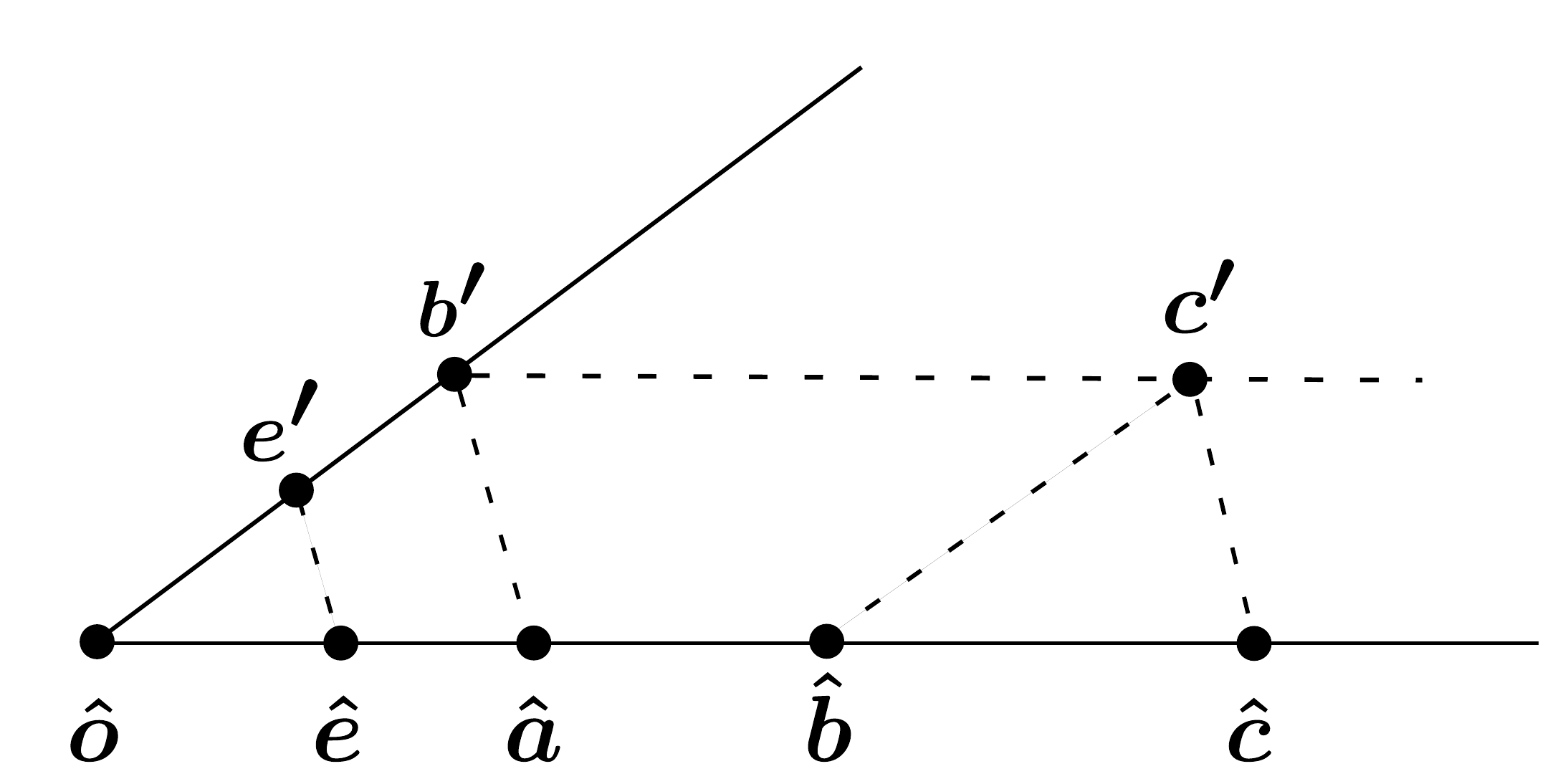}
\par\end{center}

\begin{center}
\captionsetup{width=0.84\textwidth}
\captionof{figure}{The definition that point $\mathbf{\hat{c}}$ is the geometric sum of points $\mathbf{\hat{a}}$ and $\mathbf{\hat{b}}$.}
~

\label{def:Sum-Schwabhauser}
\vspace{-5pt}
\par\end{center}

Figure \ref{def:Product-Schwabhauser} shows SST's definition \cite{Schwabhauser}
that point $\mathbf{\hat{c}}$ is the\emph{ geometric product }of
points $\mathbf{\hat{a}}$ and $\mathbf{\hat{b}}$. Informally:

Join $\mathbf{\hat{e}}$ and $\mathbf{e'}$ with a line;

Join $\mathbf{\hat{a}}$ and $\mathbf{e'}$ with a line;

Draw $\mathbf{\hat{b}b'}\parallel\mathbf{\hat{e}}\mathbf{e'}$, intersecting
$\mathbf{\hat{o}}\mathbf{e'}$ at $\mathbf{b'}$;

Draw $\mathbf{b'\hat{c}}\parallel\mathbf{\hat{a}e'}$, intersecting
$\mathbf{\hat{o}}\mathbf{\hat{e}}$ at $\mathbf{\hat{c}}$. 

We then define $\mathbf{\hat{c}}$ as the \emph{geometric product}
of $\mathbf{\hat{a}}$ and $\mathbf{\hat{b}}$. In symbols, we write
\[
\mathbf{\hat{c}}=\mathbf{\hat{a}}\hat{\cdot}\mathbf{\hat{b}}.
\]

\begin{center}
\includegraphics[scale=0.17]{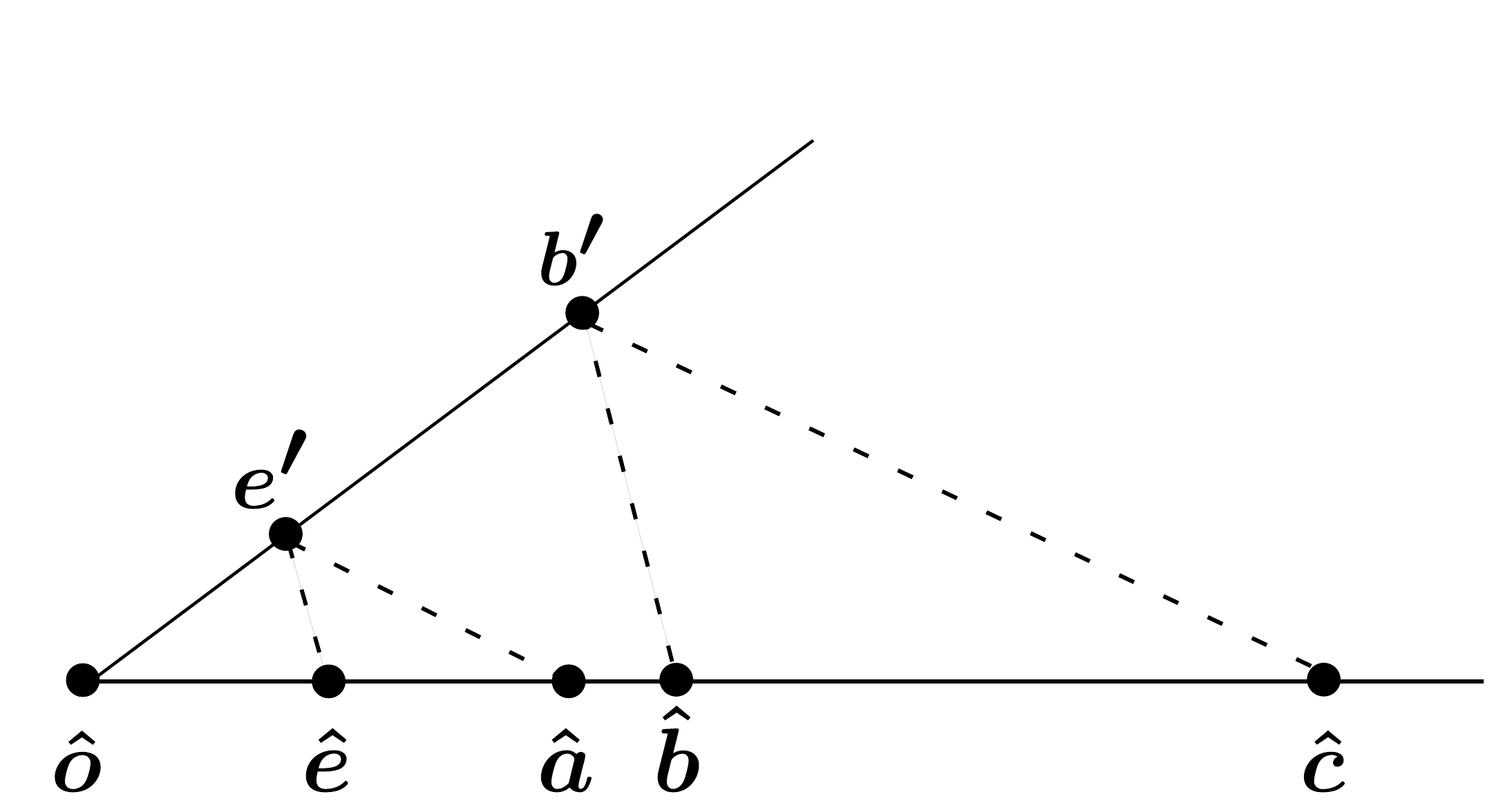}
\par\end{center}

\begin{center}
\captionsetup{width=0.84\textwidth}
\captionof{figure}{The definition that point $\mathbf{\hat{c}}$ is the geometric product of points $\mathbf{\hat{a}}$ and $\mathbf{\hat{b}}$.}
~

\label{def:Product-Schwabhauser}
\vspace{-5pt}
\par\end{center}

It should be noted that RCF is a subtheory of $\mathscr{E}_{d}$.
We first describe how to interpret sentences in RCF into $\mathscr{E}_{2}$,
namely those sentences in $\mathscr{E}_{d}$ that involve number variables
only.

Each number variable $x$ is mapped to a point variable $\hat{\mathbf{x}}$
for a point on the number line in $\mathscr{E}_{2}$. For example,
the \emph{wff} (\emph{well-formed formula}) $z=x+y$ in RCF (also
in $\mathscr{E}_{d}$) will be interpreted as $\hat{\mathbf{z}}=\hat{\mathbf{x}}\hat{+}\hat{\mathbf{y}}$
in $\mathscr{E}_{2}$. The \emph{wff} $z=x\cdot y$ in RCF (also in
$\mathscr{E}_{d}$) will be interpreted as $\hat{\mathbf{z}}=\hat{\mathbf{x}}\hat{\cdot}\hat{\mathbf{y}}$
in $\mathscr{E}_{2}$.

In the language $L(\mathscr{E}_{d})$, there are only three atomic
predicate symbols:

\begin{enumerate}

\item[(1)] $=$ as a 2--place predicate denoting that two points
are identical.

\item[(2)] $=$ as a 2-place predicate denoting that two numbers
are equal.

\item[(3)] $<$ as a 2-place predicate denoting that one number is
less than the other.

\end{enumerate}

First, consider the case where $=$ is a predicate for points, for
example, $\mathbf{p}=\mathbf{q}$. This is already a valid \emph{wff
}in $\mathscr{E}_{2}$ and it can be kept as is.

Next, consider the case where $=$ is a predicate for numbers. We
have discussed a special case when both sides of the equation involve
number constants and variables only. In general, we also need to consider
the function symbol $d(\mathbf{pq})$. The following is an example
of a \emph{wff} in $\mathscr{E}_{d}$:
\begin{equation}
3\cdot\left[d(\mathbf{pq})\cdot d(\mathbf{uv})\right]+\left[(x+2)\cdot y\right]\cdot y+1=y+d(\mathbf{pq}),\label{eq:Atomic-formula}
\end{equation}
where $\mathbf{p},\mathbf{q},\mathbf{u},\mathbf{v}$ are point variables,
and $x,y$ are number variables.

Note that $2$ is the abbreviation of $1+1$, and $3$ is the abbreviation
of $1+1+1$. 

Therefore, $1$ is interpreted as $\mathbf{\hat{e}}$, $2$ is interpreted
as $\mathbf{\hat{e}}\hat{+}\mathbf{\hat{e}}$, and $3$ is interpreted
as $\mathbf{\hat{e}}\hat{+}\mathbf{\hat{e}}\mathbf{\hat{+}\hat{e}}$.

The entire \emph{wff} can be interpreted as
\begin{equation}
\begin{aligned} & \exists\mathbf{\hat{o}}\exists\mathbf{\hat{e}}\exists\hat{\mathbf{x}}\exists\hat{\mathbf{y}}\exists\hat{\mathbf{a}}\exists\hat{\mathbf{b}}\\
 & \wedge D\left(\mathbf{pq}\mathbf{\hat{o}}\hat{\mathbf{a}}\right)\wedge D\left(\mathbf{uv}\mathbf{\hat{o}}\hat{\mathbf{b}}\right)\wedge\left(\mathbf{\hat{e}}\hat{+}\mathbf{\hat{e}}\mathbf{\hat{+}\hat{e}}\right)\hat{\cdot}\left[\hat{\mathbf{a}}\hat{\cdot}\hat{\mathbf{b}}\right]\hat{+}\left[\left(\hat{\mathbf{x}}\hat{+}\mathbf{\hat{e}}\mathbf{\hat{+}\hat{e}}\right)\hat{\cdot}\hat{\mathbf{y}}\right]\hat{\cdot}\hat{\mathbf{y}}\hat{+}\mathbf{\hat{e}}=\hat{\mathbf{y}}\hat{+}\hat{\mathbf{a}}.
\end{aligned}
\label{eq:Interpreted-Sentence}
\end{equation}

In the language of $L(\mathscr{E}_{d})$, the Pythagorean theorem
(Theorem \ref{thm:(Pythagoras)}) asserts
\[
\forall\mathbf{a}\forall\mathbf{b}\forall\mathbf{c}\,\,\,R(\mathbf{bac})\rightarrow d^{2}(\mathbf{ab})+d^{2}(\mathbf{ac})=d^{2}(\mathbf{bc}).
\]
This can be interpreted in $L(\mathscr{E}_{2})$ as
\[
\begin{aligned}\forall\mathbf{a}\forall\mathbf{b}\forall\mathbf{c}\exists\mathbf{\hat{o}}\exists\mathbf{\hat{e}}\exists\hat{\mathbf{x}}\exists\hat{\mathbf{y}}\exists\hat{\mathbf{z}}\,\, & \mathbf{a}\neq\mathbf{b}\wedge\mathbf{a}\ne\mathbf{c}\wedge\mathbf{b}\ne\mathbf{c}\wedge R(\mathbf{bac})\\
 & \wedge D\left(\mathbf{ab}\mathbf{\hat{o}}\hat{\mathbf{x}}\right)\wedge D\left(\mathbf{ac}\mathbf{\hat{o}}\hat{\mathbf{y}}\right)\wedge D\left(\mathbf{bc}\mathbf{\hat{o}}\hat{\mathbf{z}}\right)\\
 & \rightarrow\hat{\mathbf{x}}\hat{\cdot}\hat{\mathbf{x}}\hat{+}\hat{\mathbf{y}}\hat{\cdot}\hat{\mathbf{y}}=\hat{\mathbf{z}}\hat{\cdot}\hat{\mathbf{z}}.
\end{aligned}
\]

\noindent This interpretation scheme works the same for the predicate
$<$ in $\mathscr{E}_{d}$.

\noindent The interpreted formula can often be simplified for special
cases. The following are two types of \emph{wff}s that are most frequently
encountered in $\mathscr{E}_{d}$:

(1) $d(\mathbf{pq})=d(\mathbf{uv})$ can be interpreted as $D(\mathbf{pquv})$
in $\mathscr{E}_{2}$.

(2) $d(\mathbf{pq})=d(\mathbf{uv})+d(\mathbf{rs)}$ can be interpreted
as
\[
\exists\mathbf{a}\exists\mathbf{b}\exists\mathbf{c}\,\,\,B(\mathbf{abc})\wedge D(\mathbf{pqac})\wedge D(\mathbf{uvab})\wedge D(\mathbf{rsbc})
\]
in $\mathscr{E}_{2}$.

\section{Completeness and Decidability of $\mathscr{E}_{d}$}

In the previous section, we have established that $\mathscr{E}_{d}$
is mutually interpretable (with parameters) to $\mathscr{E}_{2}$
and the interpretation if faithful. Because it has been shown that
$\mathscr{E}_{2}$ is complete and decidable, $\mathscr{E}_{d}$ is
also complete and decidable.

\section{Theory $\mathscr{H}_{da}$---Quantitative Hyperbolic Geometry}

We briefly provide a sketch of a theory $\mathscr{H}_{da}$ of quantitative
hyperbolic geometry. Its language $L(\mathscr{H}_{da})$ is the same
as $L(\mathscr{E}_{da})$ with two sorts.

We only need to replace Axiom D5 of $\mathscr{E}_{da}$ with Axiom
D5-h, and retain the rest of the axioms in $\mathscr{E}_{da}$, as
listed below. We shall not go into the details in this direction.

\noindent\rule[0.5ex]{1\columnwidth}{1pt}

\noindent\textbf{\emph{$\mathscr{H}_{da}$}}\textbf{: }\textbf{\emph{The
Theory of Plane Quantitative Hyperbolic Geometry with Distance and
Angle Functions in the language of $L(d,a;\,=,+,\cdot,<,0,1)$}}

\noindent\rule[0.5ex]{1\columnwidth}{1pt}\medskip{}

\noindent\textbf{Axioms RCF1 through RCF17}\medskip{}

\noindent\textbf{Axioms D1} \textbf{through D4}\smallskip{}

\noindent\textbf{Axiom D5-h}. (AAA$\rightarrow$SSS) Given two triangles,
if all three pairs of angles are equal, then their sides are congruent.
\begin{align*}
\begin{aligned}\end{aligned}
 & \forall\mathbf{a}\forall\mathbf{b}\forall\mathbf{c}\forall\mathbf{a'}\forall\mathbf{b'}\forall\mathbf{c'}\\
 & \forall\mathbf{a}\forall\mathbf{b}\forall\mathbf{c}\forall\mathbf{a'}\forall\mathbf{b'}\forall\mathbf{c'}\,\,\,a(\mathbf{abc})=\,a(\mathbf{a'b'c'})\wedge\,a(\mathbf{bca})=\,a(\mathbf{b'c'a'})\wedge\,a(\mathbf{cab})=\,a(\mathbf{c'a'b'})\\
 & \rightarrow d(\mathbf{ab})=d(\mathbf{a'b'})\wedge d(\mathbf{bc})=d(\mathbf{b'c'})\wedge d(\mathbf{ca})=d(\mathbf{c'a'}).
\end{align*}

\noindent\textbf{Axioms D6, D7}\smallskip{}

\noindent\textbf{Axioms A1} \textbf{through A4}

\section{Theories $\mathscr{E}_{d}^{n}$ and $\mathscr{E}_{da}^{n}$---Quantitative
Euclidean Geometry in Higher Dimensions\label{sec:Geometry-of-Higher}}

\textbf{}Both systems $\mathscr{E}_{d}$ and $\mathscr{E}_{da}$
can be generalized to $\mathscr{E}_{d}^{n}$ and $\mathscr{E}_{da}^{n}$
for $n$-dimensional Euclidean geometry ($n\geqslant2$). We only
need to replace the two dimension axioms D6 and D7 with D6-n and
D7-n, as described below.\medskip{}

\noindent\textbf{Axiom D6-n}. Dimension lower bound (at least $n$)
\begin{align*}
\begin{aligned}\end{aligned}
 & \exists\mathbf{p}_{0}\exists\mathbf{p}_{1}\exists\cdots\exists\mathbf{p}_{n}\\
 & d(\mathbf{p}_{0}\mathbf{p}_{1})\neq0\wedge d(\mathbf{p}_{0}\mathbf{p}_{1})=d\mathbf{p}_{0}\mathbf{p}_{2})\wedge\cdots\wedge d(\mathbf{p}_{0}\mathbf{p}_{1})=d(\mathbf{p}_{i}\mathbf{p}_{j})\wedge\cdots\wedge d(\mathbf{p}_{0}\mathbf{p}_{1})=d(\mathbf{p}_{n-1}\mathbf{p}_{n}),
\end{align*}

\noindent where the conjunction ranges over all $0\leqslant i<j\leqslant n$.\smallskip{}

\noindent\textbf{Axiom D7-n}. Dimension upper bound (at most $n$)
\begin{align*}
\begin{aligned}\end{aligned}
 & \forall\mathbf{p}_{0}\forall\mathbf{p}_{1}\forall\cdots\forall\mathbf{p}_{n+1}\\
 & d(\mathbf{p}_{0}\mathbf{p}_{1})=d\mathbf{p}_{0}\mathbf{p}_{2})\wedge\cdots\wedge d(\mathbf{p}_{0}\mathbf{p}_{1})=d(\mathbf{p}_{i}\mathbf{p}_{j})\wedge\cdots\wedge d(\mathbf{p}_{0}\mathbf{p}_{1})=d(\mathbf{p}_{n}\mathbf{p}_{n+1})\\
 & \rightarrow d(\mathbf{p}_{0}\mathbf{p}_{1})=0,
\end{align*}

\noindent where the conjunction ranges over all $0\leqslant i<j\leqslant n+1$. 

We will not discuss further in this direction.

\vspace{6pt}
\funding{This research received no external funding.}
\dataavailability{No new data were created or analyzed in this study. Data sharing is not applicable to this article.}

\acknowledgments{The author thanks Michael Beeson, Julien Narboux, and Pierre Boutry for valuable discussions, and the anonymous reviewers for their careful reading of the paper and their insightful comments.}

\conflictsofinterest{The author declares no conflicts of interest.}

\appendixtitles{yes} 
\appendixstart
\appendix
\section[\appendixname~\thesection]{Axioms of Real Closed Fields (RCF) \cite{Epstein}\label{appendix:a}}
\medskip{}

\noindent \textbf{\emph{The Theory of Real Closed Fields in the Language
of $L(=,+,\cdot,<,0,1)$}}

\medskip{}

Note that $0$ and $1$ are constant symbols, denoting the identity
of addition and the identity of multiplication respectively. The constant
$0$ is definable by the formula $x+x=x$, and the constant $1$ is
definable by $x\cdot x=x\wedge\neg(x+x=x)$. We can have an equivalent
theory in the language of $L(=,<,+,\cdot)$ after replacing the constants
0 and 1 with the above definitions, but we choose to include these
constants in the language for convenience.\medskip{}

\noindent RCF1. $\forall x\forall y\forall z\,\,x+(y+z)=(x+y)+z$

\noindent RCF2. $\forall x\,\,0+x=x$

\noindent RCF3. $\forall x\exists y\,\,x+y=0$

\noindent RCF4. $\forall x\forall y\,\,x+y=y+x$

\noindent RCF5. $\forall x\,\,x\cdot1=x$

\noindent RCF6. $\forall x\forall y\forall z\,\,x\cdot(y\cdot z)=(x\cdot y)\cdot z$

\noindent RCF7. $\forall x\forall y\,\,x\cdot y=y\cdot x$

\noindent RCF8. $\forall x\exists y\,[x\ne0\rightarrow x\cdot y=1]$

\noindent RCF9. $\forall x\forall y\forall z\,\,x\cdot(y+z)=(x\cdot y)+(x\cdot z)$

\noindent RCF10. $1\ne0$

\noindent RCF11. $\forall x\,\,\neg(x<x)$

\noindent RCF12. $\forall x\forall y\,[x<y\rightarrow\neg(y<x)]$

\noindent RCF13. $\forall x\forall y\forall z\,[(x<y\wedge y<z)\rightarrow x<z]$

\noindent RCF14. $\forall x\forall y\,\,x<y\vee y<x\vee x=y$

\noindent RCF15. $\forall x\forall y\forall z\,[y<z\rightarrow x+y<x+z]$

\noindent RCF16. $\forall x\forall y\,[(0<x\wedge0<y)\rightarrow0<x\cdot y]$

\noindent RCF17. $[\forall x\forall y\,\,\,\Phi(x)\wedge\Psi(y)\rightarrow x<y]\rightarrow[\exists z\forall x\forall y\,\,\,\Phi(x)\wedge\Psi(y)\rightarrow x\leqslant z\leqslant y]$

\noindent where $\Phi$ is any \emph{wff} in which $x$ is free, $\Psi$
is any \emph{wff} in which $y$ is free, and neither $y$ nor $z$
appear in $\Phi$, and neither $x$ nor $z$ appear in $\Psi$. ($x\leqslant z$
is the abbreviation for $x<z\vee x=z$.)

\section[\appendixname~\thesection]{Tarski's Axioms of Plane Euclidean Geometry $\mathscr{E}_{2}$
\cite{Tarski-2,Schwabhauser}\label{appendix:b}}

The following are Tarski's axioms for $\mathscr{E}_{2}$ as listed
in \cite{Tarski-2,Schwabhauser}. Tarski included thirteen axioms,
A1 through A13, in his 1959 paper \cite{Tarski-1}, but A2 and A3
were proved as theorems later. The remaining axioms in \cite{Tarski-1}
are the same as listed below, only in different order.\medskip{}

\noindent \textbf{\emph{The Theory $\mathscr{E}_{2}$ in the Language of}
$L(=,B,D)$}

\medskip{}

$B(\mathbf{pvq})$ denotes that point $\mathbf{v}$ is between $\mathbf{p}$
and $\mathbf{q}$.

$D(\mathbf{pquv})$ denotes that point pair $\mathbf{pq}$ is congruent
to point pair $\mathbf{uv}$.

~

\noindent T1. Symmetry of congruence\smallskip{}

$\forall\mathbf{p}\forall\mathbf{q}\,D(\mathbf{pqqp})$\smallskip{}

\noindent T2. Transitivity of congruence\smallskip{}

$\forall\mathbf{p}\forall\mathbf{q}\forall\mathbf{s}\forall\mathbf{t}\forall\mathbf{u}\forall\mathbf{v}$

$D(\mathbf{pqst})\,\wedge\,D(\mathbf{pquv})\rightarrow D(\mathbf{stuv})$\smallskip{}

\noindent T3. Identity of congruence\smallskip{}

$\forall\mathbf{p}\forall\mathbf{q}\forall\mathbf{t}$

$D(\mathbf{pqtt})\rightarrow\mathbf{p}=\mathbf{q}$\smallskip{}

\noindent T4. Segment extension\smallskip{}

$\forall\mathbf{p}\forall\mathbf{q}\forall\mathbf{u}\forall\mathbf{v}\exists\mathbf{a}$

$B(\mathbf{pqa})\wedge D(\mathbf{qauv})$\smallskip{}

\noindent T5. Five segments\smallskip{}

$\forall\mathbf{a}\forall\mathbf{b}\forall\mathbf{c}\forall\mathbf{d}\forall\mathbf{\mathbf{a}'}\forall\mathbf{\mathbf{b}'}\forall\mathbf{c}'\forall\mathbf{\mathbf{d}'}$

$\mathbf{a}\neq\mathbf{b}\wedge B(\mathbf{abd})\wedge B(\mathbf{\mathbf{a}'\mathbf{b}'d'})\wedge D(\mathbf{aca'c'})\wedge D(\mathbf{aba'b'})\wedge D(\mathbf{bcb'c'})\wedge D(\mathbf{bdb'd'})$

$\rightarrow D(\mathbf{cdc'd'})$
\vspace{-12pt}
\begin{center}
\includegraphics[scale=0.2]{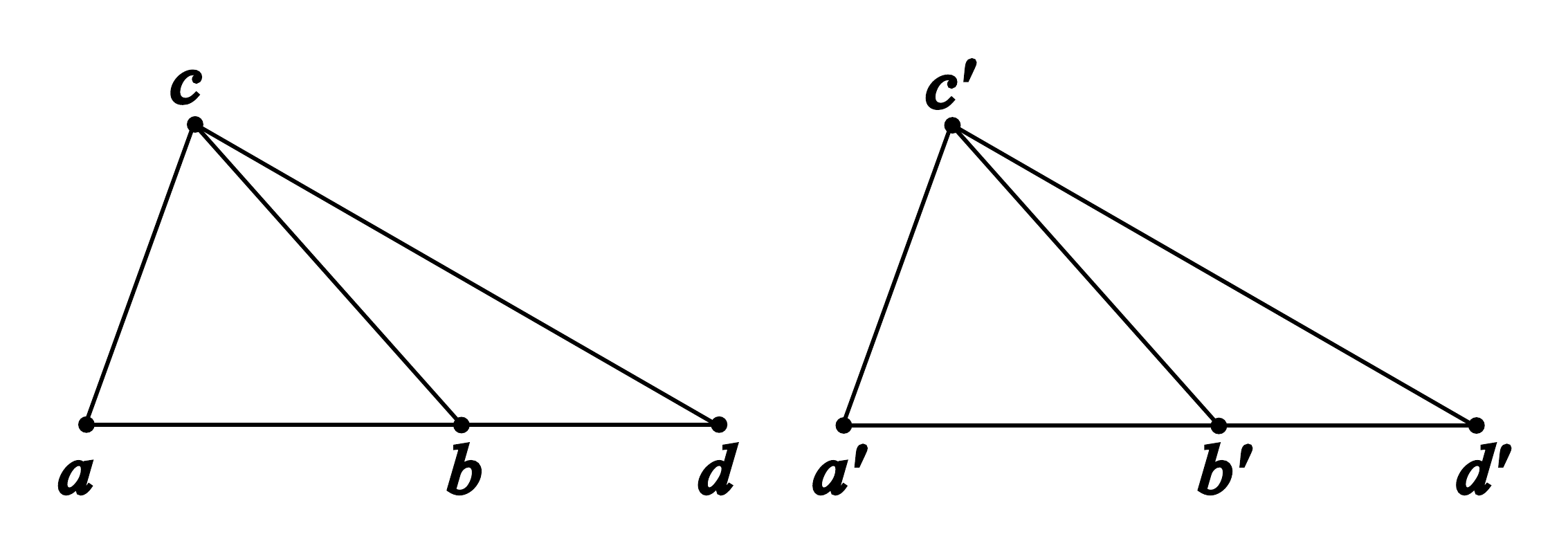}
\par\end{center}

\vspace{-18pt}
\begin{center}
\captionsetup{width=0.3\textwidth}
\captionof{figure}{Five segments.}
\end{center}
\vspace{-5pt}

\noindent T6. Identity of betweenness\smallskip{}

$\forall\mathbf{p}\forall\mathbf{t}\,B(\mathbf{ptp})\rightarrow\mathbf{p}=\mathbf{t}$

\smallskip{}

\noindent T7. Pasch (inner form)\smallskip{}

$\forall\mathbf{a}\forall\mathbf{b}\forall\mathbf{c}\forall\mathbf{p}\forall\mathbf{q}$

$B(\mathbf{bcp})\wedge B(\mathbf{aqb})]\rightarrow\exists\mathbf{t}\,\,[B(\mathbf{ptq})\wedge B(\mathbf{atc})]$
\begin{center}
\includegraphics[scale=0.2]{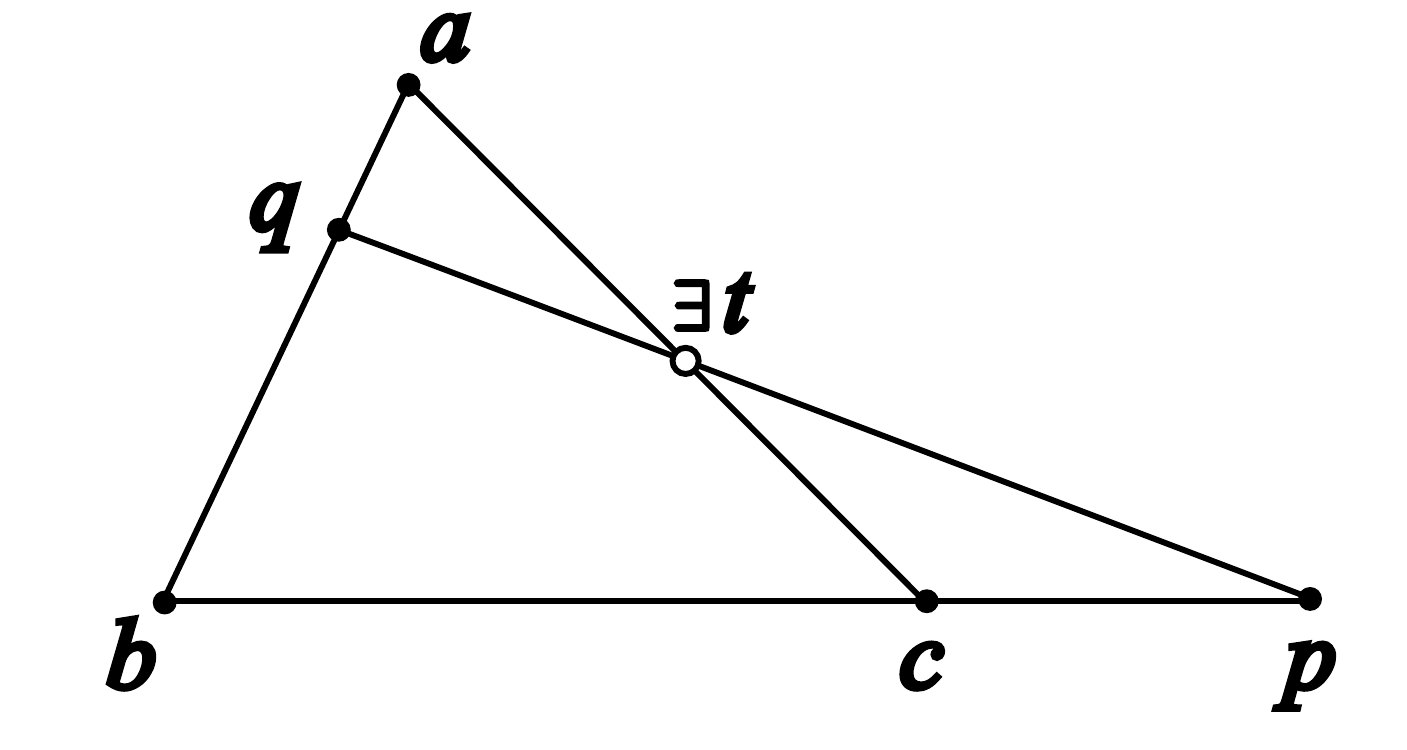}
\par\end{center}

\vspace{-12pt}
\begin{center}
\captionsetup{width=0.25\textwidth}
\captionof{figure}{Inner Pasch.}
\end{center}
\vspace{-5pt}

\noindent T8. Dimension lower bound\smallskip{}

$\exists\mathbf{a}\exists\mathbf{b}\exists\mathbf{c}\,\neg B\mathbf{(abc})\wedge\neg B(\mathbf{bca})\wedge\neg B(\mathbf{cab})$\smallskip{}

\noindent T9. Dimension upper bound\smallskip{}

$\forall\mathbf{p}\forall\mathbf{q}\forall\mathbf{t}\forall\mathbf{u}\forall\mathbf{v}$

$\mathbf{p}\ne\mathbf{q}\wedge D(\mathbf{tptq})\wedge D(\mathbf{upuq})\wedge D(\mathbf{vpvq})\rightarrow B(\mathbf{tuv})\vee B(\mathbf{uvt})\vee B(\mathbf{vtu})$\medskip{}

\begin{center}
\includegraphics[scale=0.2]{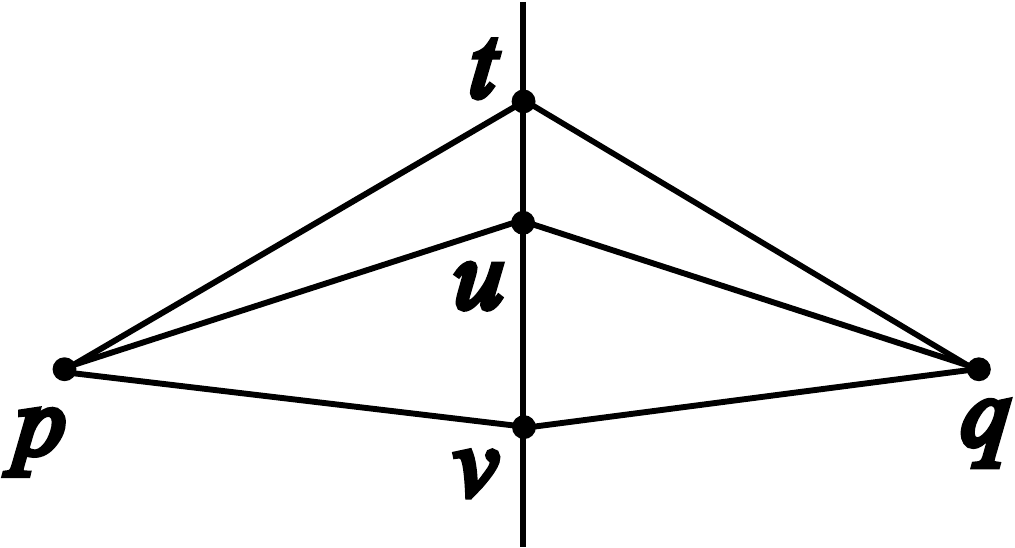}
\par\end{center}

\vspace{-12pt}
\begin{center}
\captionsetup{width=0.39\textwidth}
\captionof{figure}{The dimension is at most 2.}
\end{center}
\vspace{-5pt}

\noindent T10. Euclid\smallskip{}

$\forall\mathbf{a}\forall\mathbf{b}\forall\mathbf{c}\forall\mathbf{u}\forall\mathbf{v}$

$\exists\mathbf{p}\exists\mathbf{q}\,\mathbf{a}\neq\mathbf{u}\,B(\mathbf{buc})\wedge B(\mathbf{auv})\rightarrow B(\mathbf{pvq})\wedge B(\mathbf{abp})\wedge B(\mathbf{acq})$\bigskip{}
\vspace{-4pt}
\begin{center}
\includegraphics[scale=0.2]{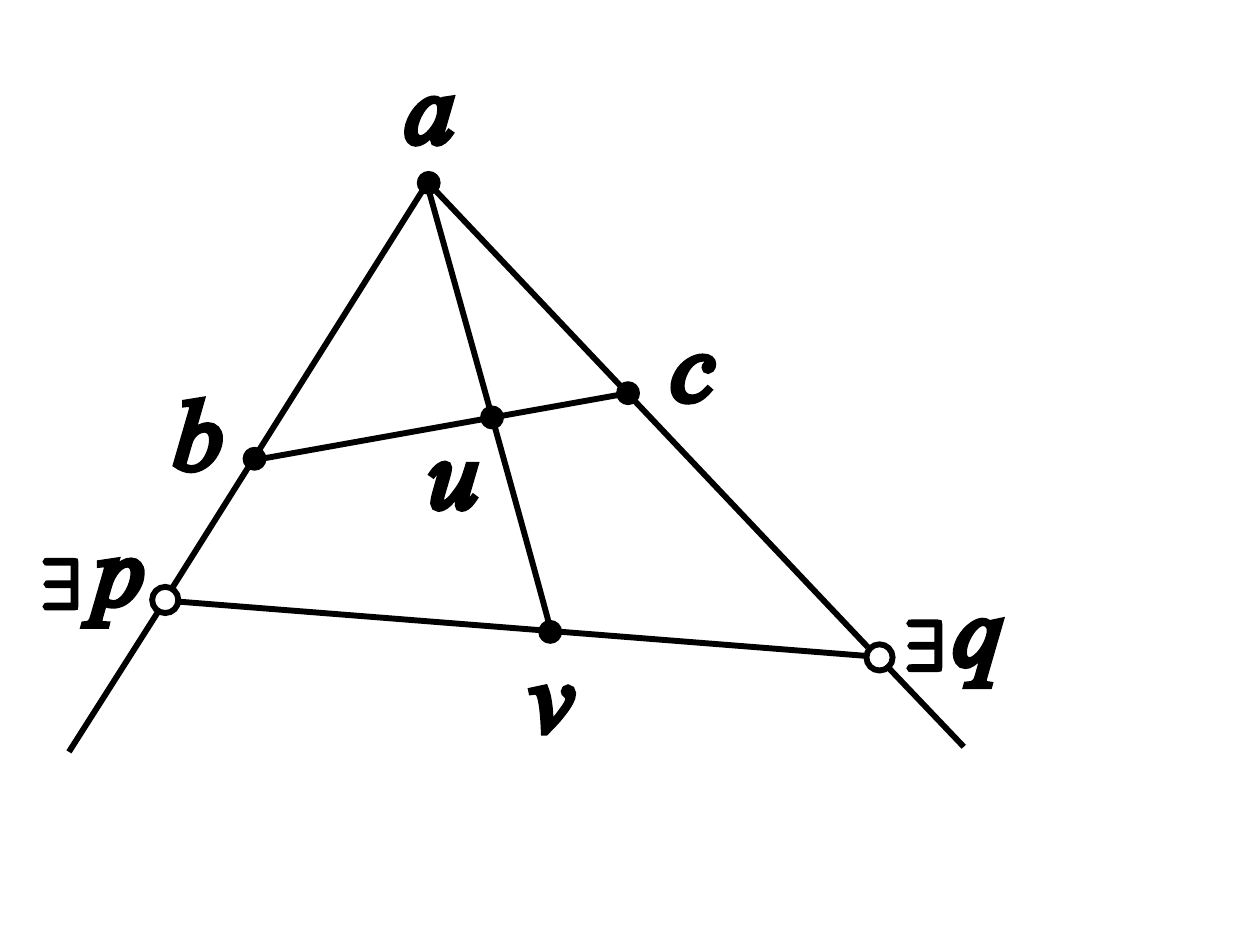}
\par\end{center}

\vspace{-18pt}
\begin{center}
\captionsetup{width=0.2\textwidth}
\captionof{figure}{Euclid.}
\end{center}
\vspace{-5pt}

\noindent T11. Axiom schema of continuity (Dedekind cut)

$\left[\exists\mathbf{a}\forall\mathbf{p}\forall\mathbf{q}\,\,\,\Phi(\mathbf{p})\wedge\Psi(\mathbf{q})\rightarrow B(\mathbf{apq})\right]\rightarrow\left[\exists\mathbf{b}\forall\mathbf{p}\forall\mathbf{q}\,\,\,\Phi(\mathbf{p})\wedge\Psi(\mathbf{q})\rightarrow B(\mathbf{pbq})\right]$,\smallskip{}

\noindent where $\Phi$ stands for any formula in which the variable
$\mathbf{p}$ occurs free but neither $\mathbf{q}$ nor $\mathbf{a}$
nor $\mathbf{b}$ occurs free, and similarly for $\Psi$ with $\mathbf{p}$
and $\mathbf{q}$ interchanged.
\begin{center}
\includegraphics[scale=0.25]{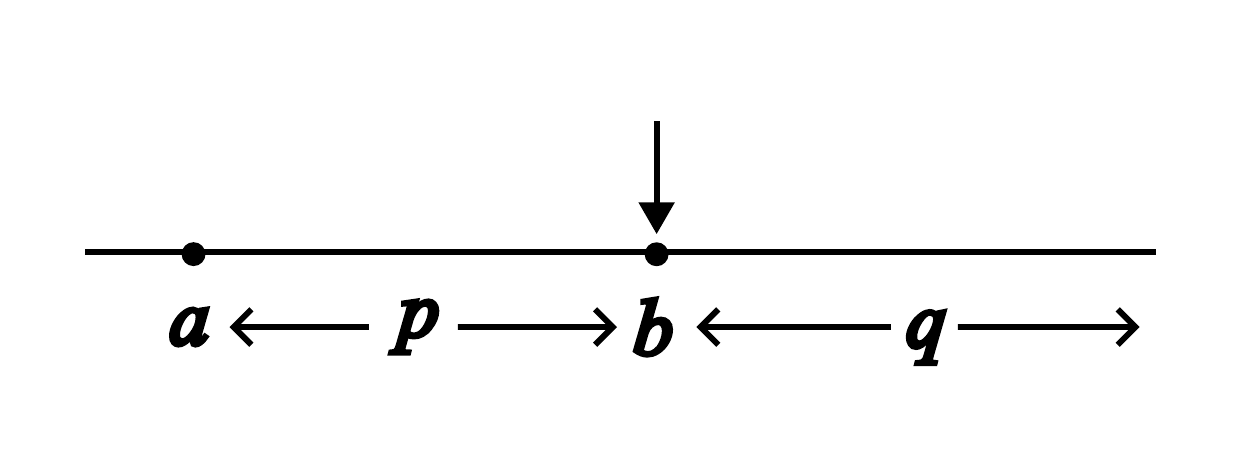}
\par\end{center}

\vspace{-18pt}
\begin{center}
\captionsetup{width=0.22\textwidth}
\captionof{figure}{Continuity.}
\end{center}
\vspace{-5pt}

\section[\appendixname~\thesection]{SMGS Axioms of Euclidean Geometry \cite{SMSG} \label{appendix:c}}

\medskip{}

\noindent Undefined terms: point, line, plane.

The officially listed undefined terms are: point, line, plane. However,
there are also undefined terms (objects, predicates, and functions)
that are used in the axioms but are not explicitly listed, such as:
distance measure, angle measure, area, volume, lie on, lie in, real
number, set, contain, correspondence, intersection, and union.

~

\noindent Postulate 1. (Line Uniqueness) Given any two different points,
there is exactly one line which contains them.\medskip{}

\noindent Postulate 2. (Distance Postulate) To every pair of different
points there corresponds a unique positive number. \medskip{}

\noindent Postulate 3. (Ruler Postulate) The points of a line can
be placed in a correspondence with the real numbers in such a way
that

\begin{tabular}{r>{\raggedright}p{0.8\columnwidth}}
(1) & To every point of the line there corresponds exactly one real number,\tabularnewline
(2) & To every real number there corresponds exactly one point of the line,
and\tabularnewline
(3) & The distance between two points is the absolute value of the difference
of the corresponding numbers.\tabularnewline
\end{tabular}

\medskip{}

\noindent Postulate 4. (Ruler Placement Postulate) Given two points
$P$ and $Q$ of a line, the coordinate system can be chosen in such
a way that the coordinate of $P$ is zero and the coordinate of $Q$
is positive.\medskip{}

\noindent Postulate 5. (Existence of Points)

\begin{tabular}{r>{\raggedright}p{0.8\columnwidth}}
a. & \noindent Every plane contains at least three non-collinear points.\tabularnewline
b. & \noindent Space contains at least four non-coplanar points.\tabularnewline
\end{tabular}

\medskip{}

\noindent Postulate 6. (Points on a Line Lie in a Plane) If two points
lie in a plane, then the line containing these points lies in the
same plane.\medskip{}

\noindent Postulate 7. (Plane Uniqueness) Any three points lie in
at least one plane, and any three non-collinear points lie in exactly
one plane.\medskip{}

\noindent Postulate 8. (Plane Intersection) If two different planes
intersect, then their intersection is a line.\medskip{}

\noindent Postulate 9. (Plane Separation Postulate) Given a line and
a plane containing it, the points of the plane that do not lie on
the line form two sets such that:

\begin{tabular}{r>{\raggedright}p{0.8\columnwidth}}
(1) & each of the sets is convex and\tabularnewline
(2) & if $P$ is in one set and $Q$ is in the other then segment $\overline{PQ}$
intersects the line.\tabularnewline
\end{tabular}\medskip{}

\noindent Postulate 10. (Space Separation Postulate) The points of
space that do not lie in a given plane form two sets such that

\begin{tabular}{r>{\raggedright}p{0.8\columnwidth}}
(1) & each of the sets is convex and\tabularnewline
(2) & if $P$ is in one set and $Q$ is in the other then segment $\overline{PQ}$
intersects the plane.\tabularnewline
\end{tabular}

\medskip{}

\noindent Postulate 11. (Angle Measurement Postulate) To every angle
$\angle BAC$ there corresponds a real number between $0$ and $180$.\medskip{}

\noindent Postulate 12. (Angle Construction Postulate) Let $\overrightarrow{AB}$
be a ray on the edge of the half-plane $H$. For every $r$ between
0 and 180 there is exactly one ray $\overrightarrow{AP}$, with $P$
in $H$, such that $m\angle PAB=r$.\medskip{}

\noindent Postulate 13. (Angle Addition Postulate) If $D$ is a point
in the interior of $\angle BAC$, then \\
$m\angle BAC=m\angle BAD+m\angle DAC$.\medskip{}

\noindent Postulate 14. (Supplement Postulate) If two angles form
a linear pair, then they are supplementary.\medskip{}

\noindent Postulate 15. (SAS Postulate) Given a correspondence between
two triangles (or between a triangle and itself). If two sides and
the included angle of the first triangle are congruent to the corresponding
parts of the second triangle, then the correspondence is a congruence.\medskip{}

\noindent Postulate 16. (Parallel Postulate) Through a given external
point there is at most one line parallel to a given line.\medskip{}

\noindent Postulate 17. (Area of Polygonal Region) To every polygonal
region there corresponds a unique positive number.\medskip{}

\noindent Postulate 18. (Area of Congruent Triangles) If two triangles
are congruent, then the triangular regions have the same area.\medskip{}

\noindent Postulate 19. (Summation of Areas of Regions) Suppose that
the region $R$ is the union of two regions $R_{1}$ and $R_{2}$.
Suppose that $R_{1}$ and$R_{2}$ intersect at most in a finite number
of segments and points. Then the area of $R$ is the sum of the areas
of $R_{1}$ and $R_{2}$.\medskip{}

\noindent Postulate 20. (Area of a Rectangle) The area of a rectangle
is the product of the length of its base and the length of its altitude.\medskip{}

\noindent Postulate 21. (Volume of Rectangular Parallelpiped) The
volume of a rectangular parallelpiped is equal to the product of the
altitude and the area of the base.\medskip{}

\noindent Postulate 22. (Cavalieri's Principle) Given two solids and
a plane. If for every plane which intersects the solids and is parallel
to the given plane the two intersections have equal areas, then the
two solids have the same volume.

\section[\appendixname~\thesection]{The Convoluted Statement of the Pythagorean Theorem in Tarski's $\mathscr{E}_{2}$ \label{appendix:d}}

In Section \ref{sec:Section-1}, we noted that SST \cite{Schwabhauser}
presents the Pythagorean theorem (Theorem 15.8) in the form
\begin{equation}
\overline{\mathbf{ab}}^{2}+\overline{\mathbf{ac}\vphantom{b}}^{2}=\overline{\mathbf{bc}}^{2}.\label{eq:Pythagorean-1-1}
\end{equation}
At first glance, this may seem to be the familiar modern algebraic
form of the Pythagorean theorem; however, it is not, because the language
of $\mathscr{E}_{2}$ cannot express the length of a segment as
a number. The terms $\overline{\mathbf{ab}}^{2}$, $\overline{\mathbf{ac}\vphantom{b}}^{2}$,
and $\overline{\mathbf{bc}}^{2}$ are all abbreviations. The term
$\overline{\mathbf{ab}}^{2}$ means the \emph{geometric product} of
$\overline{\mathbf{ab}}$ and $\overline{\mathbf{ab}}$, denoted by $\overline{\mathbf{ab}}\,\hat{\cdot}\,\overline{\mathbf{ab}}$.
Precisely, Equation (\ref{eq:Pythagorean-1-1}) should be expanded
as
\[
\overline{\mathbf{ab}}\,\hat{\cdot}\,\overline{\mathbf{ab}}\,\hat{+}\,\overline{\mathbf{ac}\vphantom{b}}\,\hat{\cdot}\,\overline{\mathbf{ac}\vphantom{b}}=\overline{\mathbf{bc}}\,\hat{\cdot}\,\overline{\mathbf{bc}},
\]
where $\hat{+}$ and $\hat{\cdot}$ denote\emph{ geometric sum} and
\emph{geometric product} respectively. These are defined by parallel
projections using the segment arithmetic (see the definitions in Section
\ref{sec:Tarski's-Interpretable}).

In what follows, we present the complete sentence corresponding to
Equation (\ref{eq:Pythagorean-1-1}) after these abbreviations are
fully expanded.

Let $\triangle\mathbf{abc}$ be a triangle with $\angle\mathbf{bac}$
being a right angle (Figure \ref{fig:Hilbert-Schwabhauser}).

\noindent (1) \hspace{6pt}Choose an arbitrary point $\mathbf{e}_{1}$
on $\mathbf{ba}$. We will take $\mathbf{be}_{1}$ as the unit length.

On line $\mathbf{bc}$, locate a point $\mathbf{a'}$ such that $D(\mathbf{ba'ba})$.
Join $\mathbf{e}_{1}$ and $\mathbf{a'}$ with a line.

Through $\mathbf{a}$, draw a line parallel to $\mathbf{a'e}_{1}$,
and let it intersect $\mathbf{bc}$ at $\mathbf{s}_{1}$.

By the definition of geometric product, $\mathbf{bs_{1}=\overline{\mathbf{ab}}^{2}}$.

\noindent (2) \hspace{6pt}Join $\mathbf{e}_{1}$ and $\mathbf{c}$
with a line.

On line $\mathbf{ba}$, locate a point $\mathbf{c'}$ such that $D(\mathbf{\mathbf{bc'bc}})$.

Through $\mathbf{c'}$, draw a line parallel to $\mathbf{ce_{1}}$,
and let it intersect $\mathbf{bc}$ at $\mathbf{s}_{3}$.

By definition, $\mathbf{bs}_{3}=\overline{\mathbf{bc}}^{2}$.

\noindent (3) \hspace{6pt}On $\mathbf{ca}$, find a point $\mathbf{e}_{2}$
such that $D(\mathbf{ce}_{2}\mathbf{be}_{1})$. Hence, $\mathbf{ce}_{2}$
also has unit length.

On line $\mathbf{bc}$, locate a point $\mathbf{a''}$ such that $D(\mathbf{ca''ca})$.

Join $\mathbf{a''}$ and $\mathbf{e}_{2}$ with a line.

Through $\mathbf{a}$, draw a line parallel to $\mathbf{a''e}_{2}$,
and let it intersect $\mathbf{bc}$ at $\mathbf{s}_{2}$.

By definition, $\mathbf{cs}_{2}=\overline{\mathbf{ac}\vphantom{b}}^{2}$.

Finally, Equation (\ref{eq:Pythagorean-1-1}) can be expressed as
$\mathbf{bs}_{1}\hat{+}\mathbf{cs}_{2}=\mathbf{bs}_{3}$. To prove
this theorem, it suffices to show $\mathbf{s}_{1}\mathbf{c}$ is congruent
to $\mathbf{s}_{3}\mathbf{s}_{2}$. 

\begin{adjustwidth}{-\extralength}{0cm}
\begin{center}
\begin{minipage}{1.28\textwidth}
\includegraphics[scale=0.18]{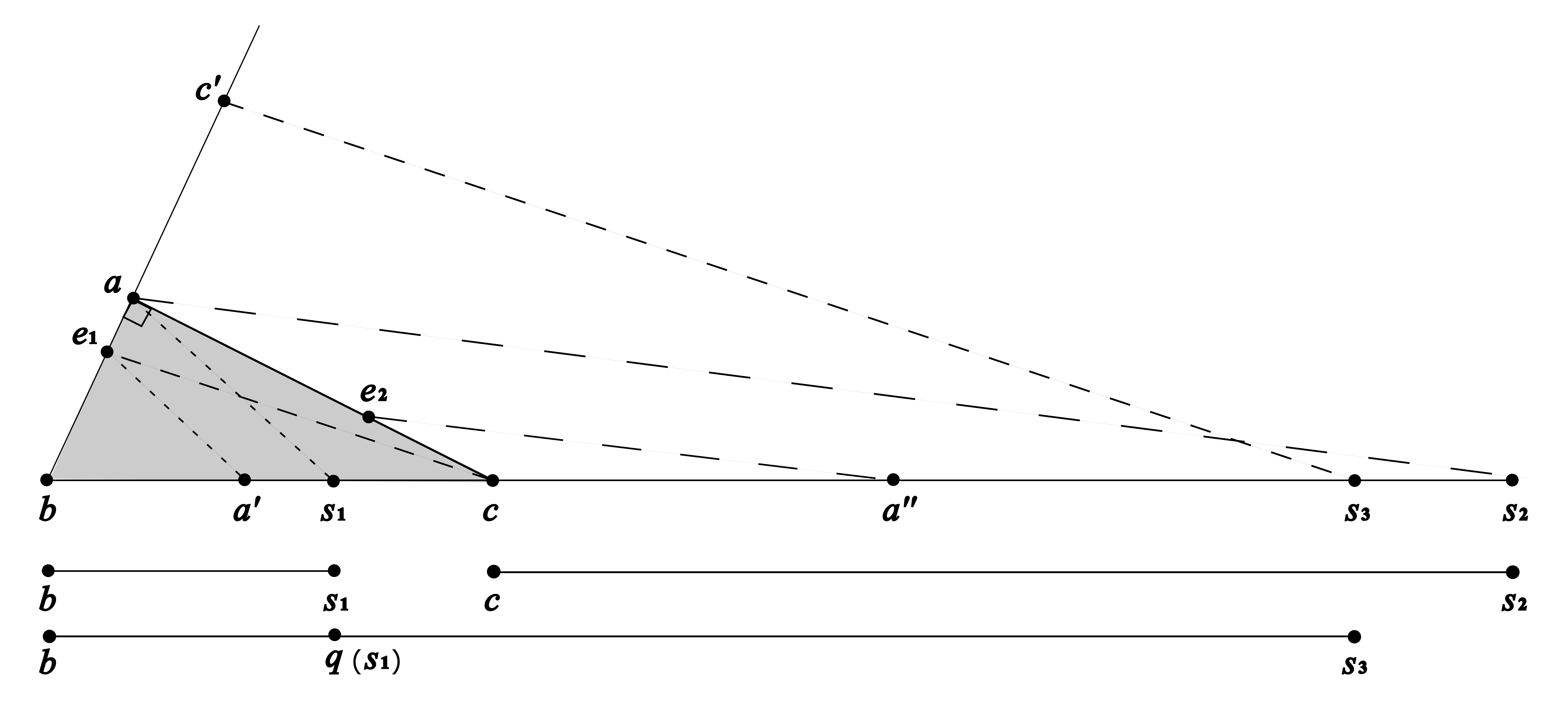}
\end{minipage}
\end{center}
\end{adjustwidth}
\hspace{-112pt}
\begin{minipage}{1.25\textwidth}
\captionof{figure}{The convoluted statement of the Pythagorean theorem using segment arithmetic: Segment $\mathbf{bs_3}$ is equal to the sum of $\mathbf{bs_1}$ and $\mathbf{cs_2}$.}
\label{fig:Hilbert-Schwabhauser}
\end{minipage}
\vspace{10pt}

The preceding description is informal. When the Pythagorean theorem
in the form of Equation (\ref{eq:Pythagorean-1-1}) is expressed in
the formal language of $L(\mathscr{E}_{2})$, it becomes:
\begin{align}
 & \forall\mathbf{a}\forall\mathbf{b}\forall\mathbf{c}\forall\mathbf{a'}\forall\mathbf{a''}\forall\mathbf{c'}\forall\mathbf{e}_{1}\forall\mathbf{e}_{2}\forall\mathbf{s}_{1}\forall\mathbf{s}_{2}\forall\mathbf{s}_{3}\nonumber \\
 & R(\mathbf{bac})\wedge[B(\mathbf{be}_{1}\mathbf{a})\vee B(\mathbf{bae}_{1})]\wedge[B(\mathbf{ce}_{2}\mathbf{a})\vee B(\mathbf{cae}_{2})]\nonumber \\
 & \wedge D(\mathbf{be}_{1}\mathbf{ce}_{2})\wedge D(\mathbf{ba'ba})\wedge D(\mathbf{\mathbf{bc'bc}})\wedge D(\mathbf{ca''ca})\nonumber \\
 & \wedge L(\mathbf{s}_{1}\mathbf{bc})\wedge L(\mathbf{s}_{2}\mathbf{bc})\wedge L(\mathbf{s}_{3}\mathbf{bc})\label{eq:Long-Pythagorean}\\
 & \wedge[\mathbf{as}_{1}\parallel\mathbf{a'e}_{1}]\wedge[\mathbf{as}_{2}\parallel\mathbf{a''e}_{2}]\wedge[\mathbf{c's}_{3}\parallel\mathbf{ce}_{1}]\nonumber \\
 & \rightarrow\exists\mathbf{q}\,\,\,\,B(\mathbf{bqs}_{3})\wedge D(\mathbf{bqbs}_{1})\wedge D(\mathbf{qs}_{3}\mathbf{cs}_{2}),\nonumber 
\end{align}
where
\begin{itemize}
\item $R(\mathbf{bac})$ is the abbreviation that $\mathbf{\angle bac}$
is a right angle (see Definition \ref{def:Perpendicular} in Section
\ref{sec:Theory-Ed});
\item $\mathbf{pq}\parallel\mathbf{uv}$ (denoting that $\mathbf{pq}$ is
parallel to $\mathbf{uv}$) is the abbreviation of
\[
\mathbf{pq}\parallel\mathbf{uv}:=\,\,\,\,\,\,\mathbf{p}\neq\mathbf{q}\wedge\mathbf{u}\neq\mathbf{v}\wedge\neg\exists\mathbf{t}\,[L(\mathbf{tpq})\wedge L(\mathbf{tuv})],
\]
where $L(\mathbf{puv})$ (denoting that $\mathbf{p},\mathbf{u},\mathbf{v}$
are collinear) is the abbreviation of (recall \mbox{Definition \ref{def:Collinear}}
in Section \ref{sec:Theory-Ed}):
\[
L(\mathbf{puv}):=\,\,\,\,\,\,B(\mathbf{puv})\vee B(\mathbf{uvp})\vee B(\mathbf{upv}).
\]
\end{itemize}
Equation (\ref{eq:Long-Pythagorean}) is a convoluted formula. In
this form, the Pythagorean theorem is not the same theorem as Proposition
I.47 in Euclid's \emph{Elements}, and neither theorem implies the
other without establishing a nontrivial connection between area and
segment arithmetic (see Hartshorne \citep{Hartshorne} (p.~179)).

The notion of geometric product employed here is a variant of Hilbert\textquoteright s
definition, using a skewed frame rather than an orthogonal frame. Although
it is not exactly the definition given by SST \cite{Schwabhauser},
it is equivalent to it. Had we adopted their definition verbatim (see
Figure \ref{def:Product-Schwabhauser} in Section \ref{sec:Tarski's-Interpretable}),
the formal statement of the Pythagorean theorem would have been even
longer.

In contrast, the language of $\mathscr{E}_{d}$ permits the Pythagorean
theorem to be expressed in the simple form
\[
d^{2}(\mathbf{ab})+d^{2}(\mathbf{ac})=d^{2}(\mathbf{bc}).
\]
This provides part of the motivation for the present paper.

\begin{adjustwidth}{-\extralength}{0cm}
\reftitle{References}





\PublishersNote{}
\end{adjustwidth}
\end{document}